\documentclass[reqno]{amsart}

\usepackage[lite]{amsrefs}

\usepackage{amssymb}

\usepackage[all,cmtip]{xy}

\PassOptionsToPackage{usenames,dvipsnames}{xcolor}
\usepackage[inline]{enumitem}
\usepackage{mathtools}
\usepackage{xfrac}
\usepackage{bm}
\usepackage{tkz-euclide}
\usepackage{tikz,xifthen}
\usepackage{tikz-cd}
\usepackage{mathrsfs}
\usetikzlibrary{arrows}
\usetikzlibrary{decorations.markings}
\usetikzlibrary{backgrounds}
\usepackage{pgfplots}
\pgfplotsset{compat=1.10}
\usepgfplotslibrary{fillbetween}
\usetikzlibrary{patterns}
\usepackage[T1]{fontenc}
\usepackage{tensor}
\usepackage{lmodern}
\usepackage{cancel}
\usepackage[linktocpage]{hyperref}
\usepackage{adjustbox}
\usepackage{graphicx}
\usepackage{yhmath}
\usepackage{etoolbox}
\usepackage{needspace}
\usepackage{listofitems}

\usepackage[dvipsnames]{xcolor}
\setlist{leftmargin=9mm}
\usepackage[textsize=tiny,linecolor=MidnightBlue,backgroundcolor=MidnightBlue!25,textwidth=2.7cm]{todonotes}
\makeatletter 
\@mparswitchfalse%
\makeatother
\normalmarginpar

\hypersetup{
    colorlinks=true,
    linkcolor=MidnightBlue!90!black,
    citecolor=black!25!red,
    urlcolor=MidnightBlue!90!black,
    pdftitle={Covariant derivatives of representation-valued forms},
    pdfpagemode=FullScreen,
}

\makeatletter
\renewcommand\subsection{
\Needspace{2cm}%prevent line break immediately after subsection heading
\@startsection{subsection}{2}%
  \z@{-.5\linespacing\@plus-.7\linespacing}{.5\linespacing}%
  {\bf}}
\makeatother

\usetikzlibrary{decorations.markings}
\tikzset{double line with arrow/.style args={#1,#2}{decorate,decoration={markings,%
mark=at position 0 with {\coordinate (ta-base-1) at (0,1pt);
\coordinate (ta-base-2) at (0,-1pt);},
mark=at position 1 with {\draw[#1] (ta-base-1) -- (0,1pt);
\draw[#2] (ta-base-2) -- (0,-1pt);
}}}}
\tikzset{Equals/.style={-,double line with arrow={-,-}}}
\makeatletter
\newsavebox{\@brx}
\newcommand{\llangle}[1][]{\savebox{\@brx}{\(\m@th{#1\langle}\)}%
  \mathopen{\copy\@brx\kern-0.5\wd\@brx\usebox{\@brx}}}
\newcommand{\rrangle}[1][]{\savebox{\@brx}{\(\m@th{#1\rangle}\)}%
  \mathclose{\copy\@brx\kern-0.5\wd\@brx\usebox{\@brx}}}
\makeatother

\makeatletter
\def\blfootnote{\xdef\@thefnmark{}\@footnotetext}
\makeatother
\setlist[enumerate]{leftmargin=25pt, label={(\roman*)}}

\DeclarePairedDelimiterX\set[1]\lbrace\rbrace{\def\given{\;\delimsize\vert\;}#1}
\renewcommand{\d}[1]{\ensuremath{\operatorname{d}\!{#1}}}
\newcommand{\D}[1]{\ensuremath{\operatorname{D}\!{#1}}}

\newcommand{\enquote}[1]{``{#1}''}
\newcommand{\rra}{\rightrightarrows}
\newcommand{\Ra}{\Rightarrow}
\newcommand{\ra}{\rightarrow}

\newcommand{\R}{\mathbb{R}}
\newcommand{\T}{\mathbb{T}}
\newcommand{\Z}{\mathbb{Z}}
\renewcommand{\AA}{\mathbb{A}}
\newcommand{\MM}{\mathbb{M}}

\newcommand{\C}{\mathcal{C}}
\newcommand{\G}{\mathcal{G}}

\newcommand{\A}{\mathscr{A}}
\newcommand{\F}{\mathcal{F}}
\newcommand{\ve}{\mathscr{V}\hspace{-0.25em}\mathscr{E}}%\operatorname{VE}
\newcommand{\wedgedot}{\:\dot\wedge\:}
\newcommand{\ul}[1]{\underline{\smash{#1}}}

\renewcommand{\O}{\mathcal{O}}

\renewcommand{\sec}{§}
\renewcommand{\L}{\mathscr{L}}
\renewcommand{\frak}{\mathfrak}

\newcommand{\End}{\operatorname{End}}
\newcommand{\Der}{\operatorname{Der}}
\newcommand{\Hom}{\operatorname{Hom}}

\newcommand{\codim}{\operatorname{codim}}
\newcommand{\sgn}{\operatorname{sgn}}

\newcommand{\pr}{\mathrm{pr}}

\newcommand{\Ad}{\mathrm{Ad}}
\newcommand{\ad}{\operatorname{ad}}
\newcommand{\Hor}{\mathrm{Hor}}
\newcommand{\id}{\mathrm{id}}
\newcommand{\obs}{\mathrm{obs}}
\newcommand{\im}{\operatorname{im}}
\newcommand{\inv}{\mathrm{inv}}
\newcommand{\ext}{\mathrm{ext}}
\newcommand{\ev}{\mathrm{ev}}

\newcommand{\vf}{\mathfrak X}

\newcommand{\deriv}[2]{\frac{d}{d{#1}}\Big|_{{#1}={#2}}}
\newcommand{\smallderiv}[2]{\left.\tfrac{d}{d{#1}}\right|_{{#1}={#2}}}
\newcommand{\arc}{\mathscr{C}}%{\A}_{cur},%\wideparen

\newcommand{\comp}[1]{{#1}^{(2)}}
\newcommand{\iso}[1]{I({#1})^\circ}
\renewcommand{\|}{\,|\,}

\DeclareRobustCommand{\gobblefive}[5]{}
\newcommand*{\SkipTocEntry}{\addtocontents{toc}{\gobblefive}}

\newcommand{\cev}[1]{{#1}^L}
\numberwithin{equation}{section}

\theoremstyle{plain} %% This is the default, anyway
\newtheorem{theorem}{Theorem}
\numberwithin{theorem}{section}
\newtheorem{corollary}[theorem]{Corollary}
\newtheorem{lemma}[theorem]{Lemma}
\newtheorem{proposition}[theorem]{Proposition}
\newtheorem{customthminner}{Theorem}
\newenvironment{customthm}[1]{%
  \renewcommand\thecustomthminner{#1}%
  \customthminner
}{\endcustomthminner}

\theoremstyle{definition}
\newtheorem{definition}[theorem]{Definition}

\theoremstyle{remark}
\newtheorem{remark}[theorem]{Remark}
\newtheorem{intermezzo}[theorem]{Intermezzo}
\newtheorem*{remark*}{Remark}
\newtheorem{example}[theorem]{Example}

\AtBeginEnvironment{lemma}{\Needspace{1cm}}
\AtBeginEnvironment{theorem}{\Needspace{1cm}}
\AtBeginEnvironment{proposition}{\Needspace{1cm}}
\AtBeginEnvironment{definition}{\Needspace{1cm}}
\AtBeginEnvironment{corollary}{\Needspace{1cm}}

\calclayout

\newcommand{\Addresses}{{% additional braces for segregating \footnotesize
  \bigskip
  \footnotesize

  Ž.\ Grad, \textsc{Department of Mathematics, Celestijnenlaan 200B, Leuven, Belgium.}\par\nopagebreak
  \textit{E-mail address}: \url{zan.grad@protonmail.com}.\par\nopagebreak
  
  \textit{Webpage}: \url{https://zangrad.github.io}.
}}

\makeatletter

\renewcommand{\tocsection}[3]{%
  \indentlabel{\@ifnotempty{#2}{\bfseries\ignorespaces#1 #2\quad}}\bfseries#3}
\renewcommand{\tocsubsection}[3]{%
  \indentlabel{\@ifnotempty{#2}{\ignorespaces#1 #2\quad}}#3}
\@ifundefined{@dotsep}{\newcommand\@dotsep{4.5}}{\renewcommand\@dotsep{4.5}}
\def\@tocline#1#2#3#4#5#6#7{\relax
  \ifnum #1>\c@tocdepth % then omit
  \else
    \par \addpenalty\@secpenalty\addvspace{#2}%
    \begingroup \hyphenpenalty\@M
    \@ifempty{#4}{%
      \@tempdima\csname r@tocindent\number#1\endcsname\relax
    }{%
      \@tempdima#4\relax
    }%
    \parindent\z@ \leftskip#3\relax \advance\leftskip\@tempdima\relax
    \rightskip\@pnumwidth plus1em \parfillskip-\@pnumwidth
    #5\leavevmode\hskip-\@tempdima{#6}\nobreak
    \ifnum #1>1 \leaders\hbox{$\m@th\mkern \@dotsep mu\hbox{.}\mkern \@dotsep mu$} \fi
	\hfill
    \nobreak
    \hbox to\@pnumwidth{\@tocpagenum{\ifnum#1=1\bfseries\fi#7}}\par% <-- \bfseries for \section page
    \nobreak
    \endgroup
  \fi}
\AtBeginDocument{%
\expandafter\renewcommand\csname r@tocindent0\endcsname{0pt}
}

\def\l@subsection{\@tocline{2}{0pt}{1.7pc}{0pc}{}}
\def\l@section{\@tocline{1}{1em}{0pt}{}{\bfseries}}% <- added
\makeatother

\begin{document}

%%% In the title, use a double backslash "\\" to show a linebreak:
%%% Use one of the following two forms:
%%% \title{Text of the title}
%%% or
%%% \title[Short form for the running head]{Text of the title}
%\title[Covariant derivatives of representation-valued multiplicative forms]{Covariant derivatives of\\representation-valued multiplicative forms}
\title[Covariant derivatives of representation-valued forms]{Covariant derivatives in the representation-valued\\Bott--Shulman--Stasheff and Weil complex}
%\title[Covariant derivatives of multiplicative forms]{Covariant derivatives of multiplicative forms}

%%% If there are multiple authors, they're described one at a time:
%%% First author: \author{} \address{} \curraddr{} \email{} \thanks{}
%%% Second author: \author{} \address{} \curraddr{} \email{} \thanks{}
%%% Third author: \author{} \address{} \curraddr{} \email{} \thanks{}
\author[Ž.\ Grad]{Žan Grad}

%\date{asdsa}
%%% In the address, show linebreaks with double backslashes:
%\address{}

%%% Current address is optional.
%\curraddr{sad}

%%% Email address is optional.
%\email{zan.grad@tecnico.ulisboa.pt}

%%% If there's a second author:
% \author{}
% \address{}
% \curraddr{}
% \email{}

%%% To have the current date inserted, use \date{\today}:
%\date{\today}
%\subjclass[2020]{58H05 (22A25, 58H10, 53B15)}
%\keywords{Lie groupoid, Lie algebroid, double complex, bundle of ideals, multiplicative Ehresmann connection, compatible double structure, horizontal exterior covariant derivative.}

\begin{abstract}
  For a Lie groupoid $G$, the differential forms on its nerve comprise a double complex. A natural question is if this statement extends to forms with values in a representation $V$ of $G$. In this paper, we research two types of covariant derivatives which commute with the simplicial differential, yielding two types of \enquote{curved} double complexes of forms with coefficients in $V$. The naïve approach is to consider a linear connection $\nabla$ on $V$, in which case $\d{}^\nabla$ commutes with the simplicial differential if and only if $\nabla$ satisfies a certain (restrictive) invariance condition. The heart of this paper focuses on another, more compelling approach: using a multiplicative Ehresmann connection for a bundle of ideals. In this case, we obtain a geometrically richer curved double complex, where the cochain map is given by the horizontal exterior covariant derivative $\D{}$, which generalizes the well-known operator from the theory of principal bundles. Moreover, both differential operators $\d{}^\nabla$ and $\D{}$ are researched in the infinitesimal setting of Lie algebroids, as well as their relationship with the van Est map. We conclude by using the operator $\D{}$ to study the curvature of an (infinitesimal) multiplicative Ehresmann connection.
\end{abstract}

\vspace*{-1.0cm}

\maketitle
\thispagestyle{empty}

\vspace{-1.8em}

%%% To include a table of contents, uncomment the following line:
%\setlength{\cftbeforetoctitleskip}{-3em}
\tableofcontents

%\blfootnote{%Research supported in part by FCT Grant UI/BD/152069/2021. 
%\textit{Date:} \today.
%\textit{E-mail address:} \href{zan.grad@tecnico.ulisboa.pt}{\texttt{zan.grad@tecnico.ulisboa.pt}}
%}
\vspace{-2em}

\pagebreak

% blue color links after TOC
%\hypersetup{linkcolor=MidnightBlue}

\section{Introduction}
\label{sec:intro}
A differential form on a Lie groupoid $\G\rra M$ is said to be \textit{multiplicative} if it is compatible with the groupoid multiplication. Such forms naturally appear when describing global counterparts of (twisted) Poisson, Dirac, and Jacobi structures \cites{symplectic_groupoids,twisted_dirac,jacobi_structures}. 
Their infinitesimal counterparts on Lie algebroids, called \textit{infinitesimal multiplicative} (IM) \textit{forms} have been extensively studied in \cites{im_forms, linear_mult, local, meinrenken_pike}, and the above-mentioned geometric structures can in fact be viewed as Lie algebroids endowed with IM forms. In both the global and the infinitesimal realm, multiplicative forms can be understood cohomologically \cites{weil, ve_mein}.  More precisely, on any Lie groupoid, real-valued multiplicative forms of degree $q$ are exactly 1-cocycles of the Bott--Shulman--Stasheff complex
\begin{align*}
  \Omega^q\big(G^{(\bullet)}\big),
\end{align*}
endowed with the simplicial differential $\delta$. Here $G^{(p)}$ is the  manifold consisting of all $p$-tuples of composable arrows of a Lie groupoid $G\rra M$. Together with the de Rham differential, we obtain a double complex $(\Omega^\bullet(G^{(\bullet)}),\delta, \d{})$, meaning the two differentials commute: $\delta\d{}=\d{}\,\delta$.
\begin{align}
  \label{eq:bss}
  \begin{tikzcd}[ampersand replacement=\&, column sep=large, row sep=large]
    {\Omega^{q+1}(M)} \& {\Omega^{q+1}(G)} \& {\Omega^{q+1}(G^{(2)})} \& \cdots \\
    {\Omega^q(M)} \& {\Omega^{q}(G)} \& {\Omega^q(G^{(2)} )} \& \cdots
    \arrow["{\delta}", from=1-1, to=1-2]
    \arrow["{\delta}", from=1-2, to=1-3]
    \arrow["{\delta}", from=1-3, to=1-4]
    \arrow["{\d{}}", from=2-1, to=1-1]
    \arrow["{\delta}", from=2-1, to=2-2]
    \arrow["{\d{}}", from=2-2, to=1-2]
    \arrow["{\delta}", from=2-2, to=2-3]
    \arrow["{\d{}}", from=2-3, to=1-3]
    \arrow["{\delta}", from=2-3, to=2-4]
  \end{tikzcd}
  \end{align}
An important consequence of commuting differentials is that the de Rham differential \textit{preserves multiplicativity}. An analogous double complex also exists infinitesimally \cite{weil}: the complex whose 1-cocycles are the IM forms (of form degree $q$) is called the \textit{Weil complex}, denoted $W^{\bullet,q}(A)$. Moreover, when $A$ is the algebroid of $G$, the two  complexes are related by a cochain map $\Omega^q(G^{(\bullet)})\ra W^{\bullet,q}(A)$, called the \textit{van Est map}.

The complexes above are real-valued; they have values in the trivial representation on the trivial line bundle $\R_M$. However, trivial coefficients are of limited use, as multiplicative forms usually appear with coefficients in representations, and more generally, in representations up to homotopy \cites{diff_cohomology, spencer, vb-valued, homogeneous}. For instance, representation-valued forms already arise when describing contact groupoids (the global counterparts of Jacobi structures), where the trivial representation is twisted by a Reeb 1-cocycle. More importantly, non-trivial coefficients appear when dealing with connections on non-abelian bundle gerbes \cites{gerbes}, and more generally, when dealing with \textit{multiplicative Ehresmann connections} (MECs) for Lie groupoid extensions and fibrations \cites{mec, mecs_for_fibrations}. Although the simplicial differential $\delta$ naturally extends to the settings where coefficients are in arbitrary representations (as well as up to homotopy), there is no canonical operator that defines the columns, as was the case above with the de Rham operator $\d{}$ in diagram \eqref{eq:bss}. 

The main question of this paper is how to come up with differential operators which define the columns of $\Omega^{\bullet}(G^{(\bullet)};V)$ for forms valued in a representation $G\curvearrowright V$, such that they commute with the simplicial differential. For a general representation, a straightforward method is to simply take a linear connection $\nabla$ on $V$ and inspect the conditions on $\nabla$ so that $\d{}^\nabla$ commutes with the simplicial differential $\delta$. This condition turns out to be very restrictive (see \eqref{eq:invariance_intro} below), but there is a way of rectifying this for special types of representations. One of the main results of this paper is that in the case when the representation is a bundle of ideals for $G$, any multiplicative Ehresmann connection naturally induces such an operator, in both the global and infinitesimal case. The original motivation for this result actually comes from the objective of generalizing Yang--Mills theory to MECs, which was achieved in the companion paper \cite{ym_mecs}, for which the results presented here are crucial.

% \pagebreak
To state our main result for the case of groupoids, let us  set the stage by fixing a \textit{bundle of ideals} $\frak k\subset \ker\rho$ for a Lie groupoid $G$, that is, a vector subbundle of the algebroid $A$ of $G$ such that $\Ad\colon G\curvearrowright\ker\rho$ restricts to a representation on $\frak k$. Here $\rho\colon A\ra TM$ denotes the anchor of $A$. A prototypical class of examples comes from surjective submersive Lie groupoid morphisms $\Phi\colon G\ra H$ (covering the identity on the base), where one takes $\frak k=\ker(\Phi_*)$, i.e., $\frak k$ is the kernel of the induced algebroid map. 
% https://q.uiver.app/#q=WzAsMyxbMCwwLCJHIl0sWzIsMCwiSCJdLFsxLDEsIk0iXSxbMCwxLCJcXFBoaSJdLFswLDJdLFsxLDJdXQ==
% https://q.uiver.app/#q=WzAsMyxbMCwwLCJHIl0sWzIsMCwiSCJdLFsxLDEsIk0iXSxbMCwxLCJcXFBoaSJdLFswLDIsIiIsMix7Im9mZnNldCI6LTF9XSxbMSwyLCIiLDAseyJvZmZzZXQiOi0xfV0sWzAsMiwiIiwxLHsib2Zmc2V0IjoxfV0sWzEsMiwiIiwwLHsib2Zmc2V0IjoxfV1d
\[\begin{tikzcd}[column sep=small]
	G && H \\
	& M
	\arrow["\Phi", from=1-1, to=1-3]
	\arrow[shift left, from=1-1, to=2-2]
	\arrow[shift right, from=1-1, to=2-2]
	\arrow[shift left, from=1-3, to=2-2]
	\arrow[shift right, from=1-3, to=2-2]
\end{tikzcd}\]
Now, a \textit{multiplicative Ehresmann connection} for $\Phi$ is defined as an Ehresmann connection $E\subset TG$ for $\Phi$ which is compatible with the groupoid multiplication on $G$: the distribution $E$ is a Lie subgroupoid of $TG\rra TM$. Alternatively, it is also described as a multiplicative form $\omega\in\Omega^1_m(G;\frak k)$ satisfying  $\omega|_{\frak k}=\id_{\frak k}$. Our main theorem in the global case is the following.
\begin{customthm}{\ref{thm:deltaD}}
For a Lie groupoid $G\rra M$ with a bundle of ideals $\frak k$, any MEC $\omega$ induces the horizontal exterior covariant derivative $\D{}^\omega=h^*\d{}^\nabla$ which is a cochain map between the $\frak k$-valued Bott--Shulman--Stasheff complexes of adjacent degrees, that is, $\D{}^\omega\delta = \delta\D{}^\omega$. 
\begin{align*}
\begin{tikzcd}[ampersand replacement=\&, column sep=large, row sep=large]
	{\Omega^{q+1}(M;\frak k)} \& {\Omega^{q+1}(G;\frak k)} \& {\Omega^{q+1}(G^{(2)};\frak k)} \& \cdots \\
	{\Omega^q(M;\frak k)} \& {\Omega^{q}(G;\frak k)} \& {\Omega^q(G^{(2)};\frak k )} \& \cdots
	\arrow["{\delta}", from=1-1, to=1-2]
	\arrow["{\delta}", from=1-2, to=1-3]
	\arrow["{\delta}", from=1-3, to=1-4]
	\arrow["{\d{}^\nabla}", from=2-1, to=1-1]
	\arrow["{\delta}", from=2-1, to=2-2]
	\arrow["{\D{}^\omega}", from=2-2, to=1-2]
	\arrow["{\delta}", from=2-2, to=2-3]
	\arrow["{\D{}^\omega}", from=2-3, to=1-3]
	\arrow["{\delta}", from=2-3, to=2-4]
\end{tikzcd}
\end{align*}
% Additionally, $\D{}^\omega$ squares to zero if and only if $E=\ker \omega$ is involutive.
\end{customthm}
Here, $\nabla$ is a certain linear connection on the vector bundle $\frak k\ra M$ induced by $\omega$ (\sec\ref{sec:conn}), and $h^*$ denotes the precomposition with the horizontal projection $TG\ra E$. In other words, the theorem says that $(\Omega^\bullet(G^{(\bullet)};\frak k),\delta,\D{}^\omega)$ is a \textit{curved} double complex since in general $(\D{}^\omega)^2\neq 0$, and moreover, it is a usual (``flat'') double complex precisely when $E=\ker\omega$ is involutive. As usual, the measure of failure of involutivity is given by the \textit{curvature}. It is defined as 
\[\Omega^\omega=\D{}^\omega\omega,\]
and as a consequence of the theorem above, it is a \textit{multiplicative} 2-form since $\omega$ is itself a 1-cocycle. In general, studying the properties of the curvature of MECs rests heavily upon this theorem. For instance, it enables the research of \textit{primitive} MECs (\sec\ref{sec:primitive}), which are MECs with multiplicatively exact curvature: $[\Omega^\omega]=0$. It also enables us to see how curvature changes as we deform a multiplicative connection (Theorem \ref{thm:expansion}). 

For a general representation $V$, the situation is much more restrictive. As a preparation for the proof of above theorem, we recover in \sec\ref{sec:invariant_linear_connections} the result from \cite{mec}*{Appendix A} which says that if we simply fix a linear connection $\nabla$ on $V\ra M$ and use the exterior covariant derivative $\d{}^\nabla$ to define the columns of $\Omega^\bullet(G^{(\bullet)};V)$, then $\d{}^\nabla$ commutes with $\delta$ if and only if $\nabla$ is an \textit{invariant connection},
\begin{align}
  \label{eq:invariance_intro}
  s^*\nabla=t^*\nabla \quad\text{(under $s^*V\cong t^*V$ given by the action $G\curvearrowright V$).}
\end{align}
See Theorem \ref{thm:G_invariant} and \ref{thm:A_invariant} for the analogous result in the infinitesimal setting. 
Infinitesimally, invariance implies that the representation $A\curvearrowright V$ is induced by the connection, so it is a very restrictive condition. For instance, the connection $\nabla$ induced by a MEC (appearing in Theorem \ref{thm:deltaD}) is invariant if and only if $\frak k$ is abelian, up to connectedness conditions. In general, we show that the existence of an invariant connection is cohomologically obstructed (Theorems \ref{thm:obstruction_invariance_G} and \ref{thm:obstruction_invariance}).

\pagebreak
At last, let us mention the infinitesimal aspects of MECs. Establishing the infinitesimal version of Theorem \ref{thm:deltaD} is technically the most challenging part of the paper. To elaborate, we set the stage by fixing a Lie algebroid $A$; a \textit{bundle of ideals} $\frak k\subset \ker \rho$  is now defined by the condition that
\[
[\Gamma(A),\Gamma(\frak k)]\subset \Gamma(\frak k).
\]
Now, an \textit{infinitesimal MEC} (or an \textit{IM connection}) is a $\frak k$-valued IM 1-form denoted $(\C,v)$ with the additional assumption that the \textit{symbol} $v$ is a vector bundle splitting of the SES below; here $\C\colon\Gamma(A)\ra\Omega^1(M;\frak k)$ is a map satisfying a certain Leibniz rule together with $v\colon A\ra \frak k$, as well as a set of nontrivial \textit{IM conditions}, see \eqref{eq:c1}--\eqref{eq:c3}.
\begin{align*}
  \begin{tikzcd}[ampersand replacement=\&]
	0 \& {\frak k} \& A \& {A/\frak k} \& 0
	\arrow[from=1-1, to=1-2]
	\arrow[from=1-2, to=1-3]
	\arrow["v", bend left=30, from=1-3, to=1-2]
	\arrow["\phi",from=1-3, to=1-4]
	\arrow[from=1-4, to=1-5]
\end{tikzcd}
\end{align*}
Alternatively, an IM connection can be viewed as an Ehresmann connection $E$ for $\phi$, which is moreover a VB-subalgebroid of the tangent algebroid $TA\Ra TM$.

As before, we expect any IM connection to define a horizontal exterior covariant derivative, but the situation is now more difficult since Weil cochains $W^{\bullet,q}(A;\frak k)$ are not forms on a manifold (rather, they can be viewed as forms on the dg-manifold $A[1]$ valued in the dg-vector bundle $A[1]\times_M \frak k$). As such, the precomposition with the horizontal projection $h^*$ has no straightforward formula in the infinitesimal setting, and in turn, the very definition of the operator $\D{}^{(\C,v)}$ is elusive. We overcome this challenge in \sec\ref{sec:homogeneous_horizontal} by passing to homogeneous cochains introduced in \cite{homogeneous}. There, the horizontal projection is again precomposition with $h\colon TA\to E$, and we transport it back to an explicit formula for $h^*$ on Weil cochains. Our main theorem in the infinitesimal setting is the following.

\begin{customthm}{\ref{thm:deltaD_inf}}
  For a Lie algebroid $A$ with a bundle of ideals $\frak k$, an IM connection $\omega$ induces the horizontal exterior covariant derivative $\D{}^{(\C,v)}=h^*\d{}^\nabla$ which is a cochain map between the $\frak k$-valued Weil complexes of adjacent degrees.
\begin{align*}
\begin{tikzcd}[ampersand replacement=\&, column sep=large,row sep=large]
	{\Omega^{q+1}(M;\frak k)} \& {W^{1,q+1}(A;\frak k)} \& {W^{2,q+1}(A;\frak k)} \& \cdots \\
	{\Omega^q(M;\frak k)} \& {W^{1,q}(A;\frak k)} \& {W^{2,q}(A;\frak k)} \& \cdots
	\arrow["{\delta}", from=1-1, to=1-2]
	\arrow["{\delta}", from=1-2, to=1-3]
	\arrow["{\delta}", from=1-3, to=1-4]
	\arrow["{\d{}^\nabla}", from=2-1, to=1-1]
	\arrow["{\delta}", from=2-1, to=2-2]
	\arrow["{\D{}^{(\C,v)}}", from=2-2, to=1-2]
	\arrow["{\delta}", from=2-2, to=2-3]
	\arrow["{\D{}^{(\C,v)}}", from=2-3, to=1-3]
	\arrow["{\delta}", from=2-3, to=2-4]
\end{tikzcd}
\end{align*}
% Additionally, $\D{}^{(\C,v)}$ squares to zero if and only if the associated distribution $E\subset TA$ is involutive.
\end{customthm}
This theorem enables us to study the \textit{curvature} $\Omega^{(\C,v)}=\D{}^{(\C,v)}(\C,v)$ of an IM connection $(\C,v)$, which is now an IM 2-form measuring the failure of involutivity of the corresponding Ehresmann connection $E\subset TA$, with respect to the usual Lie bracket of vector fields on $A$. In other words, similarly as in the global case, we get that $(W^{\bullet,\bullet}(A;\frak k),\delta,\D{}^{(\C,v)})$ is a curved double complex which is flat if and only if $\Omega^{(\C,v)}$ vanishes. Moreover, the discovery of the operator $\D{}^{(\C,v)}$ in terms of an explicit formula enables us to prove the \textit{infinitesimal Bianchi identity} (Theorem \ref{thm:bianchi_inf}),
\[
\D{}^{(\C,v)}\Omega^{(\C,v)}=0.
\]
In addition to the results mentioned above, we also show the following.
\begin{enumerate}[label={(\roman*)}]
  \item We show there exist cohomological obstructions for the existence of MECs in both the global and infinitesimal setting (Propositions \ref{prop:existence} and \ref{prop:existence_inf}).
  \item %We relate the operators $\d{}^\nabla$ induced by an invariant connection $\nabla$ on a representation $V$ in the global and infinitesimal realm, using the van Est map (Theorem \ref{thm:van_est_G_A}). 
  We show that $D^\omega$ and $D^{(\mathcal C,v)}$ commute with the van Est map on multiplicative forms (Theorem \ref{thm:van_est_D}). This fails for general cochains, and we explain why.
  \item We show that if $\frak k$ has a semisimple fibre, every multiplicative connection is primitive with a unique representative 2-form on the base (Proposition \ref{prop:semisimple}).
\end{enumerate}

% fix a \textit{bundle of ideals} $\frak k\subset \ker\rho$ on a Lie groupoid $G\rra M$, meaning $G\curvearrowright\ker\rho$ restricts to a representation on $\frak k$. For example, any

%This paper studies the differential in the vertical direction --- raising the form degree $q$ --- for the case of \textit{representation-valued} forms $\Omega^q(G^{(p)};V)$.
% Given any representation $V\ra M$ of a Lie groupoid $G\rra M$, the 

\newpage
\SkipTocEntry\subsubsection*{Outline of the paper}
In \sec\ref{sec:invariant_linear_connections}, we consider arbitrary representations, endowed with linear connections. After recalling the simplicial differentials of the global and infinitesimal cochain complexes, we compute the commutator $[\delta,\d{}^\nabla]$ of the simplicial differential $\delta$ and the exterior covariant derivative $\d{}^\nabla$ in both the global and infinitesimal settings. A direct corollary is that the invariance condition on $\nabla$ is necessary and sufficient for $\d{}^\nabla$ to be a cochain map (Theorems \ref{thm:G_invariant} and \ref{thm:A_invariant}). We %obtain an obstruction class for the existence of invariant connections (Theorem \ref{thm:obstruction_invariance}) and 
show that under the invariance assumption, the derivatives commute with the van Est map (Theorem \ref{thm:van_est_G_A}), and that this holds on normalized forms regardless of invariance. At last, we show there is a cohomological obstruction to the existence of invariant connections (Theorems $\ref{thm:obstruction_invariance_G}$ and $\ref{thm:obstruction_invariance}$).

In \sec\ref{sec:bss_boi}, we examine the  case when the representation is a bundle of ideals $\frak k \subset \ker \rho$, focusing first on Lie groupoids. By providing a novel way of expressing the linear connection $\nabla$ on $\frak k$ induced by a multiplicative Ehresmann connection $\omega\in\Omega_{m}^1(G;\frak k)$, we show that the horizontal exterior covariant derivative $\D{}^\omega = h^* \circ \d{}^\nabla$ is a cochain map (Theorem \ref{thm:deltaD}), whereas $\d{}^\nabla$ is not, unless $\frak k$ is abelian (and $G$ is source-connected). The upshot is that the horizontal projection of differential forms is a cochain map due to the multiplicativity of the connection $\omega$, so it suffices to show that the horizontal projection of the commutator of $\delta$ and $\d{}^\nabla$ vanishes.

In \sec\ref{sec:weil_boi}, we establish an analogous result in the infinitesimal setting of Lie algebroids. That is, an infinitesimal multiplicative Ehresmann connection $(\C, v)\in\Omega_{im}^1(A;\frak k)$ for a bundle of ideals $\frak k$ gives rise to the horizontal exterior covariant derivative $\D{}^{(\C, v)}$, which is a cochain map (Theorem \ref{thm:deltaD_inf}). As in the global case, the definition of $\D{}^{(\C, v)}$ requires us to define the horizontal projection $h^*$ of Weil cochains, whose definition is now significantly less trivial---to obtain the formula for $h^*$ and show it is a cochain map, we employ the viewpoint of VB-algebroids in \sec\ref{sec:homogeneous_horizontal} and utilize the complex of exterior cochains (Theorem \ref{thm:derivation_hor_proj}).
%Additionally, the horizontal projection does not directly commute with the van Est map but does so up to a cochain homotopy.  
We show that $\D{}^\omega$ and $\D{}^{(\C, v)}$ commute with the van Est map at the level of multiplicative forms (Theorem \ref{thm:van_est_D}), and explain why this fails for general cochains.

Finally, \sec\ref{sec:applications} is dedicated to the study of applications of the curved double complexes from \sec\ref{sec:bss_boi} and \sec\ref{sec:weil_boi}. We first observe that a specific cohomology class provides the obstruction to the existence of (infinitesimal) multiplicative Ehresmann connections. %On proper groupoids, we establish that multiplicative Ehresmann connections always exist, by applying a Vanishing Cohomology Theorem for horizontal cohomology, yielding a new proof of \cite{mec}*{Theorem 4.2}. 
Next, we investigate the applications of the developed horizontal exterior covariant derivatives. In \sec\ref{sec:curvature}, we study the fundamental properties of the curvature of a multiplicative connection, including the infinitesimal Bianchi identity (Theorem \ref{thm:bianchi_inf}) that is made possible by the discovery of $\D{}^{(\C,v)}$. Additionally, we examine how the curvature transforms under affine deformations of multiplicative connections (Theorems \ref{thm:expansion} and \ref{thm:expansion_inf}), where $\D{}^\omega$ and $\D{}^{(\C,v)}$ play an important role in the first-order term of the expansion. At last, in \sec\ref{sec:primitive}, we focus on a special class of \textit{primitive} multiplicative connections---those with cohomologically trivial curvature. We explore several interesting properties of this subclass of multiplicative connections. Notably, we show that if $\frak k$ has a semisimple fibre, every multiplicative connection is primitive with a unique representative 2-form on the base, which is proved by showing that the zeroth and first horizontal simplicial cohomology groups are trivial. This also holds in the case when $\frak k$ is the isotropy bundle of a transitive algebroid. We finish our discussion  by providing an explicit description of primitive IM connections in the case when $\frak k$ is abelian.

\SkipTocEntry\subsection*{Acknowledgments}
I am grateful to Rui L.\ Fernandes, Alejandro Cabrera, Thiago Drummond, João Nuno Mestre and Lennart Obster for the helpful discussions and suggestions. I am also grateful to my PhD supervisors Ioan M\u{a}rcu\c{t} and Pedro Resende for their guidance and advice.

This work is supported by PhD Grant \href{http://dx.doi.org/10.54499/UI/BD/152069/2021}{UI/BD/152069/2021} of FCT, Portugal. It is further supported by COST Action CaLISTA (CA21109) of the European Cooperation in Science and Technology (\url{http://cost.eu}). 

\SkipTocEntry\section*{Preliminaries}
We assume the reader is familiar with the basics of Lie groupoids and algebroids, found in \cite{mackenzie}. Let us only recall the basics regarding representations. A \textit{representation} of a Lie groupoid $G\rra M$ on a vector bundle $\pi\colon V\ra M$ is a smooth map $\Delta\colon G\tensor[_s]{\times}{_\pi}V\ra V$, denoted $(g,v)\mapsto g\cdot v$ satisfying
\begin{align*}
  \pi(g\cdot v)\in V_{t(g)},\quad 1_x\cdot v=v,\quad g\cdot (h\cdot v)=(gh)\cdot v,
\end{align*}
for all $g,h\in G$ and $v\in V$ for which the expressions are defined. Additionally, for any $g\in G$, we require the action map $\Delta_g\colon V_{s(g)}\ra V_{t(g)}$, $v\mapsto g\cdot v$ to be linear.

Passing to the infinitesimal level, we first recall that a \textit{representation} of a Lie algebroid $A\Ra M$ is a flat $A$-connection on a vector bundle $V\rightarrow M$. We recall this is a bilinear map
\begin{align*}
  \nabla^A\colon \Gamma(A)\times \Gamma(V)\ra\Gamma(V),\quad (\alpha,\sigma)\mapsto \nabla^A_\alpha \sigma,
\end{align*}
which is $C^\infty(M)$-linear in the first argument, satisfies the Leibniz identity
\begin{align*}
  \nabla^A_\alpha(f\sigma)=f\nabla^A_\alpha \sigma+ \rho(\alpha)(f)\sigma,
\end{align*}
and the flatness condition reads $\nabla^A_{[\alpha,\beta]}=[\nabla^A_\alpha,\nabla^A_\beta]$. If $A=\ker\d t|_M$ is the Lie algebroid of a Lie groupoid $G$, a representation of $G$ on $V$ induces a representation of $A$ on $V$, by setting
\[
(\nabla^A_\alpha \xi)(x)=\deriv\lambda 0 \phi^{\alpha^L}_{\lambda}(1_x)\cdot \xi(\phi^{\rho(\alpha)}_{\lambda}(x)),
\]
%The action in the definition makes sense due to the equality $s\circ \phi_{\lambda}^{\alpha^L}=\phi_\lambda^{\rho\alpha}\circ s$ which holds on account of $s$-relatedness of $\alpha^L$ to $\rho\alpha$, on the appropriate domain. 
where $\alpha^L\in\vf(G)$ denotes the left-invariant extension of $\alpha\in \Gamma(A)$ on $G$, $\phi^{\alpha^L}_{\lambda}$ is its flow at time $\lambda$, and we denoted by $\phi^{\smash{\rho(\alpha)}}_\lambda$ the flow of the vector field $\rho(\alpha)\in\vf(M)$. Since $s_*\alpha^L=\rho(\alpha)$, the two flows are related by $s\circ \phi^{\smash{\alpha^L}}_\lambda=\phi^{\smash{\rho(\alpha)}}_\lambda\circ s$. 

In the case when $G$ has simply connected $s$-fibres, this defines a bijective correspondence between representations of $G$ and representations of $A$ by virtue of Lie's second fundamental theorem for groupoids \cite{lie2}.

%\todo[inline]{Write about differential forms with values in Lie algebra bundles and their graded structure.}

\section{Invariant linear connections on representations}
\label{sec:invariant_linear_connections}
\subsection{Bott--Shulman--Stasheff complex with values in a representation}
Let $V$ be a representation of a Lie groupoid $G\rra M$ and consider the set of  $V$-valued differential forms of degree $q$ on the level $p$ of the nerve of $G$,
\begin{align*}
&\Omega^{0,q}(G;V)\coloneqq \Omega^q(M;V),\\
&\Omega^{p,q}(G;V)\coloneqq \Omega^q(G^{(p)};(s\circ \pr_p)^*V),\quad \text{for }p\geq 1,
\end{align*}
where $\pr_p\colon G^{(p)}\ra G$ denotes the projection to the last element. For a fixed degree $q$, this becomes a cochain complex, with the differential defined as follows. At level $p=0$,
\begin{align}
\begin{split}
\label{eq:delta_0}
&\delta^{0}\colon \Omega^q(M;V)\ra \Omega^q(G;s^*V),\quad(\delta^0\gamma)_g=(s^*\gamma)_g-g^{-1}\cdot (t^*\gamma)_g.
\end{split}
\end{align}
At level $p\geq 1$, the differential is given by the alternating sum of pullbacks along the face maps,
\begin{align}
\begin{split}
\label{eq:delta_l}
&\delta^p\colon \Omega^{p,q}(G;V)\ra \Omega^{p+1,q}(G;V),\quad
\delta^p=\sum_{i=0}^{p}(-1)^i\big(f_i^{(p+1)}\big)^* + (-1)^{p+1}\Phi_*\big(f_{p+1}^{(p+1)}\big)^*
\end{split}
\end{align}
where the face maps $f_i^{(p+1)}\colon G^{(p+1)}\ra G^{(p)}$ are defined as
%\[
%f_i^{(\ell+1)}=
%\begin{cases}
%\text{removal of last element}&i=1,\\
%\text{removal of first element}&i=\ell+1,\\
%\text{composition of $(i-1)$-th and $i$-th element}&\text{otherwise},
%\end{cases}
%\]
\[
f_i^{(p+1)}(g_1,\dots,g_{p+1})=
\begin{cases}
(g_2,\dots,g_{p+1})&i=0,\\
(g_1,\dots,g_{i}g_{i+1},\dots,g_{p+1})&1\leq i\leq p,\\
(g_1,\dots,g_{p})&i=p+1,
\end{cases}
\]
%whenever $\ell\in\mathbb N$; at $\ell=0$, the face maps are of course just the source and target maps: $f_1^0=t$ and $f_2^0=s$.
and the map $\Phi_*\colon \Omega^q(G^{(p+1)};(t\circ\pr_{p+1})^*V)\ra \Omega^q(G^{(p+1)};(s\circ\pr_{p+1})^*V)$ changes the coefficients of forms---it is induced by the isomorphism of vector bundles
\begin{align*}
  \phi\colon t^*V\ra s^*V,\quad \phi(g,v)=(g,g^{-1}\cdot v)
\end{align*}
or rather its $(p+1)$-level analogue,
\begin{align*}
&\Phi\colon(t\circ \pr_{p+1})^*V\ra (s\circ \pr_{p+1})^*V,\\
&\Phi(g_1,\dots,g_{p+1},v)= (g_1,\dots,g_{p+1},g_{p+1}^{-1}\cdot v).
\end{align*}
\begin{definition}
  Let $G\rra M$ be a Lie groupoid with a representation $V$. The differential $\delta$ is called the \textit{simplicial differential} of differential forms on the nerve. At any fixed degree $q\geq 0$, it defines a cochain complex called the \textit{Bott--Shulman--Stasheff complex}, 
  \[
  (\Omega^{\bullet,q}(G;V),\delta),
  \]
  whose cohomology is called the \textit{simplicial cohomology} of differential $q$-forms on the nerve,
  \[
    H^{p,q}(G;V)\coloneqq H^p(\Omega^{\bullet,q}(G;V),\delta).
  \]
\end{definition}
\begin{example}
At level $p=0$, the cocycles $\ker\delta^0$ are called \textit{invariant forms} on $M$, i.e., forms $\gamma\in\Omega^q(M;V)$ which satisfy
$(t^*\gamma)_g=g\cdot (s^*\gamma)_g.$
At level $p=1$, the cocycles $\ker\delta^1$ are called \textit{multiplicative forms} on $G$. These are forms $\omega\in\Omega^q(G;s^*V)$ that satisfy
\begin{align}
\label{eq:multiplicative}
  \omega_{gh}(\d m(X_i,Y_i))_i=\omega_h(Y_i)_i+h^{-1}\cdot \omega_g(X_i)_i
\end{align}
for any vectors $(X_i,Y_i)\in T_{(g,h)}G^{(2)}$, where $s(g)=t(h)$, and we denote them by $\Omega^q_m(G;V)$. Here, $(X_i)_i\coloneq (X_i)_{i=1}^q$ denotes a tuple $(X_1,\dots, X_q)$ with the number $q$ seen from the context, and similarly $(\d m (X_i,Y_i))_i$ denotes $(\d m (X_1,Y_1),\dots,\d m (X_q,Y_q))$; such notation will be used often. 

The cocycles arising as coboundaries $\im \delta^0\subset \Omega^\bullet_m(G;V)$ are called \textit{cohomologically trivial}.
\end{example}
% The respective cohomology with respect to the simplicial differential will be denoted by
% \begin{align*}
%   H^{p,q}(G;V)\coloneq H^p(\Omega^{\bullet,q}(G;V)),
% \end{align*}
% where $q\geq 0$ is fixed.

\subsubsection{Invariant linear connections in the global setting}

We now assume that a connection $\nabla$ on the vector bundle $V\ra M$ is given, with no a priori additional assumptions regarding the compatibility of $\nabla$ with the groupoid action $G\curvearrowright V$. 
\begin{definition}
  The \textit{exterior covariant derivative} of $V$-valued forms $\Omega^{p,q}(G;V)$ is given by
  \[
  \d{}^{\nabla^{s\circ\pr_p}}\colon \Omega^{p,q}(G;V)\ra \Omega^{p,q+1}(G;V),
  \]
where $\nabla^{s\circ\pr_p}$ denotes the pullback along $s\circ \pr_p\colon G^{(p)}\ra M$ of a given connection $\nabla$ on $V$. When no confusion arises, we will simply denote it by $\d{}^\nabla$. 
\end{definition}

The following result tells us precisely when $\d{}^\nabla$ commutes with $\delta$. It was already obtained in \cite{mec}*{Proposition A.7} for $p=1$, however, we hereby prove it by computing the explicit formula for the commutator 
\begin{align}
  \label{eq:commutator}
  [\d{}^\nabla,\delta]\colon \Omega^{p,q}(G;V)\ra \Omega^{p+1,q+1}(G;V)
\end{align}
which was not obtained there, and will be needed in \sec\ref{sec:bss_boi}.
\begin{theorem}
\label{thm:G_invariant}
Let $\nabla$ be a connection on a representation $V$ of a Lie groupoid $G\rra M$. The map $\d{}^\nabla$ is a cochain map if and only if $\nabla$ is $G$-invariant, that is, if the following tensor vanishes:
\begin{align}
  \label{eq:invariance_form_G}
  \Theta\in \Omega^1(G;\Hom(t^*V,s^*V)),\quad \Theta(X)\xi=\phi(\nabla_X^t\xi)-\nabla_X^s \phi (\xi),
\end{align}
for any $X\in \vf(G)$ and $\xi\in\Gamma(t^*V)$.  Hence, if $\nabla$ is $G$-invariant, $\d{}^\nabla$ preserves multiplicativity.
\end{theorem}
 %As promised, we now compute the commutator \eqref{eq:commutator}.
\begin{lemma}
\label{lem:G_invariant}
Let $\nabla$ be a connection on a representation $V$ of a Lie groupoid $G\rra M$. At level $p=0$, the commutator of $\delta$ and $\d{}^\nabla$ reads
\begin{align}
\label{eq:commutator_zero}
  (\d{}^{\nabla^{s}}\delta^0-\delta^0\d{}^{\nabla})\omega=\Theta\wedge t^*\omega
\end{align}
for any $\omega\in \Omega^q(M;V)$. At any higher level $p\geq 1$, there holds
\begin{align}
\label{eq:commutator_higher}
  (\d{}^{\nabla^{s\circ\pr_{p+1}}}\delta^p-\delta^p\d{}^{\nabla^{s\circ\pr_p}})\omega=(-1)^{p}(\pr_{p+1})^*\Theta\wedge\big (f^{(p+1)}_{p+1}\big)^*\omega
\end{align}
for any $\omega\in\Omega^{p,q}(G;V)$. 
Explicitly, the $(q+1)$-form on $G^{(p+1)}$ on the right-hand side reads
\begin{align*}
  \Big((\pr_{p+1})^*\Theta&\wedge\big (f^{(p+1)}_{p+1}\big)^*\omega\Big)(X_1,\dots,X_{q+1})\\
  &=\sum_{i=1}^{q+1}(-1)^{i+1}\Theta(X_i^{p+1})\cdot\omega((X_1^1,\dots, X_1^p),\dots,\widehat{(X_i^1,\dots,X_i^p)},\dots (X_{q+1}^1,\dots, X_{q+1}^p)),
\end{align*}
for any tangent vectors $X_i=(X_i^1,\dots, X_i^{p+1})\in TG^{(p+1)}$ from the same fibre.
\end{lemma}
\begin{proof}
We will make use of the basic fact that if $E\ra B$ is any vector bundle equipped with a connection $\nabla$, and if $\pi\colon N\ra B$ is any smooth map, then the pullback connection $\nabla^\pi$ on $\Omega^\bullet(N;\pi^*E)$ is characterized by the following identity involving its exterior covariant derivative,
\begin{align}
\label{eq:pullback}
  \d{}^{\nabla^\pi}\pi^*=\pi^*\d{}^\nabla.
\end{align}
Using it on the face maps $\pi=f_i^{(p+1)}$ for $i\leq p$, with connection $\nabla^{s\circ\pr_p}$  in place of $\nabla$, the identity $s\circ\pr_p\circ \smash{f_i^{(p+1)}}=s\circ\pr_{p+1}$ yields
\[
\d{}^{\nabla^{s\circ\pr_{p+1}}}(f_i^{(p+1)})^*=(f_i^{(p+1)})^*\d{}^{\nabla^{s\circ\pr_p}}.
\]
Moreover, using \eqref{eq:pullback} on $\pi=f_{p+1}^{(p+1)}$, the identity $s\circ\pr_p\circ f^{(p+1)}_{p+1}=s\circ\pr_p=t\circ\pr_{p+1}$ yields
\[
  \d{}^{\nabla^{t\circ\pr_{p+1}}}(f_{p+1}^{(p+1)})^*=(f_{p+1}^{(p+1)})^*\d{}^{\nabla^{s\circ\pr_p}}.
\]
Using these two equalities and the defining equation \eqref{eq:delta_l} of $\delta$, we compute
\begin{align*}
%\label{eq:intermed_d_delta}
  \d{}^{\nabla^{s\circ\pr_{p+1}}}\delta^p-\delta^p\d{}^{\nabla^{s\circ\pr_p}} &=(-1)^{p+1}\big(\d{}^{\nabla^{s\circ\pr_{p+1}}}\Phi_* (f_{p+1}^{(p+1)})^* -\Phi_*(f_{p+1}^{(p+1)})^*\d{}^{\nabla^{s\circ\pr_p}}\big)\\
  &=(-1)^{p+1}\big({\d{}^{\nabla^{s\circ\pr_{p+1}}}}\Phi_*-\Phi_*\d{}^{\nabla^{t\circ\pr_{p+1}}}\big)(f_{p+1}^{(p+1)})^*.
\end{align*}
Let us expand the differential operator on the right-hand side of the last expression, on an arbitrary form $\gamma\in\Omega^q(G^{(p+1)};(t\circ\pr_{p+1})^*V)$:
\begin{align*}
(\d{}^{\nabla^{s\circ\pr_{p+1}}}\Phi_*\gamma&-\Phi_*\d{}^{\nabla^{t\circ\pr_{p+1}}}\gamma)(X_1,\dots,X_{q+1})
\\
&=\textstyle\sum\limits_i(-1)^{i+1}\big(\nabla_{X_i}^{s\circ\pr_{p+1}}\circ\Phi-\Phi\circ\nabla_{X_i}^{t\circ\pr_{p+1}}\big)\gamma(X_1,\dots,\widehat X_i,\dots, X_{q+1}),
\end{align*}
for any vectors $X_i\in T G^{(p+1)}$ over the same point, where the terms with sums over double indices have cancelled out. But it is not hard to see that 
\[
\nabla^{s\circ\pr_{p+1}}\circ\Phi-\Phi\circ\nabla^{t\circ\pr_{p+1}}=-\pr_{p+1}^*\Theta,
\]
which is enough to be checked on pullback sections $(t\circ\pr_{p+1})^*\xi$ for $\xi\in \Gamma(V)$. Indeed, the stated equality is immediate once we use the relation $\Phi\circ\pr_{p+1}^*=\pr_{p+1}^*\circ\phi$. At last, taking $\smash{\gamma=(f_{p+1}^{(p+1)})^*}\omega$ concludes the proof of \eqref{eq:commutator_higher}. Identity \eqref{eq:commutator_zero} is simple and left to the reader.
\end{proof}

Theorem \ref{thm:G_invariant} now follows by surjectivity and submersivity of face maps and projections, since the lemma above implies that $\Theta=0$ if and only if at any level $p\geq 0$ (hence, at all levels) the exterior covariant derivative commutes with the simplicial differential.
In conclusion, any $G$-invariant connection $\nabla$ on $V$ yields the columns of a curved double complex,
\[
(\Omega^{\bullet,\bullet}(G;V),\delta,\d{}^\nabla).
\]
We once again emphasize the notable feature that $\d{}^\nabla$ does not square to zero unless $\nabla$ is flat, in which case we obtain a flat double complex.
\begin{align*}
%\label{eq:bss}
\begin{tikzcd}[ampersand replacement=\&, column sep=large, row sep=large]
	{\Omega^{q+1}(M;V)} \& {\Omega^{q+1}(G;s^*V)} \& {\Omega^{q+1}(G^{(2)};(s\circ\pr_2)^*V)} \& \cdots \\
	{\Omega^q(M;V)} \& {\Omega^{q}(G;s^*V)} \& {\Omega^q(G^{(2)};(s\circ\pr_2)^*V )} \& \cdots
	\arrow["{\delta}", from=1-1, to=1-2]
	\arrow["{\delta}", from=1-2, to=1-3]
	\arrow["{\delta}", from=1-3, to=1-4]
	\arrow["{\d{}^\nabla}", from=2-1, to=1-1]
	\arrow["{\delta}", from=2-1, to=2-2]
	\arrow["{\d{}^\nabla}", from=2-2, to=1-2]
	\arrow["{\delta}", from=2-2, to=2-3]
	\arrow["{\d{}^\nabla}", from=2-3, to=1-3]
	\arrow["{\delta}", from=2-3, to=2-4]
\end{tikzcd}
\end{align*}

\subsection{Weil complex  with values in a representation}
Let $V$ be a representation of a Lie algebroid $A\Ra M$. Following \cite{homogeneous}*{\sec 4}, the infinitesimal analogue of the complex $\Omega^{p,q}(G;V)$ is the Weil complex: it is defined as the family of sets $W^{p,q}(A;V)$ of sequences $c=(c_0,\dots,c_p)$ of alternating $\R$-multilinear maps 
\[
c_k\colon\underbrace{\Gamma(A)\times\dots\times \Gamma(A)}_{p-k \text{ copies}}\ra \Omega^{q-k}(M;S^k(A^*)\otimes V),
\]
whose failure at being $C^\infty(M)$-multilinear is controlled by the \textit{Leibniz identity}
\[
c_k(f\alpha_1,\alpha_2,\dots,\alpha_{p-k}\|\cdot)=fc_k(\alpha_1,\dots,\alpha_{p-k}\|\cdot)+\d f\wedge c_{k+1}(\alpha_2,\dots,\alpha_{p-k}\|\alpha_1,\cdot).
\]
Here, the notation reflects the fact that each $c_k$ in total inputs $p$ sections of $A$, accounting also for the $k$ arguments in which $c_k$ is symmetric. More precisely, we write
\[c_k(\alpha_1,\dots,\alpha_{p-k}\|\beta_1,\dots,\beta_k)=c_k(\ul\alpha\|\ul\beta) \in \Omega^{q-k}(M;V)\]
to mean that $c_k$ is $\R$-multilinear and antisymmetric in the first $p-k$ arguments, and  $C^\infty(M)$-multilinear and symmetric in the last $k$ arguments. The elements of $W^{p,q}(A;V)$ will be called \textit{Weil cochains}. The term $c_0$ will be called the \textit{leading term}, while the higher terms will be called the \textit{correction terms}.

\begin{remark}
There is a general principle that we will often use, as in \cite{weil}. Suppose we want to define a map $F\colon C\ra W^{p,q}(A;V)$ to the Weil complex from some set $C$. To do so, we will often prescribe the leading term $(Fc)_0$ and then infer the correction terms $(Fc)_1,\dots,(Fc)_p$ by consecutively applying the Leibniz identity. The leading term will always have a clear conceptual meaning, whereas the correction terms will usually be more complicated. We will usually omit this heuristic procedure of obtaining the correction terms, and will instead provide the proof of well-definedness of $F$. 
\end{remark}
As in the setting of groupoids, the sets $W^{p,q}(A;V)$ form a cochain complex, with the differential $\delta\colon W^{p,q}(A;V)\ra W^{p+1,q}(A;V)$ defined as follows. First  note that for arbitrary $m,n\geq 0$ the space $\Omega^m(M;S^n(A^*)\otimes V)$ is a module for the Lie algebra $\Gamma(A)$, with the action given  by the Lie derivative induced by the representation, defined by the chain rule:
\begin{align*}
(\L^A_\alpha\gamma)(\ul\beta)(\ul X)&=\nabla^A_\alpha\gamma(\ul\beta)(\ul X)\\
&-\textstyle\sum_i\gamma(\beta_1,\dots,[\alpha,\beta_i],\dots,\beta_n)(\ul X)\\
&-\textstyle\sum_j\gamma(\ul\beta)(X_1,\dots,[\rho(\alpha),X_j],\dots, X_m).
\end{align*}
Then, the leading term $(\delta c)_0$ for a given sequence $c=(c_0,\dots,c_p)\in W^{p,q}(A;V)$ is defined as the \textit{Koszul differential} of the leading term $c_0$, with respect to this module structure:
\begin{align*}
(\delta c)_0 (\alpha_0,\dots,\alpha_p)&=\textstyle\sum_i(-1)^{i}\L^A_{\alpha_i} \big(c_0(\alpha_0,\dots,\widehat{\alpha_i},\dots,\alpha_{p})\big)\\
&+\textstyle\sum_{i<j}(-1)^{i+j} c_0([\alpha_i,\alpha_j],\alpha_0,\dots,\widehat{\alpha_i},\dots,\widehat{\alpha_j},\dots,\alpha_{p}).
\end{align*}
The correction terms can be found by consecutively applying the Leibniz rule; as obtained in \cite{homogeneous}*{equation 4.2}, they read
\begin{align}
\begin{split}
\label{eq:delta_inf}
(-1)^k(\delta c)_k &(\alpha_0,\dots,\alpha_{p-k}\|\beta_1,\dots,\beta_k)\\
&=\textstyle\sum_{i=0}^{p-k}(-1)^{i}\L^A_{\alpha_i} \big(c_k(\alpha_0,\dots,\widehat{\alpha_i},\dots,\alpha_{p-k}\|\cdot)\big)(\ul\beta)\\
&+\textstyle\sum_{i<j}^{p-k}(-1)^{i+j} c_k([\alpha_i,\alpha_j],\alpha_0,\dots,\widehat{\alpha_i},\dots,\widehat{\alpha_j},\dots,\alpha_{p-k}\|\ul\beta)\\
&-\textstyle\sum_{j=1}^k \iota_{\rho(\beta_j)} c_{k-1}(\ul\alpha\|\beta_1,\dots,\widehat{\beta_j},\dots,\beta_k).
\end{split}
\end{align}
Notably, the $(p+1)$-th term simply reads
\[
(\delta c)_{p+1}(\ul\beta)=(-1)^p\textstyle\sum_j \iota_{\rho(\beta_j)} c_p(\beta_1,\dots,\widehat{\beta_j},\dots,\beta_p).
\]
\begin{definition}
  Let $A\Ra M$ be a Lie algebroid with a representation $V$. The differential $\delta$ is called the \textit{Koszul differential} of Weil cochains. At any fixed degree $q\geq 0$, it defines a cochain complex called the \textit{Weil complex}, 
  \[
  (W^{\bullet,q}(A;V),\delta),
  \]
  whose cohomology is called the \textit{Weil cohomology},
  \[
    H^{p,q}(A;V)\coloneqq H^p(W^{\bullet,q}(A;V),\delta).
  \]
\end{definition}
%To deduce the correction components, one uses the relations $\L^A_\alpha(f\gamma)=f\L^A_\alpha\gamma+\rho(\alpha)(f)\gamma$.
\begin{example}
At $p=0$, we have $W^{0,q}(A;V)=\Omega^q(M;V)$ and $\delta$ is given by
\begin{align}
\label{eq:delta0_im}
(\delta c)_0(\alpha)=\L^A_\alpha c, \qquad(\delta c)_1(\beta)=\iota_{\rho(\beta)}c.
\end{align}
The cocycles at $p=0$ are called \textit{invariant forms}, denoted $\Omega^q_{\inv}(M;V)$. At level $p=1$, the differential of a Weil cochain $c=(c_0,c_1)$ reads 
\begin{align}
\label{eq:delta_p=1}
\begin{split}
  (\delta c)_0(\alpha_1,\alpha_2)&=\L^A_{\alpha_1} c_0(\alpha_2)-\L^A_{\alpha_2} c_0(\alpha_1)-c_0[\alpha_1,\alpha_2],\\
  (\delta c)_1(\alpha,\beta)&=-(\L^A_{\alpha} c_1)(\beta)+\iota_{\rho(\beta)}c_0(\alpha)=-\L^A_{\alpha} (c_1(\beta))+ c_1[\alpha,\beta]+\iota_{\rho(\beta)}c_0(\alpha),\\
  (\delta c)_2(\beta_1,\beta_2)&=-\iota_{\rho(\beta_1)} c_1(\beta_2)-\iota_{\rho(\beta_2)} c_1(\beta_1).
\end{split}
\end{align}
The cocycles at $p=1$ are called \textit{infinitesimal multiplicative forms} (more briefly, \textit{IM forms}) or \textit{Spencer operators}, and we denote $\Omega^\bullet_{im}(A;V)=\ker\delta^1$. By the equations above, a cocycle $c=(c_0,c_1)$ satisfies
\begin{align}
  c_0[\alpha_1,\alpha_2]&=\L^A_{\alpha_1} c_0(\alpha_2)-\L^A_{\alpha_2} c_0(\alpha_1),\label{eq:c1}\tag{C.1}\\
  c_1[\alpha,\beta]&=\L^A_{\alpha} (c_1(\beta))-\iota_{\rho(\beta)}c_0(\alpha),\label{eq:c2}\tag{C.2}\\
  \iota_{\rho(\beta_1)} c_1(\beta_2)&=-\iota_{\rho(\beta_2)} c_1(\beta_1).\label{eq:c3}\tag{C.3}
  \end{align}
The equations \eqref{eq:c1}--\eqref{eq:c3} will be referred to as the \textit{compatibility conditions} for $(c_0,c_1)$; as noted in \cite{spencer}*{Remark 2.7}, condition \eqref{eq:c2} follows from the Leibniz identity and \eqref{eq:c1} unless $\dim M =q+1$, however, it nonetheless often turns out to be useful. The cocycles $\im\delta^0\subset \Omega^\bullet_{im}(A;V)$ arising as coboundaries are called \textit{cohomologically trivial} IM forms. 
\end{example}
%\begin{remark}
%The relation between $\Omega^{p,q}(G;V)$ and $W^{p,q}(G;V)$ was studied in \cite{weil} for $V=\R$ and in \cite{homogeneous} for nontrivial coefficients, and is provided by a certain cochain map---the so-called \textit{van Est map}. Under certain connectivity assumptions, it induces an isomorphism on the level of cohomology at certain degrees.
%\end{remark}
\subsubsection{Invariant linear connections in the infinitesimal setting}
As before, we now again assume that a connection $\nabla$ on the representation $V\ra M$ is given, with no a priori assumptions on its compatibility with the algebroid action $A\curvearrowright V$.
\begin{definition}
  Let $\nabla$ be a connection on a representation $V$ of a Lie algebroid $A\Ra M$. The \textit{exterior covariant derivative} of $V$-valued Weil cochains is the map
  \[\d{}^\nabla\colon W^{p,q}(A;V)\ra W^{p,q+1}(A;V),\]
  whose leading term is defined on any $c\in W^{p,q}(A;V)$ as the exterior covariant derivative of its leading term $c_0$, that is,
  \[
  (\d{}^\nabla c)_0(\ul\alpha)=\d{}^\nabla c_0(\ul\alpha).
  \]
  The correction coefficients $(\d{}^\nabla c)_k$ are defined by
  \begin{align}
    \label{eq:dnabla}
    (-1)^k(\d{}^\nabla c)_k&(\ul\alpha\|\ul\beta)=\d{}^\nabla c_k (\ul\alpha\|\ul\beta)-\textstyle\sum_{i=1}^k c_{k-1}(\beta_i,\ul\alpha\|\beta_1,\dots,\widehat{\beta_i},\dots,\beta_k).
    \end{align}
\end{definition}
In the definition above, the correction terms were obtained using the general principle. As always, one needs to check the obtained map is well-defined, which is a straightforward computation that we have provided in Lemma \ref{lemma:wd_dnabla}.
\begin{example}
At $p=1$, the exterior covariant derivative of $c=(c_0,c_1)\in W^{1,q}(A;V)$ reads
\begin{align}
\label{eq:ext_cov_der_p=1}
(\d{}^\nabla c)_0(\alpha)=\d{}^\nabla c_0(\alpha),\quad (\d{}^\nabla c)_1(\beta)=c_0(\beta)-\d{}^\nabla c_1(\beta).
\end{align}
This coincides with the operator defined in \cite{mec}*{Proposition A.10}. Going one level further, at $p=2$ we obtain:
\begin{align}
\label{eq:ext_cov_der_p=2}
\begin{split}
(\d{}^\nabla c)_0(\alpha_1,\alpha_2)&=\d{}^\nabla c_0(\alpha_1,\alpha_2),\\
(\d{}^\nabla c)_1(\alpha\|\beta)&=c_0(\beta,\alpha)-\d{}^\nabla c_1(\alpha\|\beta),\\
(\d{}^\nabla c)_2(\beta_1,\beta_2)&=\d{}^\nabla c_2(\beta_1,\beta_2)-c_1(\beta_1\|\beta_2)-c_1(\beta_2\|\beta_1),
\end{split}
\end{align}
for any $c=(c_0,c_1,c_2)\in W^{2,q}(A;V)$.
\end{example}
We now prove the infinitesimal analogue of Theorem \ref{thm:G_invariant}.
\begin{theorem}
\label{thm:A_invariant}
Let $\nabla$ be a connection on a representation $V$ of a Lie algebroid $A\Ra M$. The map $\d{}^\nabla$ is a cochain map if and only if $\nabla$ is $A$-invariant, that is, if there holds
\[\nabla^A_\alpha=\nabla_{\rho(\alpha)}\quad\text{and}\quad\iota_{\rho(\alpha)}R^\nabla=0,\]
for any $\alpha\in A$. In particular, the operator defined by \eqref{eq:ext_cov_der_p=1} maps IM forms to IM forms. 
\end{theorem}
\begin{remark}
  \label{rem:G_inv_implies_A_inv}
  In \cite{mec}*{Proposition A.8}, it was shown that $G$-invariance implies $A$-invariance, and that the converse also holds if $G$ is source-connected. A new proof of this is given in Proposition \ref{prop:g_inv_a_inv}, where the statement is seen as a consequence of the properties of the van Est map.
\end{remark}
\noindent As with groupoids, we prove the theorem by computing the formula for the commutator 
\[
[\d{}^\nabla,\delta]\colon W^{p,q}(A;V)\ra W^{p+1,q+1}(A;V),
\]
which will turn out to be vital later.
\begin{lemma}
\label{lem:A_invariant}
Let $V$ be a representation of a Lie algebroid $A\Ra M$ and let $\nabla$ be a  connection on $V$. The commutator $[\d{}^\nabla,\delta]$ evaluated on a cochain $c=(c_0,\dots,c_p)\in W^{p,q}(A;V)$ reads
\begin{align}
(\d{}^\nabla\delta c)_k(\alpha_0,&\dots,\alpha_{p-k}\|\ul\beta)-(\delta\d{}^\nabla c)_k(\alpha_0,\dots,\alpha_{p-k}\|\ul\beta)\nonumber\\
&=\textstyle\sum_{i=0}^{p-k} (-1)^i T(\alpha_i)\wedge c_k(\alpha_0,\dots,\widehat{\alpha_i},\dots,\alpha_{p-k}\|\ul\beta)\label{eq:A_invariant_commutator}\\
&+\textstyle\sum_{j=1}^k\theta(\beta_j)\cdot c_{k-1}(\ul\alpha\|\beta_1,\dots,\widehat{\beta_j},\dots,\beta_k),\nonumber
\end{align}
where the tensor $\theta\colon A\rightarrow \End(V)$ and the map $T\colon \Gamma(A)\ra \Omega^1(M;\End V)$ are given by
\begin{align}
\label{eq:invariance_form}
\begin{split}
\theta(\alpha)&=\nabla^A_\alpha-\nabla_{\rho(\alpha)},\\
T(\alpha)&=\d{}^\nabla\theta(\alpha)-\iota_{\rho(\alpha)}R^\nabla.
\end{split}
\end{align}
\end{lemma}
\begin{remark}
The exterior covariant derivative $\d{}^\nabla$ in the definition of the map $T$ is with respect to the induced connection on $\End V$. We also note that $(T,\theta)\in W^{1,1}(A;\End V)$ is a Weil cochain with values in the induced representation of $A$ on $\End V$, since it clearly satisfies the Leibniz rule \[T(f\alpha)=f T(\alpha)+\d f\otimes \theta(\alpha).\]
In \sec \ref{sec:obstruction_invariance}, we will see that this is actually an IM form with values in $\End V$, and we will use it to define an obstruction class to the existence of $A$-invariant connections. The pair $(T,\theta)$ will henceforth be called the \textit{$A$-invariance form} of the connection $\nabla$.
%We also note the $(p+1)$-th term simply reads
%\[
%(\d{}^\nabla\delta c)_{p+1}(\ul\beta)-(\delta\d{}^\nabla c)_{p+1}(\ul\beta)=\textstyle\sum_{j}\theta(\beta_j)\cdot c_p(\beta_1,\dots,\widehat{\beta_j},\dots,\beta_{p+1}).
%\]
\end{remark}
\begin{proof}
Let us first check the theorem holds at the level of leading terms. By the definition of maps $\d{}^\nabla$ and $\delta$, there holds
\begin{align*}
(\d{}^\nabla\delta c)_0(\ul\alpha)&-(\delta\d{}^\nabla c)_0(\ul\alpha)=\textstyle\sum_i (-1)^i \big({\d{}^\nabla}\L^A_{\alpha_i}-\L^A_{\alpha_i}\d{}^\nabla\big) c_0(\alpha_0,\dots,\widehat{\alpha_i},\dots,\alpha_{p}),
\end{align*}
so we must compute $({\d{}^\nabla}\L^A_\alpha-\L^A_\alpha\d{}^\nabla)\gamma$ for any $\gamma\in\Omega^q(M;V)$. Notice that 
\[
\L^A_\alpha\gamma=\L^\nabla_{\rho(\alpha)}\gamma+\theta(\alpha) \gamma,
\]
and using Cartan's magic formula, we get
\begin{align}
\label{eq:d_lie}
\d{}^\nabla\L^\nabla_X\gamma-\L^\nabla_X\d{}^\nabla\gamma =(\d{}^\nabla)^2\iota_X\gamma-\iota_X(\d{}^\nabla)^2\gamma=-(\iota_X R^\nabla)\wedge \gamma
\end{align}
for any $X\in \vf(M)$. A simple computation using these two identities shows
\begin{align*}
{\d{}^\nabla}\L^A_\alpha\gamma-\L^A_\alpha\d{}^\nabla\gamma=T(\alpha)\wedge\gamma,
\end{align*}
proving our claim for the leading coefficient.

We now claim that it was actually enough to check that the formula \eqref{eq:A_invariant_commutator} holds at the level of leading coefficients. This follows from a similar argument as in Remark \ref{rem:action_trick} after this proof, and was inspired by the technique in the proof of \cite{weil}*{Proposition 4.1}. To be precise, first observe that the right-hand side of equation \eqref{eq:A_invariant_commutator} defines an operator 
\[c\mapsto (T,\theta)\wedge c,\]
which maps Weil cochains to Weil cochains by Lemma \ref{lem:T_theta_wedge}, raising the degree and the level by 1. Hence, the case $q\leq\dim M-1$ is proved. If $q\geq\dim M$, the idea is to apply the same proof for the leading coefficient, but on a larger algebroid. Take any action of $A$ on a surjective submersion $\mu\colon P\ra M$ with $\dim P\geq q+1$ and form the action algebroid $A\ltimes P\Ra P$. The surjective algebroid morphism $p\colon A\ltimes P\ra A$ then induces a pullback of Weil cochains, denoted $p^*$, defined by \eqref{eq:weil_pullback}. The following naturality relations of the operators $\d{}^\nabla$ and $(T,\theta)\wedge \cdot$ now hold:
\begin{align*}
  p^*\d{}^\nabla&=\d{}^{\nabla^\mu}p^*,\\ p^*\big((T,\theta)\wedge c\big)&=p^*(T,\theta)\wedge p^*c,
\end{align*}
% \[\begin{tikzcd}[column sep=small]
% 	{W^{p,q}(A;V)} & {W^{p,q+1}(A;V)} \\
% 	{W^{p,q}(A\ltimes P;\mu^*V)} & {W^{p,q+1}(A\ltimes P;\mu^*V)}
% 	\arrow["{\d{}^\nabla}", from=1-1, to=1-2]
% 	\arrow["{p^*}"', from=1-1, to=2-1]
% 	\arrow["{p^*}", from=1-2, to=2-2]
% 	\arrow["{\d{}^{\nabla^\mu}}"', from=2-1, to=2-2]
% \end{tikzcd}\]
where $p^*(T,\theta)\in W^{1,1}(A\ltimes P;\mu^*V)$ is precisely the $A$-invariance form of the pullback connection $\nabla^\mu$ on the representation $\mu^*V$ of $A\ltimes P$. Since $p^*$ is injective, we are done.
\end{proof}
\begin{remark}
  \label{rem:action_trick}
  Suppose we are given two Weil cochains $c,c'\in W^{p,q}(A;V)$, and we want to show $c=c'$. Observe that for any $k\in \set{0,\dots,p-1}$, $c_k=c_k'$ implies $c_{k+1}=c'_{k+1}$ by the Leibniz rule, provided $q-k\leq \dim M$. Hence, $c_0=c_0'$ implies $c=c'$ provided $q\leq \dim M$, so in this case it is enough to check that the leading terms coincide. On the other hand, if $q>\dim M$, then the first nontrivial component of a nonzero cochain $c$ is $c_{q-\dim M}$, and we would like to see it as the leading term. The following trick is a precise way of doing so.
  
  Take any action of $A\Ra M$ on a surjective submersion $\mu\colon P\ra M$, as in \cite{actions}*{Definition 3.1} and consider the action algebroid $A\ltimes P\Ra P$. For clarity, let us briefly recall its construction. To begin with, its underlying vector bundle is the pullback bundle $\mu^* A=P\times_M A\ra P$. As such, its space of sections  is canonically isomorphic as a $C^\infty(P)$-module to 
  \[\Gamma(\mu^*A)\cong C^\infty(P)\otimes_{C^\infty(M)}\Gamma(A),\]
  and it is thus generated by $\Gamma(A)$.\footnote{The isomorphism $C^\infty(P)\otimes\Gamma(A)\ra \Gamma(\mu^*A)$ is given by  $h\otimes\alpha\mapsto (p\mapsto(p,h(p)\alpha_{\mu(p)}))$, and the tensor product over $C^\infty(M)$ identifies $h\otimes f\alpha=(f\circ \mu)h\otimes \alpha$, where $h\in C^\infty(P), \alpha\in \Gamma(A)$ and $f\in C^\infty(M)$.} The algebroid structure on $A\ltimes P$ is determined by
  \begin{align*}
    \rho_{A\ltimes P}(1\otimes \alpha)&=X^\alpha,\\
    [1\otimes\alpha,1\otimes\beta]_{A\ltimes P}&=1\otimes[\alpha,\beta]
  \end{align*}
  on generators $\alpha,\beta\in\Gamma(A)$, and extended to the whole space by $C^\infty(P)$-linearity of the anchor and the Leibniz rule of the bracket. Here, $\alpha\mapsto X^\alpha$ denotes the action $\Gamma(A)\ra \vf(P)$. We thus obtain the following natural surjective submersion of Lie algebroids.
% https://q.uiver.app/#q=WzAsNCxbMCwwLCJBXFxsdGltZXMgUCJdLFsxLDAsIkEiXSxbMCwxLCJQIl0sWzEsMSwiTSJdLFswLDEsInAiXSxbMiwzLCJcXG11IiwyXSxbMCwyXSxbMSwzXV0=
\begin{align}
  \label{eq:p_mu}
  \begin{split}
    \begin{tikzcd}[ampersand replacement=\&]
    	{A\ltimes P} \& A \\
    	P \& M
    	\arrow["p", from=1-1, to=1-2]
    	\arrow[Rightarrow, from=1-1, to=2-1]
    	\arrow[Rightarrow, from=1-2, to=2-2]
    	\arrow["\mu"', from=2-1, to=2-2]
    \end{tikzcd}
  \end{split}
\end{align}
Moreover, a representation $V$ of $A$ induces a representation $\mu^* V$ of $A\ltimes P$, determined by
\[\nabla^{A\ltimes P}_{1\otimes \alpha}(1\otimes\xi)=1\otimes \nabla^A_\alpha\xi,\] 
for any $\xi\in \Gamma(V)$, and extended by $C^\infty(P)$-linearity and the Leibniz rule. 
Now observe that the pullback along the algebroid map \eqref{eq:p_mu} induces a monomorphism of Weil complexes,
\begin{align}
  \label{eq:weil_pullback}
  \begin{split}
    &p^*\colon W^{p,q}(A;V)\ra W^{p,q}(A\ltimes P;\mu^*V),\\
    &(p^* c)_k(1\otimes \alpha_1,\dots,1\otimes \alpha_{p-k}\|1\otimes \beta_1, \dots, 1\otimes \beta_k)=\mu^* c_k(\ul\alpha\|\ul\beta),
  \end{split}
\end{align} 
defined on the generators and extended by the Leibniz rule in the antisymmetric arguments and by $C^\infty(P)$-linearity in the symmetric ones. Since $\dim P\geq\dim M$, $p^*$ is well-defined. The upshot now is that in the case $q>\dim M$, the first nontrivial term of a nonzero Weil cochain $c\in W^{p,q}(A;V)$ is $c_{q-\dim M}$, but choosing any space $P$ with $\dim P\geq q$ translates this term into $(p^* c)_0\neq 0$. Hence, showing $c=c'$ is equivalent to showing $(p^*c)_0=(p^*c')_0$.
\end{remark}

Theorem \ref{thm:A_invariant} follows directly from the last lemma. In conclusion, an $A$-invariant connection $\nabla$ on $V$ now yields the columns of a curved double complex, 
\[
(W^{\bullet,\bullet}(A;V),\delta,\d{}^\nabla).
\]
As in the case of groupoids, it has the obvious feature that $\d{}^\nabla$ does not square to zero unless $\nabla$ is flat, in which case we obtain a flat double complex.
\begin{align}
\begin{tikzcd}[ampersand replacement=\&, column sep=large, row sep=large]
	{\Omega^{q+1}(M;V)} \& {W^{1,q+1}(A;V)} \& {W^{2,q+1}(A;V)} \& \cdots \\
	{\Omega^q(M;V)} \& {W^{1,q}(A;V)} \& {W^{2,q}(A;V)} \& \cdots
	\arrow["{\delta}", from=1-1, to=1-2]
	\arrow["{\delta}", from=1-2, to=1-3]
	\arrow["{\delta}", from=1-3, to=1-4]
	\arrow["{\d{}^\nabla}", from=2-1, to=1-1]
	\arrow["{\delta}", from=2-1, to=2-2]
	\arrow["{\d{}^\nabla}", from=2-2, to=1-2]
	\arrow["{\delta}", from=2-2, to=2-3]
	\arrow["{\d{}^\nabla}", from=2-3, to=1-3]
	\arrow["{\delta}", from=2-3, to=2-4]
\end{tikzcd}
\end{align}
\begin{remark}
  If $\nabla$ is an $A$-invariant connection, its curvature tensor is an invariant form,
  \[
  R^\nabla\in \Omega^2_\inv(M;\End V),
  \]
  with respect to the induced representation of $A$ on $\End V$. This follows from computing
  \[
  \L^A_\alpha R^\nabla = {\underbrace{\L^\nabla_{\rho(\alpha)} R^\nabla}_{\mathclap{\d{}^{\nabla^{\End V}}\iota_{\rho(\alpha)} R^\nabla}}} + \theta(\alpha)\circ R^\nabla-R^\nabla\circ \theta(\alpha)=0,
  \]
  where the under-brace is due to the Bianchi identity $\d{}^{\nabla^{\End V}} R^\nabla=0$.
\end{remark}
%\todo[inline]{Give some examples. Trivial coefficients with the canonical flat connection, vector bundles with the trivial bracket acting on themselves, etc.}

\subsection{Van Est map and \texorpdfstring{$\mathrm{d}^\nabla$}{exterior covariant derivatives}}
It has already been established in \cite{homogeneous}*{Theorem 4.4} that the van Est map relating the Bott--Shulman--Stasheff with the Weil complex, 
\[
\ve\colon \Omega^{p,q}(G;V)\ra W^{p,q}(A;V),
\]
is a cochain map---it commutes with the differentials $\delta_G$ and $\delta_A$. Moreover, it was also shown there that under the assumption that $G$ is source $p_0$-connected for some $p_0\geq 0$, the induced map in cohomology $\ve\colon H^{p,q}(G;V)\ra H^{p,q}(A;V)$ is an isomorphism for all $p\leq p_0$, and injective for $p=p_0+1$, for each fixed $q$. 

In this section, we prove that under the invariance assumption on $\nabla$, $\ve$ also commutes with the exterior covariant derivatives $\d{}_G^\nabla$ and $\d{}_A^\nabla$ introduced in the previous sections. This result generalizes \cite{weil}*{Proposition 4.1}, which considered the trivial representation with the canonical flat connection.

The van Est map is defined as follows. Given a section $\alpha\in\Gamma(A)$, we will denote the flow of the left-invariant vector field $\cev\alpha$ at time $\lambda$ by $\phi^{\cev\alpha}_\lambda$. Furthermore, $\cev\alpha$ defines a vector field on the $p$-th level of the nerve $G^{(p)}$, which we will denote by 
\[
\alpha^{(p)}|_{(g_1,\dots,g_p)}\coloneqq (0_{g_1},\dots,0_{g_{p-1}},\cev\alpha_{g_p}),
\]
and whose flow is given by $\phi^{\alpha^{(p)}}_\lambda(g_1,\dots,g_p)=(g_1,\dots,g_{p-1},\phi^{\cev\alpha}_\lambda(g_p))$. We define the operator
\begin{align*}
&R_\alpha\colon \Omega^{p,q}(G;V)\ra \Omega^{p-1,q}(G;V),\\
&R_\alpha \omega|_{(g_1,\dots,g_{p-1})}=j_p^*\left(\deriv\lambda 0\phi^{\cev\alpha}_\lambda \big(1_{s(g_{p-1})}\big)\cdot \big(\phi^{\alpha^{(p)}}_\lambda\big)^*\omega\right),
\end{align*}
where $j_p\colon G^{(p-1)}\ra G^{(p)}$ is the degeneracy map that inserts the identity at the last factor (for an even more explicit expression of $R_\alpha$ see Example \ref{ex:RJ} below). This enables us to define the leading term of the van Est map:
\begin{align}
\label{eq:ve_leading}
\ve (\omega)_0(\alpha_1,\dots,\alpha_p)=\sum_{\sigma\in S_p}(\sgn\sigma) R_{\alpha_{\sigma(1)}}\dots R_{\alpha_{\sigma(p)}}\omega.
\end{align}
The correction terms are obtained from equation \eqref{eq:ve_leading} using the Leibniz rule. To write them, one shows that $R_{f\alpha}=f R_\alpha+\d f\wedge J_\alpha$, where $J_\alpha\colon \Omega^{p,q}(G;V)\ra \Omega^{p-1,q-1}(G;V)$ reads
\begin{align}
\label{eq:J_alpha}
  J_\alpha\omega=j_p^*(\iota_{\alpha^{(p)}}\omega).
\end{align}
As obtained in \cite{homogeneous}*{equation 4.4}, the correction terms then read
\begin{align*}
\ve(\omega)_k(\alpha_1,\dots,\alpha_{p-k}\|\beta_1,\dots,\beta_k)=(-1)^{\frac{k(k+1)}2}\sum_{\sigma\in S_p}(\sgn\sigma) (-1)^{\epsilon(\sigma,k)} D_{{\sigma(1)}}\dots D_{{\sigma(p)}}\omega
\end{align*}
where we are denoting
\begin{align*}
D_j&=\begin{cases}
J_{\beta_j},&\text{if }1\leq j\leq k,\\
R_{\alpha_{k-j}},&\text{if }k+1\leq j\leq p,
\end{cases}\\
\epsilon(\sigma,k)&=\#\set{(i,j)\in\set{1,\dots,k}^2\given i<j\text{ and }\sigma^{-1}(i)>\sigma^{-1}(j)}.
\end{align*}
The sign that $\epsilon$ induces is sometimes also called the \textit{Koszul sign}.
\begin{example}
\label{ex:RJ}
At level $p=2$, the maps $R_\alpha$ and $J_\alpha$ on $\omega\in \Omega^{2,q}(G;V)$ read
\begin{align*}
(R_\alpha\omega)_g(X_i)_{i=1}^q&=\deriv\lambda 0\phi^{\cev\alpha}_\lambda (1_{s(g)})\cdot \omega\big(X_i,\d(\phi^{\cev\alpha}_\lambda)_{1_{s(g)}}\d u_{s(g)} \d s_g X_i\big)_i,\\
(J_\alpha\omega)_g(X_i)_{i=1}^{q-1}&=\omega_{(g,1_{s(g)})}\big((0_g,\alpha_{s(g)}),(X_i,\d u_{s(g)} \d s_g X_i)_i\big)
\end{align*}
for any vectors $X_i\in T_g G$. We hereby provide some intuition regarding the signs appearing in the van Est map. The first correction term of $\ve(\omega)$ for $\omega$ as above reads
\[
\ve(\omega)_1(\alpha\|\beta)=-(J_\beta R_\alpha- R_\alpha J_\beta)\omega.
\]
As another example, the second correction term $\ve(\omega)_2(\alpha\|\beta_1,\beta_2)$ for $\omega\in\Omega^{3,q}(G;V)$ equals
\begin{align*}
  -(J_{\beta_1}J_{\beta_2}R_\alpha+J_{\beta_2}J_{\beta_1}R_\alpha-J_{\beta_1}R_\alpha J_{\beta_2}-J_{\beta_2}R_\alpha J_{\beta_1}+R_\alpha J_{\beta_1}J_{\beta_2}+R_\alpha J_{\beta_2}J_{\beta_1})\omega\in \Omega^{q-2}(M;V).
\end{align*}
A pictorial way of interpreting the signs in this expression, ignoring the factor $(-1)^{k(k+1)/2}$, is provided by the diagram below. Consider all paths $\Omega^{3,q}(G;V)$ to $\Omega^{q-2}(M;V)$ following horizontal and diagonal lines. Starting with the path that first follows the maximal number of horizontal segments, we obtain $JJR$. Take a positive sign there, symmetrize the expression in the indices of $J$'s and antisymmetrize in the indices of $R$'s. For other possible paths, we follow a similar procedure, with the difference that for any change of path $JR\ra RJ$, we obtain an additional minus sign; this uncovers the role of the Koszul sign.
{
\[\begin{tikzcd}
	{\Omega^{q}(M;V)} & {\Omega^{q}(G;s^*V)} & {\Omega^{2,q}(G;V)} & {\Omega^{3,q}(G;V)} \\
	{\Omega^{q-1}(M;V)} & {\Omega^{q-1}(G;s^*V)} & {\Omega^{2,q-1}(G;V)} & {\Omega^{3,q-1}(G;V)} \\
	{\Omega^{q-2}(M;V)} & {\Omega^{q-2}(G;s^*V)} & {\Omega^{2,q-2}(G;V)} & {\Omega^{3,q-2}(G;V)}
	\arrow[no head, from=1-1, to=1-2]
	\arrow[no head, from=1-2, to=1-3]
	\arrow["J"', color={rgb,255:red,32;green,121;blue,63}, from=1-3, to=2-2]
	\arrow["R"', color={rgb,255:red,32;green,121;blue,63}, from=1-4, to=1-3]
	\arrow["J"', color={rgb,255:red,186;green,28;blue,30}, from=1-4, to=2-3]
	\arrow["J", shift left, color={rgb,255:red,14;green,72;blue,139}, from=1-4, to=2-3]
	\arrow[no head, from=2-1, to=1-1]
	\arrow[no head, from=2-1, to=2-2]
	\arrow[no head, from=2-2, to=1-2]
	\arrow["J"', shift right, color={rgb,255:red,32;green,121;blue,63}, from=2-2, to=3-1]
	\arrow["J", color={rgb,255:red,186;green,28;blue,30}, from=2-2, to=3-1]
	\arrow[no head, from=2-3, to=1-3]
	\arrow["R", color={rgb,255:red,186;green,28;blue,30}, from=2-3, to=2-2]
	\arrow[no head, from=2-3, to=2-4]
	\arrow["J", color={rgb,255:red,14;green,72;blue,139}, from=2-3, to=3-2]
	\arrow[no head, from=2-4, to=1-4]
	\arrow[no head, from=3-1, to=2-1]
	\arrow[no head, from=3-2, to=2-2]
	\arrow["R", color={rgb,255:red,14;green,72;blue,139}, from=3-2, to=3-1]
	\arrow[no head, from=3-2, to=3-3]
	\arrow[no head, from=3-3, to=2-3]
	\arrow[no head, from=3-3, to=3-4]
	\arrow[no head, from=3-4, to=2-4]
\end{tikzcd}\]
}
\end{example}

\begin{theorem}
\label{thm:van_est_G_A}
Let $\nabla$ be a connection on a representation $V$ of a Lie groupoid $G\rra M$ with Lie algebroid $A\Ra M$. If $\nabla$ is $G$-invariant, then the van Est map commutes with the exterior covariant derivatives:
\begin{align}
  \label{eq:ve_d}
  \ve \d{}_G^\nabla=\d{}_A^\nabla \ve.
\end{align}
Moreover, this equality holds on all normalized forms regardless of $G$-invariance of the connection $\nabla$, so in particular, it holds for multiplicative forms.
\end{theorem}
%\begin{remark}
  We recall from \cite{weil} that a form $\omega\in\Omega^{p,q}(G;V)$ is said to be \textit{normalized}, if its pullbacks along all the degeneracy maps $G^{(p-1)}\ra G^{(p)}$ vanish. Any multiplicative form $\omega\in\Omega_m^q(G;V)$ necessarily satisfies $u^*\omega=0$, so it is normalized.
%\end{remark}
\begin{lemma}
  Let $\nabla$ be a connection on a representation $V$ of a Lie groupoid $G\rra M$. If either the connection $\nabla$ is $G$-invariant, or a given form $\omega\in\Omega^{p,q}(G;V)$ satisfies $j_p^*\omega=0$, then there holds  
\begin{align}
  \label{eq:R_jL}
  R_\alpha\omega=j_p^* \L^{\nabla^{s\circ\pr_p}}_{\alpha^{(p)}}\omega,
\end{align}
for any $\alpha\in\Gamma(A)$.
\end{lemma}
\begin{remark}
Given a connection $\nabla$ on a vector bundle $V\ra M$, the Lie derivative of a $V$-valued differential form $\omega\in\Omega^q(M;V)$ is defined by Cartan's formula, $\L^\nabla_X\omega=\d{}^\nabla\iota_X\omega+\iota_X\d{}^\nabla\omega$ for any $X\in \vf(M)$. Equivalently, we can use parallel transport:
  \[
  (\L^\nabla_X\omega)_x(Y_i)_i=\deriv \lambda 0 \tau(\gamma^X_x)^\nabla_{\lambda,0}\big(\omega(\d(\phi^X_\lambda)_x(Y_i))_i\big),
  \]
  where $\gamma^X_x$ denotes the (maximal) integral path of $X$ starting at $x$, and \[\tau(\gamma^X_x)^\nabla_{\lambda,0}\colon V_{\gamma^X_x(\lambda)}\ra V_x\] denotes the parallel transport along $\gamma^X_x$ with respect to the connection $\nabla$ from time $\lambda$ to $0$.
\end{remark}
\begin{proof}
  Consider first the level $p=1$. Let us compute the difference
  \begin{align*}
    \big(R_\alpha\omega-u^*\L^{\nabla^s}_{\alpha^L}\omega\big)_x(X_i)_i=\deriv \lambda 0\bigg(\Delta_{\phi^{\alpha^L}_\lambda(1_x)}-\tau\big(\gamma^{\alpha^L}_{1_x}\big)^{\nabla^s}_{\lambda,0}\bigg)\underbrace{\omega\big({\d(\phi^{\alpha^L}_\lambda)_{1_x}\d u (X_i)}\big)_i}_{\xi_{\gamma^{\rho\alpha}_{x}(\lambda)}}
  \end{align*}
  where $\Delta$ denotes the action of $G$ on $V$, and the expression denoted by $\xi$ defines a section of $V$ along the path $\gamma^{\rho\alpha}_x$. However, for any section $\xi$ along the path $\gamma^{\rho\alpha}_x$, there holds
  \begin{align*}
    \deriv \lambda 0\bigg(\Delta_{\phi^{\alpha^L}_\lambda(1_x)}-\tau\big(\gamma^{\alpha^L}_{1_x}\big)^{\nabla^s}_{\lambda,0}\bigg)\cdot \xi_{\gamma^{\rho\alpha}_{x}(\lambda)}=\nabla^A_{\alpha_x}\xi-\nabla_{\rho\alpha_x}\xi=\theta(\alpha_x)\cdot \xi_x,
  \end{align*}
  where we have observed there holds $\tau\big(\gamma^{\alpha^L}_{1_x}\big)^{\nabla^s}=\tau(s\circ\gamma^{\alpha^L}_{1_x})^\nabla=\tau(\gamma^{\rho\alpha}_x)^\nabla$ since $\nabla^s$ is a pullback connection and $s_*\alpha^L=\rho\alpha$. This implies
  \begin{align}
    \label{eq:R_uL_theta}
    R_\alpha\omega-u^*\L^{\nabla^s}_{\alpha^L}\omega=\theta(\alpha)\cdot u^*\omega.
  \end{align}
  More generally, at any higher level $p>1$, for a given $\omega\in \Omega^{p,q}(G;V)$ we similarly obtain
  \begin{align}
  \big(R_\alpha\omega-j_p^*\L^{\nabla^{s\circ \pr_p}}_{\alpha^{(p)}}\omega\big)_{(g_1,\dots,g_{p-1})}=\theta(\alpha_{s(g_{p-1})})\cdot (j_p^*\omega)_{(g_1,\dots,g_{p-1})},
  \end{align}
  proving the lemma since $G$-invariance implies $A$-invariance.
\end{proof}
\begin{proof}[Proof of Theorem \ref{thm:van_est_G_A}]
Let us first check the theorem holds at the level of leading terms, 
\begin{align}
  \label{eq:ve_dnabla_0}
  (\ve \d{}^\nabla\omega)_0=(\d{}^\nabla\ve\omega)_0.
\end{align}
We do so by first establishing the following identity for the case when either $\nabla$ is $G$-invariant, or the form $\omega$ satisfies $j_p^*\omega=0$:
\begin{align}
\label{eq:R_d}
R_\alpha \d{}^{\nabla^{s\circ\pr_p}}\omega=\d{}^{\nabla^{s\circ\pr_{p-1}}}R_\alpha\omega.
\end{align}
Similarly to equation \eqref{eq:d_lie}, the commutator of the $V$-valued Lie derivative with $\d{}^\nabla$ reads
\[
\L^{\nabla^{s\circ\pr_p}}_{\alpha^{(p)}}\d{}^{\nabla^{s\circ\pr_p}}-\d{}^{\nabla^{s\circ\pr_p}}\L^{\nabla^{s\circ\pr_p}}_{\alpha^{(p)}}=\iota_{\alpha^{(p)}}R^{\nabla^{s\circ\pr_p}}\wedge \cdot=(s\circ\pr_p)^*(\iota_{\rho\alpha}R^\nabla)\wedge \cdot
\]
Moreover, using the identity \eqref{eq:pullback} on $\pi=j_p$ with $\nabla^{s\circ\pr_p}$ as the connection, we get
\[
\d{}^{\nabla^{s\circ\pr_{p-1}}}j_p^*=j_p^*\d{}^{\nabla^{s\circ\pr_p}},
\]
since $s\circ\pr_p\circ j_p=s\circ\pr_{p-1}$. Hence, by the last lemma, we obtain
\[
  R_\alpha \d{}^{\nabla^{s\circ\pr_p}}\omega-\d{}^{\nabla^{s\circ\pr_{p-1}}} R_\alpha \omega=(s\circ \pr_{p-1})^* (\iota_{\rho\alpha} R^\nabla)\wedge j_p^*\omega,
\]
This proves the identity \eqref{eq:R_d} and thus also \eqref{eq:ve_dnabla_0} for the case when $\nabla$ is $G$-invariant. Dropping the assumption of invariance and instead assuming $\omega$ is normalizable, identity \eqref{eq:R_d} now follows by the fact that $R_\alpha$ preserves normalizability---more precisely, for any $\omega\in\Omega^{p,q}(G;V)$ there holds
\[
(j_{p-1}^k)^*R_\alpha\omega=R_\alpha(j_{p}^k)^*\omega,
\] 
for any $k\leq p-1$, where $j_{p}^k\colon G^{(p-1)}\ra G^{(p)}$ denotes the degeneracy map that inserts the unit into the $k$-th place (hence $j_p=j_p^p$). Hence, for any $\alpha_1,\dots,\alpha_p\in\Gamma(A)$,
\begin{align*}
  R_{\alpha_1}\dots R_{\alpha_p}\d{}^{\nabla^{s\circ\pr_p}}\omega&=R_{\alpha_1}\dots R_{\alpha_{p-1}}\d{}^{\nabla^{s\circ\pr_{p-1}}}R_{\alpha_p}\omega&\text{(since $j_p^*\omega=0$)}\\
  &=R_{\alpha_1}\dots R_{\alpha_{p-2}}\d{}^{\nabla^{s\circ\pr_{p-1}}}R_{\alpha_{p-2}}R_{\alpha_p}\omega &\text{(since $(j_p^{p-1})^*\omega=0$)}\\
  &\vdotswithin{=}\\
  &=\d{}^\nabla R_{\alpha_1}\dots R_{\alpha_p}\omega &\text{(since $(j_p^{1})^*\omega=0$)}
\end{align*}
From this the desired result for normalized forms follows.
% One should also note that at $p=1$, the same arguments show the following similar identities:
% \[
% R_\alpha=u^*\L^{\nabla^s}_{\cev \alpha},\quad R_\alpha \d{}^{\nabla^s}=\d{}^\nabla R_\alpha.
% \]

To see that equation \eqref{eq:ve_dnabla_0} was actually enough to show, we again use the trick from Remark \ref{rem:action_trick}. More precisely, we now note that if $G\rra M$ acts on a surjective submersion $\mu\colon P\ra M$, the action differentiates to an action  $\Gamma(A)\ra \vf(P)$ of its Lie algebroid, and the obtained action algebroid $A\ltimes P$ is now just the algebroid of the action groupoid $G\ltimes P$. Denoting the respective surjective submersive Lie groupoid and algebroid morphisms by $p_G\colon G\ltimes P\ra G$ and $p_A\colon A\ltimes P\ra A$, there now clearly holds 
\[R_{1\otimes\alpha}(p_G)^*=(p_G)^*R_\alpha,\quad J_{1\otimes \alpha}(p_G)^*=(p_G)^* J_\alpha.\]
Moreover, at level $p=0$, both pullbacks $(p_G)^*$ and $(p_A)^*$ restrict simply to $\mu^*$. Hence, the naturality of the van Est map follows:
% https://q.uiver.app/#q=WzAsNCxbMCwwLCJcXE9tZWdhXntwLHF9KEc7VikiXSxbMCwxLCJXXntwLHF9KEE7VikiXSxbMSwwLCJcXE9tZWdhXntwLHF9KEdcXGx0aW1lcyBQO1xcbXVeKlYpIl0sWzEsMSwiV157cCxxfShBXFxsdGltZXMgUDtcXG11XipWKSJdLFswLDEsIlxcdmUiLDJdLFsyLDMsIlxcdmUiXSxbMCwyLCJwXioiXSxbMSwzLCJwXioiLDJdXQ==
\[\begin{tikzcd}[row sep=large]
	{\Omega^{p,q}(G;V)} & {\Omega^{p,q}(G\ltimes P;\mu^*V)} \\
	{W^{p,q}(A;V)} & {W^{p,q}(A\ltimes P;\mu^*V)}
	\arrow["{(p_G)^*}", from=1-1, to=1-2]
	\arrow["\ve"', from=1-1, to=2-1]
	\arrow["\ve", from=1-2, to=2-2]
	\arrow["{(p_A)^*}"', from=2-1, to=2-2]
\end{tikzcd}\]
This concludes the proof.
%For the correction terms, we use the formulae above to obtain
%\[
%J_\alpha\d{}^{\nabla^{s\circ\pr_p}}=R_\alpha-\d{}^{\nabla^{s\circ\pr_{p-1}}}J_\alpha
%\]
%and establish the required identity by taking care of the signs.
\end{proof}

\subsection{Obstruction to existence of invariant connections}
\label{sec:obstruction_invariance}
Given a representation $V$ of either a Lie groupoid or an algebroid, we now construct the cohomological class which controls the existence of invariant connections on $V$. Let
\[\A_\inv(G;V)\quad\text{and}\quad\A_\inv(A;V)\]
denote the sets of all invariant linear connections on $V$ for the global and the infinitesimal case, respectively. Let us start with the global case.

\subsubsection{Global case}
We first observe that the form \eqref{eq:invariance_form_G} controlling the $G$-invariance of a linear connection $\nabla$ on a representation $V$ can be seen as an $\End(V)$-valued form,
\begin{align}
  \label{eq:invariance_form_theta}
  \Theta\in\Omega^1(G;s^*{\End V}),\quad \Theta(X)\xi=\phi\big(\nabla^t_X(\phi^{-1}\xi)\big)-\nabla_X^s\xi,
\end{align}
for any $\xi\in\Gamma(s^*V)$. It is called the \textit{$G$-invariance form} of the connection $\nabla$. Without further ado:
\begin{theorem}
  \label{thm:obstruction_invariance_G}
  Let $V\ra M$ be a representation of a Lie groupoid $G\rra M$. 
  \begin{enumerate}[label={(\roman*)}]
    \item Given any connection $\nabla$ on $V\ra M$, its $G$-invariance form \eqref{eq:invariance_form_theta} is multiplicative,
    \[\Theta\in\Omega_m^1(G;\End V).\]
    \item The set $\A_\inv(G;V)$ is an affine space over  invariant, endomorphism-valued 1-forms \[\Omega^1_\inv(M;\End V).\]
    \item The cohomological class of the $G$-invariance form $\Theta$ of any connection $\nabla$ on $V$ is independent of the choice of connection $\nabla$. We denote it by
    \begin{align*}
      \obs_{\A_\inv(G;V)}\coloneq [\Theta]\in H^{1,1}(G;\End V).
    \end{align*}
    \item A $G$-invariant connection on $V$ exists if and only if $\obs_{\A_\inv(G;V)}=0$.
  \end{enumerate}
\end{theorem}
\begin{proof}
  The proof of (i) essentially consists of understanding what multiplicativity means for $\End(V)$-valued forms. By definition of the induced representation on $\End(V)$, we need to check that for any $X\in\vf(G^{(2)})$ and $\xi\in\Gamma((s\circ\pr_2)^* V)$ there holds
  \begin{align*}
    (m^*\Theta)(X)\cdot\xi=(\pr_2^*\Theta)(X)\cdot\xi+\Phi\big((\pr_1^*\Theta)(X)\cdot(\Phi^{-1}\xi)\big),
  \end{align*}
  where $\Phi\colon(s\circ \pr_1)^*V\ra (s\circ\pr_2)^*V$ is given by $(g,h,\xi)\mapsto (g,h,h^{-1}\cdot\xi)$. By definition of $\Theta$, the terms on the right-hand side above equal
  \begin{align}
    \label{eq:intermed_twoterms_theta}
    (\pr_2^*\Theta)(X)\cdot\xi&=\Phi_2\big(\nabla^{t\circ\pr_2}_X(\Phi_2^{-1}\xi)\big)-\nabla^{s\circ\pr_2}_X \xi\\
    \Phi\big((\pr_1^*\Theta)(X)\cdot(\Phi^{-1}\xi)\big)&=\Phi_2\Phi_1\big(\nabla^{t\circ\pr_1}_X(\Phi_1^{-1}\Phi_2^{-1}\xi)\big)-\Phi_2\big(\nabla^{s\circ\pr_1}_X(\Phi_2^{-1}\xi)\big),\label{eq:intermed_twoterms_theta2}
  \end{align}
  where we are denoting by $\Phi_i$ ($i=1,2$) the vector bundle morphisms covering $\id_{G^{(2)}}$, determined by the identities $\Phi_i (\pr_i^*\eta)=\pr_i^*\phi(\eta)$ for any $\eta\in\Gamma(t^* V)$. They are given by:
% https://q.uiver.app/#q=WzAsMTEsWzAsMCwiKHRcXGNpcmNcXHByXzEpXipWIl0sWzEsMCwiKHNcXGNpcmNcXHByXzEpXipWIl0sWzIsMCwiKHNcXGNpcmNcXHByXzIpXipWIl0sWzAsMSwiKGcsaCxcXHhpKSJdLFsxLDEsIihnLGgsZ157LTF9XFxjZG90XFx4aSkiXSxbMSwyLCIoZyxoLFxceGkpIl0sWzIsMiwiKGcsaCxoXnstMX1cXGNkb3RcXHhpKSJdLFs0LDAsIlxccHJfaV4qIChzXipWKSJdLFs0LDIsInNeKlYiXSxbMywwLCJcXHByX2leKiAodF4qVikiXSxbMywyLCJ0XipWIl0sWzAsMSwiXFxQaGlfMSJdLFsxLDIsIlxcUGhpXzIiXSxbMyw0LCIiLDAseyJzdHlsZSI6eyJ0YWlsIjp7Im5hbWUiOiJtYXBzIHRvIn19fV0sWzUsNiwiIiwwLHsic3R5bGUiOnsidGFpbCI6eyJuYW1lIjoibWFwcyB0byJ9fX1dLFs3LDhdLFs5LDcsIlxcUGhpX2kiXSxbMTAsOCwiXFxwaGkiLDJdLFs5LDEwXV0=
\[\begin{tikzcd}[row sep=-2pt]
	{(t\circ\pr_1)^*V} & {(s\circ\pr_1)^*V} & {(s\circ\pr_2)^*V} & \hspace{-1.5em}{\pr_i^* (t^*V)} & {\pr_i^* (s^*V)} \\
	{(g,h,\xi_{t(g)})} & {(g,h,g^{-1}\cdot\xi_{t(g)})} \\
	& {(g,h,\xi_{s(g)})} & {(g,h,h^{-1}\cdot \xi_{s(g)})} & {t^*V} & {s^*V}
	\arrow["{\Phi_1}", from=1-1, to=1-2]
	\arrow["{\Phi_2}", from=1-2, to=1-3]
	\arrow["{\Phi_i}", from=1-4, to=1-5]
	\arrow[from=1-4, to=3-4]
	\arrow[from=1-5, to=3-5]
	\arrow[maps to, from=2-1, to=2-2]
	\arrow[maps to, from=3-2, to=3-3]
	\arrow["\phi"', from=3-4, to=3-5]
\end{tikzcd}\]
Identities  \eqref{eq:intermed_twoterms_theta} and \eqref{eq:intermed_twoterms_theta2} are easily shown on pullback sections, i.e., those of the form $\xi=\pr_2^*\eta$ and $\Phi^{-1}\xi=\pr_1^*\eta$ for $\eta\in\Gamma(s^*V)$, respectively, by observing there holds $\Phi=\Phi_2$.  Notice that the first term of \eqref{eq:intermed_twoterms_theta} and the second term of \eqref{eq:intermed_twoterms_theta2} cancel out since $s\circ\pr_1=t\circ \pr_2$, hence we just need to check $\Theta$ satisfies
\begin{align*}
  (m^*\Theta)(X)\cdot\xi=\Phi_2\Phi_1\big(\nabla^{t\circ\pr_1}_X(\Phi_1^{-1}\Phi_2^{-1}\xi)\big)-\nabla^{s\circ\pr_2}_X \xi,
\end{align*}
This is clear since $t\circ m=t\circ \pr_1$, $s\circ m=s\circ \pr_2$ and the map $\Phi_m\colon (t\circ\pr_1)^*V\ra (s\circ\pr_2)^*V$ determined by $\Phi_m m^*=m^*\phi$ clearly equals $\Phi_m=\Phi_2\Phi_1$, hence (i) is proved.

The points (ii), (iii) and (iv) are direct consequences of the  following observation: if $\tilde\nabla$ and $\nabla$ are two connections on $V$ with respective invariance forms $\tilde\Theta$ and $\Theta$, then
\[
\Tilde\Theta-\Theta=-\delta^0\gamma
\]
where $\gamma=\tilde\nabla-\nabla\in\Omega^1(M;\End V)$. This easily follows from the identity $s^*\gamma=\tilde\nabla^s-\nabla^s$ and a similar one for $t^*\gamma$, concluding the proof.
\end{proof}

\subsubsection{Infinitesimal case}
We now prove the infinitesimal analogue of Theorem \ref{thm:obstruction_invariance_G}.
\begin{theorem}
  \label{thm:obstruction_invariance}
  Let $V\ra M$ be a representation of a Lie algebroid $A\Ra M$. 
  \begin{enumerate}[label={(\roman*)}]
    \item Given any connection $\nabla$ on $V\ra M$, its $A$-invariance form \eqref{eq:invariance_form} is multiplicative: \[(T,\theta)\in\Omega^1_{im}(A; \End V).\]
    \item The set $\A_\inv(A;V)$ is an affine space over  invariant, endomorphism-valued 1-forms \[\Omega^1_\inv(M;\End V).\]
    \item The cohomological class of the $A$-invariance form $(T,\theta)$ of any connection $\nabla$ on $V$ is independent of the choice of connection $\nabla$. We denote it by
    \begin{align*}
      \obs_{\A_\inv(A;V)}\coloneq [(T,\theta)]\in H^{1,1}(A;\End V).
    \end{align*}
    \item An $A$-invariant connection on $V$ exists if and only if $\obs_{\A_\inv(A;V)}=0$.
  \end{enumerate}
\end{theorem}
\begin{proof}
  The proof of (i) consists of a straightforward computation verifying the condition \eqref{eq:c1} for the map $T$, which we will now carry out. Writing out the map $T$ in full, it has a formula similar to the usual curvature tensor,
  \begin{align}
    \label{eq:T_explicit}
    T(\alpha)(X)\xi=\nabla_X\nabla^A_\alpha\xi-\nabla^A_\alpha\nabla_X\xi+\nabla_{[\rho\alpha,X]}\xi,
  \end{align}
  for any $\alpha\in\Gamma(A)$, $X\in \vf(M)$ and $\xi\in\Gamma(V)$. If $\beta\in\Gamma(A)$ is another section, then
  \begin{align*}
    \L^A_\alpha (T\beta)(X)\xi&=\nabla^A_\alpha(T(\beta)(X)\xi)-T(\beta)(X)\nabla^A_\alpha\xi-T(\beta)([\rho\alpha,X])\xi\\
    &=\nabla^A_\alpha(\nabla_X\nabla^A_\beta\xi-\nabla^A_\beta\nabla_X\xi+\nabla_{[\rho\beta,X]}\xi)\\
    &-(\nabla_X\nabla^A_\beta-\nabla^A_\beta\nabla_X+\nabla_{[\rho\beta,X]})\nabla^A_\alpha\xi\\
    &-(\nabla_{[\rho\alpha,X]}\nabla^A_\beta\xi-\nabla^A_\beta\nabla_{[\rho\alpha,X]}\xi+\nabla_{[\rho\beta,[\rho\alpha,X]]}\xi)
  \end{align*}
  In this expression, the following terms come in pairs with respect to interchanging $\alpha$ and $\beta$: first and fifth, third and eighth, sixth and seventh. Hence, subtracting $\L^A_\beta (T\alpha)(X)\xi$ from this expression kills these terms. We are left with  
  \begin{align*}
    \L^A_\alpha (T\beta)(X)\xi-\L^A_\beta (T\alpha)(X)\xi=\nabla_X \nabla^A_{[\alpha,\beta]}\xi-\nabla^A_{[\alpha,\beta]}\nabla_X\xi+\nabla_{[[\rho\alpha,\rho\beta],X]}\xi=T[\alpha,\beta](X)\xi,
  \end{align*}
  where we have used the flatness $\smash{\nabla^A_{[\alpha,\beta]}=[\nabla_\alpha^A,\nabla_\beta^A]}$ of the representation $\nabla^A$ and the Jacobi identity on $\vf(M)$. This proves (i).

  The remaining points (ii), (iii) and (iv) are direct consequences of the following more general observation. If $\tilde \nabla$ and $\nabla$ are two connections on $V$ with respective invariance forms $(\tilde T,\tilde \theta)$ and $(T,\theta)$, then there holds
  \begin{align}
    (\tilde T,\tilde \theta)-(T,\theta) = -\delta^0\gamma,
  \end{align}
  where $\gamma=\tilde\nabla-\nabla\in \Omega^1(M;\End V)$. To see this, we first note that on the level of symbols, we clearly have $\tilde\theta(\alpha)-\theta(\alpha)=-\gamma(\rho\alpha)$. For the leading terms, we insert $\nabla+\gamma$ in place of connection $\nabla$ in the equation \eqref{eq:T_explicit} to obtain
  \begin{align*}
    \tilde T(\alpha)(X)\xi&=T(\alpha)(X)\xi+\gamma(X)\nabla^A_\alpha\xi-\nabla^A_\alpha(\gamma(X)\xi)+\gamma([\rho\alpha,X])\xi\\
    &=T(\alpha)(X)\xi-(\L^A_\alpha\gamma)(X)\xi,
  \end{align*}
  that is, $\tilde T(\alpha)=T(\alpha)-\L^A_\alpha \gamma$. This concludes the proof.
\end{proof}
Finally, we show that the van Est map relates the global and infinitesimal invariance tensor.
\begin{proposition}
  \label{prop:g_inv_a_inv}
  Let $A$ be the Lie algebroid of $G$. For any connection $\nabla$ on a representation $V$ of $G$, the van Est map sends the $G$-invariance form of $\nabla$ to its $A$-invariance form:
  \begin{align}
    \label{eq:ve_Theta}
    \ve(\Theta)=(T,\theta).
  \end{align}
  In particular, $G$-invariance implies $A$-invariance of $\nabla$, and if $G$ is source-connected, the converse also holds. Moreover, the van Est map relates the obstruction classes:
    \[\ve(\obs_{\A_{\inv}(G;V)})=\obs_{\A_{\inv}(A;V)}.\]
\end{proposition}
\begin{proof}
  The second part follows from the first part and Theorem \ref{thm:obstruction_invariance_G} (i), since the restriction of the van Est map to multiplicative forms is injective when $G$ is source-connected. That the van Est map relates the obstruction classes is also a direct consequence of \eqref{eq:ve_Theta}.

  Let us prove the identity \eqref{eq:ve_Theta}. We begin by proving that the symbols coincide: for any $\alpha\in\Gamma(A)$ and $\xi\in\Gamma(V)$, there holds
  \begin{align*}
    u^*(\Theta(\alpha^L))\cdot\xi=u^*(\Theta(\alpha^L)\cdot s^*\xi)=u^*\big(\phi\big(\nabla^t_{\alpha^L}(\phi^{-1}s^*\xi)\big)\big)-u^*(\nabla^s_{\alpha^L}s^*\xi).
  \end{align*}
  The second term clearly equals $\nabla_{\rho\alpha}\xi$ since $s_*\alpha^L=\rho\alpha$. Using parallel transport, we can write the first term at any $g\in G$ as 
  \begin{align*}
    \phi\big(\nabla^t_{\alpha^L}(\phi^{-1}s^*\xi)\big)_g &= g^{-1}\cdot\deriv\lambda 0 \underbrace{\tau(\gamma^{\alpha^L}_g)^{\nabla^t}_{\lambda,0}}_{\id_{V_{t(g)}}}\Big(\underbrace{\phi^{\alpha^L}_\lambda(g)}_{\mathclap{g\phi^{\alpha^L}_\lambda(1_{s(g)})}}\cdot \xi_{\phi^{\rho\alpha}_\lambda(s(g))}\Big)=\deriv\lambda 0 \phi^{\alpha^L}_\lambda(1_{s(g)})\cdot\xi_{\phi^{\rho\alpha}_\lambda(s(g))},
  \end{align*}
  which equals $(s^*\nabla^A_\alpha\xi)_g$. We have used here that the action is linear and $t_*\alpha^L=0$. By tensoriality of $\Theta$, we have shown $\smash{\Theta(\alpha^L)=s^*\theta(\alpha)}$ and thus proved that the symbols coincide.

  For the leading terms, we need to show $R_\alpha\Theta=T(\alpha)$. To this end, we first use multiplicativity of $\Theta$ with equation \eqref{eq:R_jL} to express the left-hand side as $R_\alpha\Theta=u^*\L_{\alpha^L}^{\nabla^s}\Theta$ where $\nabla^s$ is the induced connection on $s^*\End(V)$. That is, for any $X\in\vf(M)$ and $\xi\in\Gamma(V)$,
  \begin{align*}
    (R_\alpha\Theta)(X)\cdot\xi&=u^*(\L^{\nabla^s}_{\alpha^L}\Theta)(X)\cdot\xi=u^*\big((\L^{\nabla^s}_{\alpha^L}\Theta)(Y)\cdot s^*\xi\big)\\
    &=u^*\big(\nabla^s_{\alpha^L}(\Theta(Y)\cdot s^*\xi)-\Theta(Y)\cdot \nabla^s_{\alpha^L}s^*\xi-\Theta[\alpha^L,Y]\cdot s^*\xi\big)
  \end{align*}
  where $Y\in \vf(G)$ extends $u_*X$ on $u(M)$ and can be chosen $s$-projectable to $s_*Y=X$. Expanding the three terms on the right-hand side and denoting by $\smash{\overline\nabla^t=\phi\nabla^t\phi^{-1}}$ the pullback connection along $t$ on the bundle $s^*V\ra G$, we get
  \begin{align*}
    \nabla^s_{\alpha^L}\big(\Theta(Y)\cdot s^*\xi\big)&=\nabla^s_{\alpha^L}\overline\nabla^t_Ys^*\xi-s^*\nabla_{\rho\alpha}\nabla_X\xi,\\
    \Theta(Y)\cdot \nabla^s_{\alpha^L}s^*\xi&=\overline\nabla^t_Y\nabla^s_{\alpha^L}s^*\xi-s^*\nabla_X\nabla_{\rho\alpha}\xi,\\
    \Theta[\alpha^L,Y]\cdot s^*\xi&=\overline\nabla^t_{[\alpha^L,Y]}s^*\xi-s^*\nabla_{[\rho\alpha,X]}\xi,
  \end{align*}
  where we have observed that $s_*[\alpha^L,Y]=[\rho\alpha,X]$. Hence, we obtain
  \[
    (R_\alpha\Theta)(X)\cdot\xi = u^*\big(\underbrace{\big(\nabla^s_{\alpha^L}\overline\nabla^t_Y-\overline\nabla^t_Y\nabla^s_{\alpha^L}-\overline\nabla^t_{[\alpha^L,Y]}\big)s^*\xi}_{\eqqcolon S(\alpha)(Y)\cdot s^*\xi}\big)-R^\nabla(\rho\alpha,X)\cdot \xi.
  \]
  We observe that the expression denoted by $S$ is tensorial in both $Y$ and $s^*\xi$ (but not in $\alpha$), so it defines a map $S\colon\Gamma(A)\ra \Omega^1(G;s^*\End V)$. Moreover, we can express it as
  \[
  S(\alpha)(Y)\cdot s^*\xi=R^{\overline \nabla^t}(\alpha^L,Y)\cdot s^*\xi+\overline \nabla^t_Y\big(\Theta(\alpha^L)\cdot s^*\xi \big)-\Theta(\alpha^L)\cdot\overline\nabla^t_Y s^*\xi.
  \]
  Since $R^{\overline\nabla^t}=\phi R^{\nabla^t}\phi^{-1}$ and $R^{\nabla^t}=t^* R^\nabla$, the first term vanishes by $t_*\alpha^L=0$. Hence, it is enough to show that 
  \[
    u^*\Big(\overline \nabla^t_Y\big(\Theta(\alpha^L)\cdot s^*\xi \big)-\Theta(\alpha^L)\cdot\overline\nabla^t_Y s^*\xi\Big)=\nabla_X(\theta(\alpha)\cdot\xi)-\theta(\alpha)\cdot\nabla_X\xi,
  \]
  but this is clear since $u^*\Theta(\alpha^L)=\theta(\alpha)$ and $u^*\overline\nabla^t_Y=\nabla_X$ since $Y$ extends $u_*X$.
\end{proof}

\section{Bott--Shulman--Stasheff complex with values in a bundle of ideals}
\label{sec:bss_boi}
In the remainder of this article, we will consider a special class of representations: subrepresentations of the adjoint representation $\Ad\colon G\curvearrowright \ker\rho$. Note that the latter is only a representation in the set-theoretic sense unless $\ker\rho$ is a vector bundle (that is, unless $G$ is regular), so we instead phrase this in the following way.
\begin{definition}
\label{defn:boi}
On a Lie groupoid $G\rra M$, a vector bundle $\frak k\subset \ker \rho$ is called a \textit{bundle of ideals}, if for any $g\in G$ the map
\[\Ad_g\coloneqq \d(C_g)_{1_{s(g)}}\colon \frak \ker\rho_{s(g)}\ra \frak \ker\rho_{t(g)}\] 
restricts to a map $\frak k_{s(g)}\ra \frak k_{t(g)}$, where $C_g\colon G^{s(g)}_{s(g)}\ra G^{t(g)}_{t(g)}$ denotes the conjugation map $k\mapsto gkg^{-1}$. 

On a Lie algebroid $A\Ra M$, a vector bundle $\frak k\subset \ker \rho$ is said to be a \textit{bundle of ideals} if for any $\alpha\in \Gamma(A)$ and $\xi\in\Gamma(\frak k)$ there holds
\begin{align}
  \label{eq:boi_inf}
  [\alpha,\xi]\in\Gamma(\frak k).
\end{align}
\end{definition}
\begin{remark}
  The terminology comes from \cite{mec}, however, it is not standard. In \cite{ideals}, a vector bundle $\frak k\subset\ker\rho$ satisfying \eqref{eq:boi_inf} is instead called a \textit{naïve ideal} of a Lie algebroid.
\end{remark}
If $A$ is the Lie algebroid of $G$, then 
\[[\alpha,\xi]_x=\deriv\lambda0\Ad_{\phi^{\alpha^L}_\lambda(1_x)}\big(\xi_{\phi^{\rho\alpha}_\lambda(x)}\big),\]
 holds at every $x\in M$, so that if $\frak k$ is a bundle of ideals on $G$, it is also a bundle of ideals on $A$. In other words, the map $(\alpha,\xi)\mapsto [\alpha,\xi]$ is just the infinitesimal version of the adjoint representation of $G\rra M$ on $\frak k$, i.e., a flat $A$-connection on $\frak k\ra M$, which we will also denote by $\nabla^A_\alpha\xi=[\alpha,\xi]$. 
 \begin{example}
  \label{ex:boi_groupoid_morphism}
  The main class of examples of bundles of ideals on Lie groupoids comes from surjective submersive groupoid morphisms $\Phi\colon G\ra H$ covering the identity, where one takes $\frak k$ to be the kernel of the associated Lie algebroid morphism, $\frak k=\ker\d\Phi|_M$. Not all bundles of ideals arise in this way---given a bundle of ideals $\frak k$ on $G\rra M$, we define its \textit{smearing} as the distribution $K\subset TG$,
  \begin{align}
    \label{eq:smearing}
    K_g\coloneqq \d (L_g)_{1_{s(g)}}(\frak k_{s(g)})=\d(R_g)_{1_{t(g)}}(\frak k_{t(g)}).
  \end{align}
  It is easy to see $K$ is involutive. If the corresponding foliation $\F(K)$ on $G$ is simple, the natural projection $\Phi\colon G\ra G/\F(K)$ is a surjective submersive Lie groupoid morphism with $\ker\d\Phi|_M=\frak k$. On Lie algebroids, however, every bundle of ideals $\frak k$ is clearly the kernel of the fibrewise surjective Lie algebroid morphism $\phi\colon A\ra A/\frak k$ that covers the identity.
 \end{example}

\subsection{Multiplicative Ehresmann connections}
In the study of bundles of ideals, the notion of a multiplicative connection offers a richer framework than that of an invariant connection. To lay the groundwork, we begin by recalling the definition and highlighting some fundamental properties.
\begin{definition}
Let $\frak k$ be a bundle of ideals on a Lie groupoid $G\rra M$. A \textit{multiplicative Ehresmann connection} for $\frak k$ is a distribution $E\subset TG$ which is also a wide Lie subgroupoid of $TG\rra TM$, satisfying
\[
TG=E\oplus K,
\]
where $K$ is the smearing of $\frak k$, defined by \eqref{eq:smearing}.
\end{definition}
\begin{example}
  If $\frak k$ comes from a surjective submersive Lie groupoid morphism $\Phi\colon G\ra H$, there holds $K=\ker\d\Phi$, hence $E$ is an Ehresmann connection in the usual sense, with the additional property of being multiplicative.
 \end{example}
\iffalse
In other words, a multiplicative Ehresmann connection is a splitting of the following short exact sequence, in the category of VB-groupoids (see \cite{vb_groupoids} for a basic reference on this notion):
% https://q.uiver.app/#q=WzAsNSxbMCwwLCIwIl0sWzEsMCwiSyJdLFsyLDAsIlRHIl0sWzMsMCwiVEcvSyJdLFs0LDAsIjAiXSxbMCwxXSxbMSwyXSxbMiwzXSxbMyw0XV0=
\[\begin{tikzcd}[column sep=small]
	0 & K & TG & {TG/K} & 0.
	\arrow[from=1-1, to=1-2]
	\arrow[from=1-2, to=1-3]
	\arrow[from=1-3, to=1-4]
	\arrow[from=1-4, to=1-5]
\end{tikzcd}\]
In fact, ...
\fi

There are several equivalent descriptions of multiplicative Ehresmann connections \cite{mec}*{Proposition 2.8}, among which we particularly favor the description with differential forms. Given a multiplicative Ehresmann connection $E\subset TG$, we can define a differential form $\omega \in \Omega^1(G;s^*\frak k)$ on $G$ with values in the pullback bundle $s^*\frak k\ra G$, 
\begin{align}
\label{eq:omega}
\omega(X)=\d(L_{g^{-1}})_g(v(X))\in \frak k_{s(g)},
\end{align}
for any $X\in T_gG$, where $v(X)$ is the so-called \textit{vertical} component in the unique decomposition $X=h(X)+v(X)\in E_g\oplus K_g=T_gG$. Here, the vector $h(X)$ is called the \textit{horizontal} component of $X$. The differential form just defined favors two characteristic properties \cite{mec}*{Proposition 2.8}.
\begin{proposition}
Let $G\rra M$ be a Lie groupoid. Given a multiplicative Ehresmann connection $E\subset TG$, the corresponding 1-form $\omega\in \Omega^1(G;s^*\frak k)$ satisfies:

\begin{enumerate}[label={(\roman*)}]
\item $\omega$ is \textit{multiplicative}: for any composable pair $(X,Y)\in T_{(g,h)}G^{(2)}$, there holds
\begin{align*}
%\label{eq:mult}
\omega_{gh}(\d m(X,Y))=\Ad_{h^{-1}}\circ \omega_g(X) + \omega_h(Y).
\end{align*}
\item $\omega$ restricts to identity on $\frak k$: 
\[\omega|_{\frak k}=\id_{\frak k}.\]
\end{enumerate}
Conversely, given such a form $\omega$, $E=\ker \omega$ defines a multiplicative Ehresmann connection.
\end{proposition}
\iffalse
\begin{proof}
Property (ii) is clear. To show property (i) holds, take $(v,w)\in T_{(g,h)}G^{(2)}$ and plug into $\d m$ the decomposed expressions $v=v^E+v^K$ and $w=w^E+w^K$:
\[
\d m_{(g,h)}(v,w)=\d m_{(g,h)}(v^E,w^E)+\d(R_h)_g(v^K)+\d(L_g)_h(w^K).
\]
Since $E\subset TG$ is a subgroupoid, the first term is contained in $E_{gh}$ and hence vanishes when plugged into $\omega$. So we obtain
\begin{align*}
(m^*\omega)_{(g,h)}(v,w)&=\omega_{gh}(\d(R_h)_g(v^K)+\d (L_g)_h(w^K))\\
&=\d (L_{h^{-1}g^{-1}})_{gh}(\d(R_h)_g(v^K)+\d (L_g)_h(w^K))\\
&=\d (C_{h^{-1}})_{1_{s(g)}}\circ \d(L_{g^{-1}})_g(v^K)+\d(L_{h^{-1}})_h(w^K)
\\
&=\Ad_{h^{-1}}\circ (\pr_1{}^*\omega)_{(g,h)}(v,w)+(\pr_2{}^*\omega)_{(g,h)}(v,w).
\end{align*}
The converse statement is left to the reader.
\end{proof}
\fi
\begin{remark}
  Another way of expressing the differential form $\omega$ is by means of the \textit{Maurer--Cartan form} on $G$; this is a vector bundle isomorphism, defined on $X\in \ker\d t_g$ by 
  $
  \Theta_{MC}(X)=\d{(L_{g^{-1}})}_g(X).
  $
  \[\begin{tikzcd}[column sep=small]
	{\ker\d t} && {s^* A} \\
	& G
	\arrow["{\Theta_{MC}}", from=1-1, to=1-3]
	\arrow[from=1-1, to=2-2]
	\arrow[from=1-3, to=2-2]
\end{tikzcd}\]
  Given a multiplicative Ehresmann connection $E$, its connection form $\omega$ is defined in terms of the Maurer--Cartan form simply by $\omega=\Theta_{MC}\circ v$. Moreover, we note that by multiplicativity, property (ii) above is equivalent to $\omega|_K=\Theta_{MC}|_K$, or in other words, $\omega(\xi^L)=s^*\xi$ for any section $\xi\in\Gamma(\frak k)$.
  % https://q.uiver.app/#q=WzAsMyxbMCwwLCJcXGtlclxcZCB0Il0sWzIsMCwic14qIEEiXSxbMSwxLCJHIl0sWzAsMSwiXFxUaGV0YV97TUN9Il0sWzAsMl0sWzEsMl1d
\end{remark}
We will denote the set of forms satisfying the properties from the last proposition as 
\begin{align*}
  \A(G;\frak k)=\set{\omega\in \Omega^1_m(G;\frak k)\given \omega|_{\frak k}=\id_{\frak k}},
\end{align*}
and sometimes simply call them  \textit{multiplicative connections} on $G$. %Similarly to the case of connections on principal bundles, we have the following simple fact, the proof of which is left for the reader. 
The following observation will be needed in the following section: a multiplicative connection defines a distribution $E^s\subset\ker\d s\subset TG$, 
\[
E^s=E\cap \ker\d s.
\]
Clearly, this distribution satisfies the following two properties:
\begin{align}
  \ker \d s&= E^s\oplus K,\\
  \d(R_h)_g(E^s_g)&=E^s_{gh},\quad\text{for any }(g,h)\in \comp G.\label{eq:right_inv}
\end{align}
The second property will be referred to as \textit{right-invariance} of $E^s$. In terms of the connection 1-form: \[(R_h)^*(\omega|_{\ker\d s_g})=\Ad_{h^{-1}}\circ\omega|_{\ker\d s_{gh}},\] for any composable pair 	$(g,h)\in\comp G$. 
\begin{remark}
  In the particular case when $G$ is regular and $\frak k=\ker \rho$, these two properties imply that for any $x\in M$, $E^s|_{G_x}$ defines a principal connection on the principal $G_x^x$-bundle $t\colon G_x\ra\O_x$. If $G$ is the gauge groupoid of a principal bundle $P\ra M$, we recover the classical notion of a principal connection on $P$.
\end{remark}

\subsection{Induced connection on bundle of ideals}
\label{sec:conn}
%Our next task is to generalise the notion of curvature of a connection to the Lie groupoid setting. We again follow \cite{mec} to do so, with some modifications that we will now explain.

In this section, we show that any multiplicative Ehresmann connection $\omega$ for a bundle of ideals $\frak k$, induces a linear connection $\nabla$ on $\frak k\rightarrow M$. Later, in \sec\ref{sec:mec_inf}, we will see that $\nabla$  is actually a part of the infinitesimal data contained within $\omega$. In what follows, $\nabla$ is constructed using a global approach, in terms of the given Lie groupoid $G\rra M$, since this will turn out to be useful for the subsequent sections. 

To begin, let us denote by $v\colon TG\ra K$ and $h\colon TG\ra E$ the vertical and horizontal projection (both are morphisms of vector bundles over $G$) with respect to a given multiplicative Ehresmann connection that complements $\frak k$. 
\begin{proposition}
\label{prop:conn}
Let $G\rra M$ be a Lie groupoid with a multiplicative Ehresmann connection $\omega\in\A(G;\frak k)$. Let the map $\nabla\colon \vf(M)\times \Gamma(\frak k)\ra  \Gamma(\frak k)$ be given by:
\begin{align}
\label{eq:nabla}
\nabla_X\xi=v[h(Y),\xi^L]|_M
\end{align}
where $Y\in \vf(G)$ is any lift of $X$ along the source map, i.e., $s_*Y=X$, and the bracket denotes the Lie bracket on $TG$. %, and $\xi^L$ denotes the left-invariant extension of $\xi\in\Gamma(\frak k)$
The following holds.
\begin{enumerate}[label={(\roman*)}]
  \item $\nabla$ defines a linear connection on $\frak k$.
  \item The left-invariant extension of $\nabla_X\xi$ reads $(\nabla_X\xi)^L=v[h(Y),\xi^L]$.
  \item $\nabla$ preserves the Lie bracket on the bundle of ideals:
\begin{align}
\label{eq:nabla_preserves_bracket}
  \nabla_X[\xi,\eta]_{\frak k}=[\nabla_X\xi,\eta]_{\frak k}+[\xi,\nabla_X\eta]_{\frak k}.
\end{align}
In particular, $\frak k$ is a locally trivial Lie algebra bundle.
\end{enumerate}
\end{proposition}
\begin{remark}
The defining equation \eqref{eq:nabla} above is similar to \cite{mec}*{equation 2.4}, where $\nabla$ is obtained by first integrating $\frak k$ to a bundle of simply-connected Lie groups. In fact, our definition of $\nabla$ was motivated by the referenced equation, and our aim was to construct it entirely in terms of the groupoid $G$. As innocuous as it may seem, the proposition above will turn out to be of paramount importance in the next section.
\end{remark}
\begin{proof}
First, we prove that our definition \eqref{eq:nabla} is independent of the choice of the lift $Y\in\vf(G)$ of $X$ along $s\colon G\ra M$. Pick another such lift $\tilde Y\in \vf(G)$, so $Y-\tilde Y\in\Gamma(\ker \d s)$, and hence $Z\coloneqq h(Y-\tilde Y)\in \Gamma(E^s)$. Note that the flow $\phi_\lambda^{\smash{\xi^L}}=R_{\exp(\lambda \xi)}$ of a left-invariant vector field $\smash{\xi^L}$ is given by the right translation along the bisection $\exp(\lambda \xi)$, hence by \eqref{eq:right_inv}, 
	\begin{align}
	\label{eq:principaltrick}
	\big((\phi_\lambda^{\xi^L})_*Z\big)_{g}\in E^s_{g},
	\end{align}
for all $\lambda\in\R$. This implies $\smash{[Z,\xi^L]=[h(Y),\xi^L]-[h(\tilde Y),\xi^L]}\in \Gamma(E^s)$. Since this difference is horizontal, the vertical parts of the two terms in the difference coincide, proving the expression $v[h(Y),\xi^L]$ is independent of the lift $Y$, and hence the same thus holds for $\nabla$. Let us now show $\nabla$ defines a connection. For $C^\infty(M)$-linearity in the first argument of the map $\nabla$, pick any $f\in C^\infty(M)$ and observe that if $Y$ is a lift of $X$, then $(f\circ s)Y$ is a lift of $fX$. Now use the Leibniz rule of the Lie bracket on $TG$ to obtain 
\[[h((f\circ s)Y),\xi^L]=(f\circ s)[h(Y),\xi^L]-\d(f\circ s)(\xi^L)h(Y).\]
The second term vanishes since $\frak k\subset \ker\d s$ and the first term gives us the wanted $C^\infty(M)$-linearity when its vertical part is evaluated at the units. 
For the Leibniz rule of $\nabla$, observe there holds $(f\xi)^L=(f\circ s)\xi^L$, so the Leibniz rule of the Lie bracket on $TG$ yields
\[
[h(Y),(f\circ s)\xi^L]=(f\circ s)[h(Y),\xi^L]+\d (f\circ s)(h(Y))\xi^L.
\]
Since $\d s_{1_x}(h(Y))=\d s_{1_x}(Y)=X_x$, vertically project and evaluate at the units to conclude.

To prove (ii), we must show that $v[h(Y),\xi^L]\in \Gamma(K)$ is left-invariant. For brevity, we will assume that $Y$ is a horizontal lift of $X$ along the source map, so $h(Y)=Y$. Our technique of proving (ii) utilizes bisections on $G$: fix $g\in G$, denote $s(g)=x$, and pick any local bisection $\sigma\colon U\ra G$ with $\sigma(x)=g$ such that\footnote{Such a choice of a local bisection is always possible, by a similar proof as in \cite{mackenzie}*{Proposition 1.4.9}.} 
\begin{align}
\label{eq:bisection_in_E}
  \im\d \sigma_{x}\subset E_g.
\end{align}
The left translation by $\sigma$, given as the diffeomorphism \[
L_\sigma\colon t^{-1}(U)\ra t^{-1}(t\circ\sigma(U))
,\quad L_\sigma(h)=\sigma(t(h))h,
\]
then preserves the splitting $T_hG=E_h\oplus K_h$ for any $h\in G^x$, i.e., its differential restricts to
\begin{align}
  \label{eq:L_sigma_E}
    \d(L_\sigma)_{h}\colon E_{h}\ra E_{gh}.
\end{align}
This follows from a simple computation: differentiate the equation $L_\sigma=m\circ(\sigma\circ t,\id)$ which defines  $L_\sigma$ and use it on an arbitrary horizontal vector $v\in E_h$ to obtain
\[\d(L_\sigma)_h(v)=\d m_{(g,h)}(\underbrace{\d\sigma_x (\d t_h (v))}_{\in E_g},v)\in E_{gh},
\]
where we have used \eqref{eq:bisection_in_E} and the assumption that $E$ is multiplicative. We now use the left translation $L_\sigma$ to locally write $(L_\sigma)_*[Y,\xi^L]=[(L_\sigma)_*Y,\xi^L]$ and thus 
\begin{align*}
	\d(L_\sigma)_{1_{x}}[Y,\xi^L]-[Y,\xi^L]_g&=[(L_\sigma)_*Y-Y,\xi^L]_g=\deriv\lambda0\big((\phi_\lambda^{\xi^L})_*((L_\sigma)_*Y-Y)\big)_g.
\end{align*}
For convenience, let us denote $k_\lambda\coloneqq \exp(-\lambda \xi)_x=\phi^{\xi^L}_{-\lambda}(1_x)\in G^x_x$. We have
$$
\big((\phi_\lambda^{\xi^L})_*((L_\sigma)_*Y-Y)\big)_g=\d(\phi_\lambda^{\xi^L})_{gk_\lambda}\big({\d(L_\sigma)_{k_\lambda}(Y)}-Y_{gk_\lambda}\big).
$$
By equation \eqref{eq:L_sigma_E}, the difference $\smash{\d(L_\sigma)_{k_\lambda}(Y)-Y_{gk_\lambda}}$ is horizontal, and since $\d s$ annihilates it by the fact that $s\circ L_\sigma=s$, it must be contained in $\smash{E^s_{gk_\lambda}}$ for all $\lambda\in\R$. But since the flow of $\xi^L$ is given by right translations, \eqref{eq:right_inv} now implies
$$
\d(L_\sigma)_{1_{x}}[Y,\xi^L]-[Y,\xi^L]_g\in E^s_g,
$$
and taking the vertical part of this expression yields
\[
v[Y,\xi^L]_g=v(\d(L_\sigma)_{1_{x}}[Y,\xi^L])=v(\d(L_\sigma)_{1_{x}}(h[Y,\xi^L]_{1_x}+v[Y,\xi^L]_{1_x}))=\d(L_g)_{1_{x}}(v[Y,\xi^L]_{1_x}),
\]
where we have used the equation \eqref{eq:L_sigma_E} on the third equality, together with the fact that $L_\sigma$ restricts on $t$-fibres to the usual left-translation. This proves (ii).

Finally, preservation of the Lie bracket in (iii) easily follows from the Jacobi identity of the Lie bracket on $TG$. That is, supposing $Y$ is a horizontal lift of $X$ along the source map,
\[
(\nabla_X[\xi,\eta]_{\frak k})^L=v[Y,[\xi,\eta]^L]=v[Y,[\xi^L,\eta^L]]=v[[Y,\xi^L],\eta^L]+v[\xi^L,[Y,\eta^L]]
\]
and now observe that in the first term on the right-hand side, $h[Y,\xi^L]$ is horizontal and $s$-projectable to zero, hence it is a section of $E^s$, so by \eqref{eq:right_inv} the bracket $[h[Y,\xi^L],\eta^L]$ is again a section of $E^s$ and vanishes when vertically projected. Now use involutivity of $K$.
\end{proof}

\begin{remark}
Instead of working with lifts along $s\colon G\ra M$ and left-invariant extensions, we can work with lifts along $t\colon G\ra M$ and right-invariant extensions. In fact, there holds
\[
(\nabla_X\xi)^R=v[h(W),\xi^R],
\]
for any lift $W\in \vf(G)$ of $X\in\vf(M)$ along $t\colon G\ra M$, and any $\xi\in \Gamma(\frak k)$. This follows from the following two facts: firstly, if $Y$ is a lift of $X$ along the source map, then $\inv_* Y$ is a lift of $X$ along the target map; secondly, for any $\xi\in \Gamma(\frak k)$, there holds $\xi^L=-\inv_*(\xi^R)$, which is a consequence of the equality $\d(\inv)|_{\frak k}=-\id_{\frak k}$ that comes from the theory of Lie groups. Hence:
\begin{align*}
v[h(W),\xi^R]&=v[\inv_*(h(Y)),-\inv_*(\xi^L)]=-v(\inv_*[h(Y),\xi^L])\\
&=-\inv_*(v[h(Y),\xi^L])=-\inv_*(\nabla_X\xi)^L=(\nabla_X\xi)^R,
\end{align*}
where we have observed that $\inv_*$ commutes with both $v$ and $h$.
\end{remark}

The linear connection $\nabla$ on $\frak k$ obtained in Proposition \ref{prop:conn} has a particularly nice expression when differentiating with respect to vectors tangent to the orbit foliation, by the following corollary.

\begin{corollary}
\label{cor:conn_orb}
Let $G\rra M$ be a Lie groupoid with a multiplicative Ehresmann connection $\omega\in\A(G;\frak k)$. For any $\alpha\in\Gamma(A)$ and $\xi\in\Gamma(\frak k)$, there holds
\begin{align}
  \nabla_{\rho(\alpha)} \xi=[h(\alpha),\xi],
\end{align}
where the bracket denotes the Lie bracket on the algebroid $A$.
\end{corollary}
\begin{proof}
The vector field $h(\alpha)^L\in\vf(G)$ is a horizontal lift of $\rho(\alpha)$ along the source map, hence we can use Proposition \ref{prop:conn} (ii), to write $(\nabla_{\rho(\alpha)}\xi)^L=v[h(\alpha)^L,\xi^L]=v([h(\alpha),\xi]^L)$. Containment $[\Gamma(A),\Gamma(\frak k)]\subset\Gamma(\frak k)$ ensures that $[h(\alpha),\xi]^L$ is already vertical.
\end{proof}
%\begin{remark}
%This corollary can be interpreted as follows: along the orbits, the linear connection $\nabla$ from Proposition \ref{prop:conn} coincides with the representation of $A$ on $\frak k$, modulo a horizontal lift from $T\F$ to $A$ -- compare this with definition \eqref{eq:leafwise_nabla}. This restriction of $\nabla$ to a leafwise connection will be denoted by $\nabla^\F$.
%\end{remark}

\subsection{Horizontal exterior covariant derivative}
\label{sec:hor_ext_cov_der}
We now show that any multiplicative Ehresmann connection $\omega$ gives rise to the horizontal exterior covariant derivative on $\frak k$-valued forms. Just like the operator $\d{}^\nabla$ for an invariant connection $\nabla$, this operator will turn out to be a cochain map. Before defining it, we observe that a bundle of ideals $\frak k$ determines an  intrinsic subcomplex of $(\Omega^{\bullet,q}(G;\frak k),\delta)$. 
\begin{definition}
\label{defn:horizontal}
A form $\alpha\in\Omega^{p,q}(G;\frak k)$ is said to be \textit{horizontal}, if it vanishes when evaluated on a $p$-tuple of composable vectors from $K$, i.e., $\iota_X\alpha=0$ for any $X\in K^{(p)}$, where
\[
  K^{(p)}= (\underbrace{K\times\dots\times K}_{p \text{ copies}})\cap T G^{(p)}.
\]
The set of such forms will be denoted by \[\Omega^{p,q}(G;\frak k)^\Hor=\Gamma(\Lambda^q(K^{(p)})^\circ\otimes  (s\circ \pr_p)^*\frak k).\]
This definition applies when $p\geq 1$; at level $p=0$, all forms are defined to be horizontal.
At any fixed $q$, we obtain a subcomplex $(\Omega^{\bullet,q}(G;\frak k)^\Hor,\delta)\subset(\Omega^{\bullet,q}(G;\frak k),\delta)$. This follows from Definition \eqref{eq:delta_0} since $K\subset\ker\d s\cap \ker\d t$ and Definition \eqref{eq:delta_l} since $K\rra 0_M$ is a subgroupoid of $TG\rra TM$. It is called the \textit{horizontal subcomplex} of $\Omega^{\bullet,q}(G;\frak k)$.
\end{definition}
A multiplicative Ehresmann connection $\omega\in\A(G;\frak k)$ induces a \textit{horizontal projection} of tangent vectors in $TG^{(p)}$, by projecting all the components. That is,
\[
h\colon TG^{(p)}\ra E^{(p)}=(\underbrace{E\times\dots\times E}_{p\text{ copies}})\cap TG^{(p)},\quad h(X_1,\dots,X_p)=(h(X_1),\dots,h(X_p)).
\]
In turn, we obtain a horizontal projection of differential forms, given by the precomposition with the horizontal projection $\smash{TG^{(p)}\stackrel h\ra E^{(p)}\hookrightarrow TG^{(p)}}$ in all arguments. We denote it by
\begin{align*}
&h^*\colon \Omega^{p,q}(G;\frak k)\ra \Omega^{p,q}(G;\frak k)^\Hor.
\end{align*}
By multiplicativity of $E=\ker\omega$, it is also clear from \eqref{eq:delta_l} that $h^*$ is a cochain map, 
\[h^*\delta=\delta h^*.\]
\begin{definition}
\label{defn:D}
Let $G\rra M$ be a Lie groupoid with a multiplicative Ehresmann connection $\omega\in\A(G;\frak k)$. The \textit{horizontal exterior covariant derivative} is defined as the map 
\[
\D{}^\omega=h^*\d{}^\nabla\colon \Omega^{p,q}(G;\frak k)\ra \Omega^{p,q+1}(G;\frak k),
\]
where $\nabla$ is the induced connection by $\omega$ on $\frak k$. Namely, for any $\vartheta\in\Omega^{p,q}(G;\frak k)$, 
\[
(\D{}^\omega\vartheta) (X_0,\dots,X_q)=(\d{}^{\nabla^{s\circ\pr_p}}\vartheta)(h(X_0),\dots,h(X_q)),
\]
for any given vector fields $X_i\in \vf(G^{(p)})$.
\end{definition}

The task at hand is to show that $\D{}^\omega$ is a cochain map; we will also see that in general, the connection $\nabla$ induced by $\omega$ is not $G$-invariant, so that $\d{}^\nabla$ is not a cochain map.\footnote{More precisely, if $\nabla$ is $G$-invariant, then $\frak k$ is abelian; the converse holds if $G$ is source-connected. This is a direct consequence of Proposition \ref{prop:g_inv_a_inv} and the fact that infinitesimally, invariance of $\nabla$ is equivalent to $\frak k$ being abelian, as shown in Proposition \ref{prop:invariant_abelian}.} To this end, we need to compute the tensor $\Theta$ appearing in Theorem \ref{thm:G_invariant}. We note that defining the connection $\nabla$ entirely in terms of the groupoid $G$ as in equation \eqref{eq:nabla} turns out to be crucial for this purpose, together with the left-invariance property in Proposition \ref{prop:conn} (ii). 

The tensor $\Theta$ relates the pullback connections $\nabla^s$ and $\nabla^t$ on the pullback bundles $s^*\frak k\ra G$ and $t^*\frak k\ra G$. Note that these vector bundles are now both canonically isomorphic to $K\ra G$, with the isomorphisms given by restricting the Maurer--Cartan forms, that is,
\begin{align}
\begin{split}
  s^*\frak k\ra K,\ (g,\xi_{s(g)})&\mapsto \d(L_g)_{1_{s(g)}}(\xi_{s(g)}),\label{eq:iso_pullback_k}\\
t^*\frak k\ra K,\ (g,\xi_{t(g)})&\mapsto \d(R_g)_{1_{t(g)}}(\xi_{t(g)}).
\end{split}
\end{align}
This means that $\nabla^s$ and $\nabla^t$ can both be seen as connections on the vector bundle $K\ra G$. In what follows, these identifications will be employed, in order to relate the two pullback connections  simply by their difference $\nabla^t-\nabla^s$, thus avoiding notational complications. Importantly, we will show that $\nabla_X^t-\nabla_X^s$ depends only on the vertical component of $X$. We begin with a few simple observations. % We also note that under these identifications, pullback sections correspond to invariant vector fields. 
\begin{lemma}
\label{lem:nabla_st}
Let $G\rra M$ be a Lie groupoid with a multiplicative Ehresmann connection $\omega\in\A(G;\frak k)$. For any $\xi\in\Gamma(\frak k)$ and $X\in \vf(G)$ we have:
\begin{align}
\nabla^s_X(\xi^L)&=v[h(X),\xi^L],\label{eq:nablascinfty}\\
\nabla^t_X(\xi^R)&=v[h(X),\xi^R].\label{eq:nablascinfty2}
\end{align}
\end{lemma}
\begin{proof}
We only prove \eqref{eq:nablascinfty}; the proof of \eqref{eq:nablascinfty2} is similar. By definition of $\nabla^s$ and Proposition \ref{prop:conn} (ii), \eqref{eq:nablascinfty} already holds whenever $X\in\vf(G)$ is an $s$-projectable vector field, since $\xi^L$ is identified under \eqref{eq:iso_pullback_k} with $s^*\xi$. However, for any fixed $\xi\in \Gamma(\frak k)$, the map $\vf(G)\ra \Gamma(K)$ given as $X\mapsto v[h(X),\xi^L]$ is $C^\infty(G)$-linear. This implies that the right-hand side of \eqref{eq:nablascinfty} depends on $X$ pointwise, making the requirement of $s$-projectability redundant.
%; indeed, for any $f\in C^\infty(G)$ there holds
%\[
%v[h(fX),\xi^L]=fv[h(X),\xi^L]-(\xi^Lf)v(h(X)),
%\]
%and the second term vanishes
\end{proof}
\begin{remark}[Restriction of $\nabla^s$ to a $K$-connection is independent of $\omega$]
\label{rem:nabla_s}
Importantly, note that the last lemma implies $\nabla^s_X(\xi^L)=0$ for any vertical vector $X\in K$ and any $\xi\in\Gamma(\frak k)$. By virtue of Leibniz rule, the restricted $K$-connection on $K$ (again denoted by $\nabla^s$) is completely determined by this condition, i.e., that it vanishes on left-invariant sections of $K$. This shows that the $K$-connection $\nabla^s$ is intrinsic---it is independent of the choice of a multiplicative connection. Under the isomorphism \eqref{eq:iso_pullback_k}, it coincides with the trivial pullback connection on $s^*\frak k\ra G$ restricted to $K$, since the latter is defined precisely by the property that it vanishes on the pullback sections, and these are identified with left-invariant sections of $K$. More explicitly, in a local frame $(b_i)_i$ of $\frak k$ over an open subset $U\subset M$, any $Y\in\Gamma(K)$ may be expressed as $Y=Y^ib_i^L$ for some functions $Y^i\in C^\infty(s^{-1}(U))$, so that
\[
\nabla^s_X Y=X(Y^i)b_i^L
\]
for any $X\in K_g$, $g\in s^{-1}(U)$. 
%We also observe that in a way, this restriction of $\nabla^s$ measures the left-invariance of sections of $K$.

%More generally, denoting by $\overline{\ker\rho}\subset TG$ the smearing of $\ker \rho$, it is not hard to show that $\overline\nabla^s_X Y\coloneqq X(Y^i)b_i^L$ yields a well-defined linear $\overline{\ker\rho}$-connection on $\overline{\ker\rho}$, independent of the choice of a local frame $(b_i)_i$ of $\ker \rho$; as before, we have expressed $Y=Y^ib_i^L$. It is not hard to see that this connection is flat, that is, it has a vanishing Riemann curvature. The restrictions of $\overline\nabla^s$ and $\nabla^s$ to $K$-connections on $K$, coincide.
\end{remark}
%\begin{remark}
%Even though $K$ and $s^*\frak k$ are isomorphic as vector bundles, they are in general not isomorphic as Lie algebroids. To establish the relation between their brackets, first note that the Lie bracket $[\cdot,\cdot]$ on the former is obtained by seeing $K\subset TG$ as a distribution on $G$ (with anchor the inclusion), and the Lie bracket $[\cdot,\cdot]_{s^*\frak k}$ on the latter is obtained by seeing $s^*\frak k$ as a bundle of Lie algebras (vanishing anchor). Since the two brackets agree on left-invariant sections, i.e.,
%\[
%[\xi^L,\eta^L]=[s^*\xi,s^*\eta]_{s^*\frak k},
%\]
%it is merely an application of the Leibniz rule of $[\cdot,\cdot]$ to see that the two brackets are related by 
%\begin{align}
%\label{eq:relation_brackets}
%  [X,Y]=[X,Y]_{s^*\frak k}+\nabla^s_XY-\nabla^s_YX,
%\end{align}
%for any two vector fields $X,Y\in \Gamma(K)$, where $\nabla^s$ is the $K$-connection on $K$ from Remark \ref{rem:nabla_s}.	
%\end{remark}

\begin{corollary}
\label{cor:diff_nabla_ts}
Let $G\rra M$ be a Lie groupoid with a multiplicative Ehresmann connection $\omega\in\A(G;\frak k)$. For any $X\in \vf(G)$ and $Y\in\Gamma(K)$ there holds:
\begin{align}
\label{eq:diff_nabla_ts}
  \nabla^t_XY-\nabla^s_XY=\nabla^t_{v(X)}Y-\nabla^s_{v(X)}Y.
\end{align}
\end{corollary}
\begin{proof}
Since both sides of equation \eqref{eq:diff_nabla_ts} are $C^\infty(G)$-linear in $Y$, it is enough to show that this equality holds for any left-invariant section $Y\in \Gamma(K)$. Pick  any local frame $(b_i)_i$  of $\frak k$ over an open subset $U\subset M$, and expand the left-invariant section $Y$ with respect to the right-invariant extension of the local frame $(b_i)_i$. That is, write $Y=Y^ib_i^R$ for some functions $Y^i\in C^\infty(t^{-1}(U))$, using the summation convention. The left-hand side reads:
\begin{align*}
\nabla_X^t(Y^ib_i^R)-v[h(X),Y^ib_i^R]&=Y^i\nabla_X^t(b_i^R)+X(Y^i)b_i^R-Y^iv[h(X),b_i^R]-(h(X)Y^i)b_i^R\\
&=(v(X)Y^i)b_i^R=\nabla^t_{v(X)}Y
\end{align*}
where we have used the Leibniz rules for $\nabla^t$ and for the Lie bracket on $K$ in the second equality, and Lemma \ref{lem:nabla_st} in the third. The obtained expression clearly equals the right-hand side of equation \eqref{eq:diff_nabla_ts} since $\nabla^s_{v(X)}Y=0$ due to left-invariance of $Y$.
\end{proof}
\begin{theorem}
\label{thm:deltaD}
Let $G\rra M$ be a Lie groupoid with a multiplicative Ehresmann connection $\omega\in\A(G;\frak k)$. The horizontal exterior covariant derivative $\D{}^\omega$ is a cochain map, i.e., 
\begin{align}
\label{eq:deltaD}
  \delta\D{}^\omega=\D{}^\omega\delta.
\end{align}
In particular, $\D{}^\omega$ maps multiplicative forms to multiplicative forms.
\end{theorem}
In conclusion, a multiplicative Ehresmann connection yields the columns of a curved double complex, depicted in the diagram below, with the feature that $\D{}^\omega$ does not square to zero unless $E=\ker \omega$ is involutive.
\begin{align}
\label{eq:bss_ideals}
\begin{tikzcd}[ampersand replacement=\&, column sep=large, row sep=large]
	{\Omega^{q+1}(M;\frak k)} \& {\Omega^{q+1}(G;s^*\frak k)} \& {\Omega^{q+1}(G^{(2)};(s\circ\pr_2)^*\frak k)} \& \cdots \\
	{\Omega^q(M;\frak k)} \& {\Omega^{q}(G;s^*\frak k)} \& {\Omega^q(G^{(2)};(s\circ\pr_2)^*\frak k )} \& \cdots
	\arrow["{\delta}", from=1-1, to=1-2]
	\arrow["{\delta}", from=1-2, to=1-3]
	\arrow["{\delta}", from=1-3, to=1-4]
	\arrow["{\d{}^\nabla}", from=2-1, to=1-1]
	\arrow["{\delta}", from=2-1, to=2-2]
	\arrow["{\D{}^\omega}", from=2-2, to=1-2]
	\arrow["{\delta}", from=2-2, to=2-3]
	\arrow["{\D{}^\omega}", from=2-3, to=1-3]
	\arrow["{\delta}", from=2-3, to=2-4]
\end{tikzcd}
\end{align}
\begin{proof}
Corollary \ref{cor:diff_nabla_ts} states that the tensor $\Theta$ equals 
\[
\Theta(X)\xi=\phi(\nabla^t_{v(X)}\xi)-\nabla^s_{v(X)}\phi(\xi)
\]
for any $X\in \vf(G)$ and $\xi\in \Gamma(t^*\frak k)$,
where $\phi\colon t^*\frak k\ra s^*\frak k$ is the bundle isomorphism given by $(g,v)\mapsto (g,\Ad_{g^{-1}}(v))$ and $\nabla^s$ and $\nabla^t$ denote the trivial connections on the pullback bundles (as in Remark \ref{rem:nabla_s}). Since $h^*$ is a cochain map, we have 
\[
[\D{}^\omega,\delta]= h^*{\d{}^\nabla}\delta-\delta h^*{\d{}^\nabla}=h^*[\d{}^\nabla,\delta],
\]
which vanishes by the expression for $[\d{}^\nabla,\delta]$ from Lemma \eqref{lem:G_invariant} since $\Theta(h(X))=0$.
\end{proof}
\begin{remark}
  We observe that the proof generalizes to the following setting: suppose $V$ is an arbitrary representation of $G\rra M$ equipped with a linear connection $\nabla$, and $E$ is a multiplicative Ehresmann connection for a fixed bundle of ideals. Then $\D{}=h^*\circ\d{}^\nabla$ is a cochain map $\Omega^{\bullet,q}(G;V)\ra \Omega^{\bullet,q+1}(G;V)$ if and only if there holds
  $h^*\Theta=0$, where $\Theta$ is the $G$-invariance tensor of $\nabla$ from \eqref{eq:invariance_form_G}. Furthermore, taking $\frak k=0_M$ recovers the framework from \sec \ref{sec:invariant_linear_connections}, hence the usual exterior covariant derivative can be seen as a special case of the horizontal exterior covariant derivative.
\end{remark}
%Moreover, we can restrict this diagram to the horizontal forms and obtain a subcomplex of this double complex. 

\needspace{4cm}
\section{Weil complex with values in a bundle of ideals}
\label{sec:weil_boi}
\subsection{Infinitesimal multiplicative Ehresmann connections}
\label{sec:mec_inf}
As observed in \cite{mec}, any multiplicative Ehresmann connection is mapped with the van Est map to an infinitesimal multiplicative form whose symbol restricts to the identity map on the bundle of ideals $\frak k\subset \ker\rho$. Importantly, this infinitesimal notion of a multiplicative connection makes sense on its own, without the need for a given algebroid to be integrable. 

Just like in the case of groupoids, we start this section by briefly recalling some definitions and results regarding infinitesimal multiplicative connections. To begin with, the notion of a bundle of ideals on a Lie algebroid has already been defined in Definition \ref{defn:boi}.
\begin{definition}
Let $\frak k\subset A$ be a bundle of ideals of a Lie algebroid $A\Rightarrow M$. An \textit{infinitesimal multiplicative connection} (more briefly, an \textit{IM connection}) for $\frak k$, is a $\frak k$-valued IM form $(\C,v)\in \Omega^1_{im}(A;\frak k)$, whose symbol $v\colon A\rightarrow \frak k$ restricts on $\frak k$ to the identity:
\begin{align}
\label{eq:symbol_id}
  v|_{\frak k}=\id_{\frak k}.
\end{align}
The set of all IM connections on $A$ for $\frak k$ is denoted $\A(A,\frak k)$.
\end{definition}
%Just as in the case of groupoids, there are several equivalent descriptions of IM Ehresmann connections \cite{mec}*{Proposition 5.4}, but we will only use the one above. 

It is important to recognize that the defining property  \eqref{eq:symbol_id} of an IM connection $(\C,v)$ means precisely that its symbol $v$ is a splitting of the following short exact sequence of vector bundles:
\begin{align}
\label{eq:splitting}
  \begin{tikzcd}[ampersand replacement=\&]
	0 \& {\frak k} \& A \& {A/\frak k} \& 0
	\arrow[from=1-1, to=1-2]
	\arrow[from=1-2, to=1-3]
	\arrow["v", bend left=30, from=1-3, to=1-2]
	\arrow[from=1-3, to=1-4]
	\arrow[from=1-4, to=1-5]
\end{tikzcd}
\end{align}
So, let us denote by $H=\ker (v)$ the \textit{horizontal subbundle} of $A$, and denote by $\alpha=v(\alpha)+h(\alpha)$ the unique decomposition of any vector $\alpha\in A$, pertaining to the splitting $A=\frak k\oplus H$ given by the symbol. We will call $v(\alpha)\in \frak k$ and $h(\alpha)\in H$ the \textit{vertical} and \textit{horizontal} component of $\alpha$, respectively. 

In the case when a Lie groupoid $G$ integrates $A$ and $(\C,v)=\ve(\omega)$ for a multiplicative Ehresmann connection $\omega$, it is clear that the two splittings of the sequence $\eqref{eq:splitting}$ induced by $v$ and $\omega$ coincide, since $v=\omega|_A$.

\subsubsection{Associated coupling data of an IM connection}
The splitting of \eqref{eq:splitting}, determined by the symbol of an IM connection $(\C,v)$, enables us to isolate the information which is contained within the leading term $\C$, into the following two objects.
\begin{enumerate}[label={(\roman*)}]
  \item A linear connection $\nabla$ on the bundle of ideals $\frak k$, given by $\nabla=\C|_{\frak k}$, that is 
\begin{align}
\label{eq:nabla2}
  \nabla_X\xi=\C(\xi)(X).
\end{align}
  \item A tensor field $U\in \Gamma(H^*\otimes T^*M\otimes \frak k)$, given by
\begin{align}
\label{eq:defU}
  U(\alpha)(X)=-\C(\alpha)(X).
\end{align}
\end{enumerate}
The pair $(\nabla,U)$ constructed above will be called the \textit{coupling data} of an IM connection $(\C,v)$, which, owing to the Leibniz identity, indeed consists of a connection and a tensor field. In other words, we can write
\begin{align}
\label{eq:split_C}
  \C(\alpha)=\nabla(v\alpha)-U(h\alpha),
\end{align}
for any $\alpha\in\Gamma(A)$. %\todo{Why are we still using this minus sign? The section on curvings shows it's natural to choose $+U$.}
%We remark that the minus sign appearing in the definition of $U$ will become clear soon, with equality \eqref{eq:s2}.
To see the meaning behind the tensor $U$, note that compatibility condition \eqref{eq:c2} implies
\begin{align}
\label{eq:U_along_orbits}
  v[h(\alpha),h(\beta)]=U(h\alpha)(\rho\beta),
\end{align}
so that the orbital part of $U$ measures the failure of the splitting \eqref{eq:splitting} to be a splitting of Lie algebroids, i.e, $U$ vanishes on $T\F$ precisely when $H$ is a Lie subalgebroid of $A$. However, as we will see, $U$ does not encode the whole information about the curvature of $(\C,v)$, since we also need to account for the curvature $R^\nabla$ of $\nabla$. This will become clearer in \sec\ref{sec:curvature}.

%Turning to the linear connection $\nabla$, we see that its construction is obviously more direct than in the groupoid case (Proposition \ref{prop:conn}). In the integrable case, the two connections coincide:
%\begin{proposition}
%Let $A$ be a Lie algebroid with an IM connection $(\C,v)\in \A(A;\frak k)$. If a Lie groupoid $G$ integrates $A$ and a multiplicative Ehresmann connection $\omega\in \A(G;\frak k)$ integrates the IM connection $(\C,v)$, i.e., $(\C,v)=\operatorname{Lie}(\omega)$, then the two induced linear connections on $\frak k$ coincide.
%\end{proposition}

Turning to the linear connection $\nabla$, we see that its construction is  more direct than in the groupoid case (Proposition \ref{prop:conn}). If $A$ is the Lie algebroid of a Lie groupoid $G$ equipped with the IM connection $(\C,v)=\ve(\omega)$ for a multiplicative Ehresmann connection $\omega$, then the two linear connections on $\frak k$ coincide---this is shown in Proposition \ref{prop:conn_global_inf}. As expected, we also obtain the infinitesimal version of Corollary \ref{cor:conn_orb}.
\begin{proposition}
\label{prop:conn2}
Let $A$ be a Lie algebroid with an IM connection $(\C,v)\in \A(A;\frak k)$. For any $\alpha\in\Gamma(A)$ and $\xi\in\Gamma(\frak k)$, there holds
\begin{align}
\label{eq:conn_orb2}
  \nabla_{\rho(\alpha)}\xi=[h(\alpha),\xi].
\end{align}
%Hence, $\nabla$ restricts to a leafwise connection 
%\[
%\nabla^\F\colon \Gamma(T\F)\times \Gamma(\frak k)\rightarrow \Gamma(\frak k),\quad \nabla^\F_X\xi=[X^*,\xi],\]
%where $X^*$ is any horizontal lift of $X$ along the anchor.
\end{proposition}
\begin{proof}
Equation \eqref{eq:conn_orb2} is shown with a straightforward computation:
\begin{align*}
  \nabla_{\rho(\alpha)}\xi=\C(\xi)(\rho\alpha)=\L^A_\xi v(\alpha)-v[\xi,\alpha]=[\xi,v(\alpha)]+[\alpha,\xi]=[h(\alpha),\xi],
\end{align*}
where we have used the definition of $\nabla$ and condition \eqref{eq:c2}. %Well-definedness of  $\nabla^\F$ in the second part follows from the fact that $\ker\rho=\frak k\oplus (H\cap \ker\rho)$ is a direct sum of ideals, which again follows from condition \eqref{eq:c2}.
\end{proof}

To conclude this section, we note that the compatibility conditions \eqref{eq:c1}--\eqref{eq:c3} for an IM connection $(\C,v)$ can actually be rewritten  entirely in terms of $\nabla$ and $U$, as shown in \cite{mec}*{Propositions 5.11 and 5.13}. We record these rewritten conditions here for our convenience. The following holds for all $\alpha,\beta\in\Gamma(A)$ and $\xi,\eta\in\Gamma(\frak k)$: 
\begin{enumerate}[label={(\roman*)}]
  \item The connection $\nabla$ preserves the Lie bracket on $\frak k$:
  \begin{align}
  \nabla[\xi,\eta]_{\frak k}&=[\nabla\xi,\eta]_{\frak k}+[\xi,\nabla\eta]_{\frak k}.\tag{S.1}\label{eq:s1}
  \end{align}
In particular, $\frak k$ is a locally trivial bundle of Lie algebras.
	\item The curvature tensor $R^\nabla$ of $\nabla$ relates to $U$ as follows:
  \begin{align}
     %R^\nabla(\rho(\alpha),X)\xi=[U(h\alpha)(X),\xi]_{\frak k}.\tag{S.2}\label{eq:s2}
	\iota_{\rho(\alpha)}R^\nabla\cdot\xi=[U(h\alpha),\xi]_{\frak k}.\tag{S.2}\label{eq:s2}
  \end{align}
%In other words, $\iota_{\rho(\alpha)}R^\nabla=\ad\circ U(h\alpha)$.
  \item The tensor $U$ acts on the bracket of sections as:
\begin{align}
%\begin{split}
U(h[\alpha,\beta])=\L^\nabla_{\rho(\alpha)}U(h\beta)-\L^\nabla_{\rho(\beta)}U(h\alpha)+\nabla U(h\alpha)(\rho\beta).%\nabla v[h\alpha,h\beta].%
%U(h[\alpha,\beta])X&=U(h\alpha)[\rho\beta,X]-U(h\beta)[\rho\alpha,X]\\
%&+\nabla_{\rho(\alpha)}U(h\beta)(X)-\nabla_{\rho(\beta)}U(h\alpha)(X)+\nabla_X U(h\alpha)(\rho\beta).
%\end{split}
\tag{S.3}\label{eq:s3}
\end{align}
%In short, $U(h[\alpha,\beta])=\L^\nabla_{\rho(\alpha)}U(h\beta)-\L^\nabla_{\rho(\beta)}U(h\alpha)+\nabla v[\alpha,\beta]$.
\end{enumerate}
Using the point (ii) above, we can now show that $A$-invariance of the connection $\nabla$ is equivalent to the bundle of ideals $\frak k$ being abelian.
\begin{proposition}
\label{prop:invariant_abelian}
Let $A$ be a Lie algebroid with an IM connection $(\C,v)\in \A(A;\frak k)$. The connection $\nabla=\C|_\frak k$ is $A$-invariant if and only if the bundle of ideals $\frak k$ is abelian.
\end{proposition}
\begin{proof}
Equation \eqref{eq:conn_orb2} implies that for any $\alpha\in\Gamma(A)$, we have
\begin{align}
\label{eq:diff_a_conn}
  \nabla^A_\alpha-\nabla_{\rho(\alpha)}=[v(\alpha),\cdot],
\end{align}
so it is clear that $\nabla^A_\alpha=\nabla_{\rho(\alpha)}$ holds if and only if $\frak k$ is abelian. In this case, $\iota_{\rho(\alpha)}R^\nabla=0$ is implied by the identity \eqref{eq:s2}.
\end{proof}
\subsection{Horizontal projection of Weil cochains}
Following the same approach as in the Lie groupoid picture, we would now like to show that an IM connection $(\C,v)\in\A(A;\frak k)$ defines a horizontal projection of Weil cochains. We begin by providing the infinitesimal analogue of horizontal forms on the nerve (Definition \ref{defn:horizontal}). 
\begin{definition}
\label{defn:horizontal_inf}
A Weil cochain $c=(c_0,\dots,c_p)\in W^{p,q}(A;\frak k)$ is said to be \textit{horizontal} if 
\begin{align*}
c_i(\cdot\|\xi,\cdot)=0, \quad\text{for all }i\geq 1\text{ and } \xi\in\Gamma(\frak k).
\end{align*}
Note this is a condition on correction terms only. It is clear from equation \eqref{eq:delta_inf} that $\delta$ maps horizontal cochains to horizontal cochains, so at each fixed $q\geq 0$ we obtain the \textit{horizontal subcomplex} 
\[
W^{\bullet,q}(A;\frak k)^\Hor\leq W^{\bullet,q}(A;\frak k).
\]
Equation \eqref{eq:J_alpha} ensures the van Est map restricts to a map between horizontal subcomplexes.
\end{definition}
\begin{example}
\label{ex:horizontal_p=1}
At level $p=1$, horizontal cochains $c=(c_0,c_1)$ are simply the ones whose symbol restricts on $\frak k$ to zero: \[c_1|_{\frak k}=0.\] It is clear from definition \eqref{eq:delta0_im} of $\delta^0$ that any cohomologically trivial form is horizontal (and multiplicative), that is, $\im\delta^0\subset \Omega^\bullet_{im}(A;\frak k)^\Hor$. On a transitive algebroid, every IM form of degree $q\geq 2$ is horizontal by condition \eqref{eq:c3}, that is, $\Omega^q_{im}(A;\frak k)=\Omega^q_{im}(A;\frak k)^\Hor$.
\end{example}
Just as with multiplicative connections on Lie groupoids, we expect an IM connection for $\frak k$ to induce a horizontal projection map of Weil cochains,
\[
h^*\colon W^{p,q}(A;\frak k)\ra W^{p,q}(A;\frak k)^\Hor.
\] 
However, in contrast with the case of groupoids, there is now no straightforward and intuitive way of defining $h^*$. The issue, essentially, lies in the model $W^{p,q}(A;\frak k)$ that we are using to describe forms in the infinitesimal setting. We overcome this obstacle by employing a technique described in detail in \sec\ref{sec:homogeneous_horizontal}; the idea is the following.
\begin{itemize}
  \item Firstly, note that an IM connection $(\C,v)$ can equivalently be described in terms of an Ehresmann connection $E\subset TA$ for the projection $A\ra A/\frak k$. Infinitesimal multiplicativity of $(\C,v)$ is equivalent to $E$ being a VB-subalgebroid of $TA$. 
  \item Instead of Weil cochains $W^{p,q}(A;\frak k)$, use an alternative model, where the horizontal projection has a simple definition. The alternative model is provided by the so-called \textit{exterior cochains}, which are special forms on a certain VB-algebroid. Similarly to the groupoid case, the map $h^*$ on this alternative model is given  by the precomposition with the map $h\colon TA\ra E$. 
  % Such a model is given by the skew-symmetric multilinear forms $\Gamma_{\ext}(\MM_q,\Lambda^p\AA_q^*)$ from \cite{homogeneous}*{\sec 4}. These are particular sections of the VB-algebroid below.
  % % https://q.uiver.app/#q=WzAsNCxbMCwwLCJcXEFBX3EiXSxbMSwwLCJcXE1NX3EiXSxbMCwxLCJBIl0sWzEsMSwiTSJdLFsyLDNdLFswLDJdLFswLDFdLFsxLDNdXQ==
  % \begin{equation}
  %   \label{eq:aa_mm}
  %   \vcenter{\hbox{
  %     \begin{tikzcd}
  %       {\AA_q} & A \\
  %       {\MM_q} & M
  %       \arrow[from=1-1, to=1-2]
  %       \arrow[Rightarrow, from=1-1, to=2-1]
  %       \arrow[Rightarrow, from=1-2, to=2-2]
  %       \arrow[from=2-1, to=2-2]
  %     \end{tikzcd}
  %   }}
  %   \qquad\qquad
  %   \begin{aligned}
  %   \AA_q &= \oplus^q_A TA\oplus_A \pi^* \frak k^*
  %   \\
  %   \MM_q &= \oplus^q_M TM\oplus_M \frak k^*
  %   % \AA_q &= \underbrace{\makebox[0pt][l]{$TA$}\phantom{TM}\oplus_{\makebox[0pt][l]{\scriptsize $A$}\phantom{M}}\dots\oplus_{\makebox[0pt][l]{\scriptsize $A$}\phantom{M}} \makebox[0pt][l]{$TA$}\phantom{TM}}_{q \text{ copies}}\oplus_A \pi^* V^*
  %   % \\
  %   % \MM_q &= \overbrace{TM\oplus_M\dots\oplus_M TM}\oplus_M V^*
  %   \end{aligned}
  %   \end{equation}
  %   Here, $\pi\colon A\ra M$ denotes the vector bundle projection and $\frak k^*$ is the dual of $\frak k$.
  \item Use the isomorphism %$\ev\colon\Gamma_{\ext}(\MM_q,\Lambda^p\AA_q^*)\ra W^{p,q}(A;\frak k)$ 
  between the two models to derive the wanted formula for $h^*$ on Weil cochains %, defining it by demanding $\ev\circ h^*=h^*\circ\ev$. This can also be used to infer elementary properties of the obtained map $h^*$.
  and infer its elementary properties (Theorem \ref{thm:derivation_hor_proj}). Importantly, establishing the properties of $h^*$ is significantly easier if done in the realm of exterior cochains---it is more conceptual and less combinatorial/computational. %We emphasize that the two models are different in character, and one should not be preferred over the other---the Weil cochain model offers a more hands-on approach, while the exterior cochain model offers more conceptual clarity. 
  %Each of them serves a purpose and offers an enlightening viewpoint when the other cannot. 
\end{itemize}
Ultimately, this procedure yields the map $h^*$ for Weil cochains, which we will now describe. Preliminarily, note that since $V=\frak k$, we may introduce the pairing
\begin{align}
  \label{eq:pairing}
\begin{split}
  &\Omega^k(M;S^j(A^*)\otimes \frak k)\times \Omega^\ell(M;\frak k) \xlongrightarrow{\wedgedot}\Omega^{k+\ell}(M;S^{j-1}(A^*)\otimes \frak k),\\
  &(\gamma\wedgedot \vartheta)(\beta_1,\dots,\beta_{j-1})(X_1,\dots,X_{k+\ell})\\
  &\phantom{(\gamma\wedgedot}=\frac 1{k!\ell!}\smashoperator{\sum_{\quad\sigma\in S_{k+\ell}}}(\sgn\sigma)\gamma\big(\vartheta(X_{\sigma(1)},\dots,X_{\sigma(\ell)}),\beta_1,\dots,\beta_{j-1}\big)(X_{\sigma(\ell+1)},\dots,X_{\sigma(\ell+k)}).
\end{split}
\end{align}
We will actually only need the case $\ell=1$, when $\vartheta$ is a 1-form. The form we obtain by pairing $\gamma$ consecutively with 1-forms $\vartheta_1,\dots,\vartheta_\ell\in\Omega^1(M;\frak k)$, where $\ell\leq j$, will be denoted
\begin{align}
\label{eq:wedgedot_one_by_one}
\gamma\wedgedot (\vartheta_1,\dots,\vartheta_\ell)\coloneqq \gamma\wedgedot \vartheta_\ell\wedgedot\vartheta_{\ell-1}\wedgedot\cdots\wedgedot\vartheta_1 \in \Omega^{k+\ell}(M;S^{j-\ell}(A^*)\otimes\frak k).
\end{align}
It will turn out that we only need the case when $j=\ell$, i.e., when we use up all the symmetric arguments by pairing with 1-forms. A short computation shows that \eqref{eq:wedgedot_one_by_one}, evaluated on vector fields $X_1,\dots, X_{k+\ell}\in\vf(M)$, reads
\begin{align}
  \label{eq:wedgedot_multiple}
\begin{split}
  \big(\gamma\wedgedot (&\vartheta_1,\dots,\vartheta_\ell)\big)(X_1,\dots, X_{k+\ell})\\&=\frac 1{k!}\smashoperator{\sum_{\quad\sigma\in S_{k+\ell}}}(\sgn\sigma)\gamma\big(\vartheta_1(X_{\sigma(1)}),\dots,\vartheta_\ell(X_{\sigma(\ell)})\big)(X_{\sigma(\ell+1)},\dots,X_{\sigma(\ell+k)}).
\end{split}
\end{align}

\begin{definition}
\label{def:hor_proj}
The \textit{horizontal projection} of Weil cochains, induced by an IM connection $(\C,v)\in\A(A;\frak k)$ on a Lie algebroid $A$, is defined as the map
\[
h^*\colon W^{p,q}(A;\frak k)\ra W^{p,q}(A;\frak k)^\Hor,
\]
whose leading term for a given $c\in W^{p,q}(A;\frak k)$ is given by
\begin{align*}
\begin{split}
(h^*c)_0(\alpha_1,\dots,\alpha_p)=\sum_{j=0}^p\,(-1)^j\smashoperator{\sum_{\hspace{2em}\sigma\in S_{(j,p-j)}}}(\sgn\sigma)c_j(\alpha_{\sigma(j+1)},\dots,\alpha_{\sigma(p)})\wedgedot (\C\alpha_{\sigma(1)}, \dots, \C\alpha_{\sigma(j)}),
%(h^*c)_0(\alpha_1,\dots,\alpha_p)=\sum_{j=0}^p\,\frac{(-1)^j}{j!(p-j)!}\smashoperator{\sum_{\hspace{0.3em}\sigma\in S_p}}(\sgn\sigma)c_j(\alpha_{\sigma(j+1)},\dots,\alpha_{\sigma(p)})\wedgedot (\C\alpha_{\sigma(1)}, \dots, \C\alpha_{\sigma(j)}),
\end{split}
\end{align*}
where $S_{(j,p-j)}\subset S_p$ denotes the set of $(j,p-j)$-shuffles. %, that is, permutations $\sigma\in S_{p}$ with $\sigma(1)<\dots<\sigma(j)$ and $\sigma(j+1)<\dots<\sigma(p)$. 
The correction terms are given as
\begin{align*}
\begin{split}
(h^*c)_k&(\alpha_1,\dots \alpha_{p-k}\|\beta_1,\dots,\beta_k)\\
&=\sum_{j=k}^p (-1)^{j-k}\smashoperator{\sum_{\hspace{3em}\sigma\in S_{(j-k,p-j)}}}(\sgn\sigma)c_j(\alpha_{\sigma(j-k+1)},\dots,\alpha_{\sigma(p-k)}\| h\ul\beta,\cdot)\wedgedot (\C\alpha_{\sigma(1)}, \dots, \C\alpha_{\sigma(j-k)}).
%(h^*c)_k&(\alpha_1,\dots \alpha_{p-k}\|\beta_1,\dots,\beta_k)\\
%&=\sum_{j=k}^p \frac{(-1)^{j-k}}{(j-k)!(p-j)!}\smashoperator{\sum_{\hspace{1em}\sigma\in S_{p-k}}}(\sgn\sigma)c_j(\alpha_{\sigma(j-k+1)},\dots,\alpha_{\sigma(p-k)}\| h\ul\beta,\cdot)\wedgedot (\C\alpha_{\sigma(1)}, \dots, \C\alpha_{\sigma(j-k)}).
\end{split}
\end{align*}
The last correction term here simply reads $(h^*c)_p(\ul\beta)=c_p(h\ul\beta).$
\end{definition}
That the correction terms of $h^*$ indeed satisfy the Leibniz identity is a direct consequence of Theorem \ref{thm:derivation_hor_proj}. One can also show well-definedness of $h^*$ directly, using Lemma \ref{lem:wedgedot} (ii) and (v).
\begin{example}
\label{ex:low_levels_h}
We now write down the formula for the horizontal projection for low levels $p$. The most important case is $p=1$, where the above definition reads
\begin{align}
\begin{split}
\label{eq:h_p=1}
(h^*c)_0(\alpha)=c_0(\alpha)-c_1\wedgedot \C\alpha,\quad(h^*c)_1(\beta)=c_1(h\beta).
\end{split}
\end{align}
In other words, $(\alpha\mapsto c_1\wedgedot \C\alpha,\beta\mapsto c_1(v\beta))$ is the vertical part of $c$. We remark that by definition of $\wedgedot$, the second term in the leading coefficient reads
\[
(c_1\wedgedot\C\alpha)(X_1,\dots,X_q)=\textstyle\sum_i(-1)^{i+1}c_1(\C\alpha(X_i))(X_1,\dots,\widehat{X_i},\dots,X_q).
\]
At level $p=2$, we obtain
\begin{align*}
(h^*c)_0(\alpha_1,\alpha_2)&=c_0(\alpha_1,\alpha_2)-(c_1(\alpha_2)\wedgedot \C\alpha_1-c_1(\alpha_1)\wedgedot \C\alpha_2)+c_2\wedgedot (\C\alpha_1,\C\alpha_2),\hspace{0.3em}\\
(h^*c)_1(\alpha\|\beta)&=c_1(\alpha\| h\beta)-c_2(h\beta,\cdot)\wedgedot\C\alpha,\\
(h^*c)_2(\beta_1,\beta_2)&=c_2(h\beta_1,h\beta_2).
\end{align*}
We observe that the $k$-th correction term $(h^*c)_k$ of the horizontal projection of a Weil cochain $c$ contains all the higher correction terms $c_{\ell}$, for $\ell\geq k$.
% To get a good grasp on what the double sum in the definition of the horizontal projection does, let us also write its definition for the level $p=3$:
% \begin{align*}
% (h^*c)_0(\alpha_1,\alpha_2,\alpha_3)&=c_0(\alpha_1,\alpha_2,\alpha_3)\\
% &-(c_1(\alpha_2,\alpha_3)\wedgedot \C\alpha_1 + c_1(\alpha_3,\alpha_1)\wedgedot \C\alpha_2 + c_1(\alpha_1,\alpha_2)\wedgedot \C\alpha_3)\\
% &+(c_2(\alpha_3)\wedgedot(\C\alpha_1,\C\alpha_2) + c_2(\alpha_1)\wedgedot(\C\alpha_2,\C\alpha_3)+c_2(\alpha_2)\wedgedot(\C\alpha_3,\C\alpha_1))\\
% &-c_3\wedgedot(\C\alpha_1,\C\alpha_2,\C\alpha_3),\\
% (h^*c)_1(\alpha_1,\alpha_2\|\beta)&=c_1(\alpha_1,\alpha_2\|h\beta)\\
% &-(c_2(\alpha_2\|h\beta,\cdot)\wedgedot\C\alpha_1-c_2(\alpha_1\|h\beta,\cdot)\wedgedot\C\alpha_2)\\
% &+c_3(h\beta,\cdot)\wedgedot(\C\alpha_1,\C\alpha_2),\\
% (h^*c)_2(\alpha\|\beta_1,\beta_2)&=c_2(\alpha\|h\beta_1,h\beta_2)-c_3(h\beta_1,h\beta_2,\cdot)\wedgedot\C\alpha,\\
% (h^*c)_3(\ul\beta)&=c_3(h\ul\beta).
% \end{align*}
\end{example}
\begin{remark}
\label{rem:wedges}
The pairing \eqref{eq:pairing} should be regarded as follows. Restricting any one of the symmetric arguments of $\gamma\in \Omega^k(M;S^j(A^*)\otimes \frak k)$ to $\frak k$, we obtain a tensor from $\Omega^k(M;S^{j-1}(A^*)\otimes \End\frak k)$, which will again for simplicity just be denoted $\gamma$. The pairing $\wedgedot$ is then related to the following natural pairing: for any vector bundle $V\ra M$, we can define
\begin{align*}
  &\wedge\colon\Omega^k(M;\End V)\times\Omega^\ell(M;V)\rightarrow \Omega^{k+\ell}(M;V),
\\
  &(\gamma\wedge\vartheta)(X_1,\dots,X_{k+\ell})=\frac 1{k!\ell!}\sum_{\mathclap{\hspace{1em}\sigma\in S_{k+\ell}}}\sgn (\sigma)\,\gamma(X_{\sigma(1)},\dots,X_{\sigma(k)})\cdot\vartheta(X_{\sigma(k+1)},\dots,X_{\sigma(k+\ell)}),
\end{align*}
and by an easy combinatorial exercise, the two pairings are related by
\begin{align}
\label{eq:wedge_wedgedot}
\gamma\wedgedot\vartheta=(-1)^{k\ell}\gamma\wedge \vartheta,
\end{align}
where we have omitted the arguments $\beta_1,\dots,\beta_{j-1}$. The usual wedge $\wedge$ is, in a way, more natural to work with, but it would introduce some additional complications regarding the signs if we used it in place of $\wedgedot$ in the definition of $h^*$. In the following lemma, we use this relation to infer some important properties of $\wedgedot$.
\end{remark}
\begin{lemma}
\label{lem:wedgedot}
The pairing \eqref{eq:pairing} satisfies the following properties for $\gamma\in \Omega^k(M;S^j(A^*)\otimes \frak k)$.
\begin{enumerate}[label={(\roman*)}]
\item $\gamma\wedgedot (\vartheta_1,\dots,\vartheta_j)$ is alternating and $C^\infty(M)$-multilinear in $\vartheta_1,\dots,\vartheta_j\in\Omega^1(M;\frak k)$.
\item For a simple tensor $\vartheta\otimes\xi$, where $\vartheta\in\Omega^1(M)$ and $\xi\in\Gamma(\frak k)$, there holds \[\gamma\wedgedot (\vartheta\otimes\xi)=\vartheta\wedge \gamma(\xi,\cdot).\]
More generally, if 1-forms $\vartheta_2,\dots,\vartheta_\ell\in\Omega^1(M;\frak k)$ are given, then
\[
\gamma\wedgedot(\vartheta\otimes\xi, \vartheta_2,\dots,\vartheta_\ell)=\vartheta\wedge\big(\gamma(\xi,\cdot)\wedgedot (\vartheta_2,\dots,\vartheta_\ell)\big).
\]
\item For any $\vartheta\in\Omega^\ell(M;\frak k)$ and $X\in\vf(M)$, there holds
\[\iota_X(\gamma\wedgedot\vartheta)=\gamma\wedgedot(\iota_X\vartheta)+(-1)^\ell (\iota_X\gamma)\wedgedot \vartheta.\]
\item For any $\vartheta\in\Omega^\ell(M;\frak k)$ and $\alpha\in\Gamma(A)$, there holds
\[
\L^A_\alpha(\gamma\wedgedot\vartheta)%(\beta_1,\dots,\beta_{j-1})
=(\L^A_\alpha\gamma)\wedgedot\vartheta+\gamma\wedgedot(\L^A_\alpha\vartheta). 
\]
\item If $\gamma=c_j$ is the $j$-th correction term of $c=(c_0,\dots,c_p)\in W^{p,q}(A;\frak k)$, then there holds
\begin{align*}
c_j(f\alpha_1,\alpha_2,&\dots,\alpha_{p-j})\wedgedot (\vartheta_1,\dots,\vartheta_i)\\
&=f c_j(\ul\alpha)\wedgedot\ul\vartheta+(-1)^i\d f\wedge c_{j+1}(\alpha_2,\dots,\alpha_{p-j}\|\alpha_1,\cdot)\wedgedot \ul\vartheta,
\end{align*}
for any sections $\alpha_1,\dots,\alpha_{p-j}\in\Gamma(A)$, forms $\vartheta_1,\dots,\vartheta_i\in\Omega^1(M;\frak k)$ where $i\leq j$, and any function $f\in C^\infty(M)$. 
\end{enumerate}
\end{lemma}
\begin{proof}
Properties (i) and (ii) are clear from the definition. Property (iii) is easily inferred from relation \eqref{eq:wedge_wedgedot} by noting that the identity
\[
\iota_X(T\wedge\vartheta)=(\iota_X T)\wedge\vartheta+(-1)^k T\wedge(\iota_X\vartheta)
\]
holds for any $T\in\Omega^k(M;\End V)$ and $\vartheta\in\Omega^\ell(M;V)$, for any given vector bundle $V\ra M$. The statement (iv) is a consequence %\todo{Write this argument again} 
of the following general fact:  if an $A$-connection $\nabla^A$ on $V$ is given, then $\L^A$ is distributive over $\wedge$, that is, 
\begin{align}
\label{eq:L_A_wedge}
  \L^A_\alpha(T\wedge\vartheta)=(\L^A_\alpha T)\wedge\vartheta + T\wedge (\L^A_\alpha \vartheta),
\end{align}
where $\L^A_\alpha T$ is defined, as usual, by the chain rule $(\L^A_\alpha T)\cdot\xi=\L^A_\alpha(T\cdot \xi)-T\cdot \nabla_\alpha^A\xi$, for any $\xi\in\Gamma(V)$. It then follows from relation \eqref{eq:wedge_wedgedot} that $\L^A$ is distributive over $\wedgedot$ as well. Finally, item (v) is a straightforward consequence of item (ii), combined with the Leibniz rule for the cochain $c$.
\end{proof}
We now arrive to the crucial property of the horizontal projection $h^*$.
\begin{proposition}
\label{prop:h_cochain}
Let $A$ be a Lie algebroid with an IM connection $(\C,v)\in\A(A;\frak k)$ for a bundle of ideals $\frak k$. The horizontal projection of Weil cochains is a cochain map, that is,
\[
h^*\delta=\delta h^*.
\]
\end{proposition}
This is a direct consequence of Proposition \ref{prop:h_cochain_ext} and Theorem \ref{thm:derivation_hor_proj} since the two models are isomorphic, however, the last lemma enables us to prove it directly. The computation for the general case is tedious, so we only provide a direct proof for the level $p=1$, revealing the key ingredient is infinitesimal multiplicativity of $(\C,v)$.
\begin{proof}
Let $c=(c_0,c_1)\in W^{1,q}(A;\frak k)$ be a Weil 1-cochain; we begin with the leading term. Using the definition \eqref{eq:delta_p=1} of $\delta$ together with the definition of $h^*$, we obtain
\begin{align*}
\begin{split}
(h^*\delta c)_0(\alpha_1,\alpha_2)&=(\delta c)_0(\alpha_1,\alpha_2)-\big((\delta c)_1(\alpha_2\|\cdot)\wedgedot \C\alpha_1- (\delta c)_1(\alpha_1\|\cdot)\wedgedot \C\alpha_2\big)+\cancel{(\delta c)_2\wedgedot (\C\alpha_1,\C\alpha_2)}\\
&=\L^A_{\alpha_1}c_0(\alpha_2)-\L^A_{\alpha_2}c_0(\alpha_1)-c_0[\alpha_1,\alpha_2]-(-\L^A_{\alpha_2}c_1\wedgedot \C\alpha_1+\L^A_{\alpha_1}c_1\wedgedot \C\alpha_2),
\end{split}
\end{align*}
where we have used that $\frak k\subset\ker\rho$. We now apply Lemma \ref{lem:wedgedot} (iv) on the last two terms, and combine with the first two terms to obtain
\begin{align*}
(h^*\delta c)_0&(\alpha_1,\alpha_2)=\L^A_{\alpha_1}(h^*c)_0(\alpha_2)-\L^A_{\alpha_2}(h^*c)_0(\alpha_1)-\big(c_0[\alpha_1,\alpha_2]-c_1\wedgedot(\L^A_{\alpha_1}\C\alpha_2-\L^A_{\alpha_2}\C\alpha_1)\big).
\end{align*}
Using the condition \eqref{eq:c1} for the IM connection $(\C,v)$ shows this is equal to $(\delta h^* c)_0(\alpha_1,\alpha_2)$. We next inspect the first correction term:
\begin{align*}
(h^*\delta c)_1(\alpha\|\beta)&=(\delta c)_1(\alpha\|h\beta)-(\delta c)_2(h\beta,\cdot)\wedgedot \C\alpha\\
&=-\L^A_\alpha (c_1 (h\beta))+c_1[\alpha,h\beta]+\iota_{\rho(\beta)} c_0(\alpha)+(\iota_{\rho(\beta)}c_1)\wedgedot \C\alpha,
\end{align*}
where we have observed that $\rho (h\beta)=\rho(\beta)$. On the other hand,
\begin{align*}
(\delta h^* c)_1(\alpha\|\beta)&=-\L^A_\alpha((h^*c)_1(\beta))+(h^*c)_1[\alpha,\beta]+\iota_{\rho(\beta)}(h^* c)_0(\alpha)\\
&=-\L^A_\alpha (c_1 (h\beta))+c_1 (h[\alpha,\beta])+\iota_{\rho(\beta)}c_0(\alpha)-\iota_{\rho(\beta)}(c_1\wedgedot \C\alpha).
\end{align*}
Using the condition \eqref{eq:c2} on the IM connection, we get 
\[
h[\alpha,\beta]=[\alpha,h\beta]+\iota_{\rho(\beta)}\C\alpha,
\]
and now use Lemma \ref{lem:wedgedot} (iii) to conclude $(h^* \delta c)_1=(\delta h^* c)_1$. The equality for the second correction term is simple and left to the reader.
\end{proof}

\subsection{Horizontal exterior covariant derivative}
\begin{definition}
Let $A$ be a Lie algebroid with an IM connection $(\C,v)\in\A(A;\frak k)$. The \textit{horizontal exterior covariant derivative} of Weil cochains is defined as the map
\begin{align*}
\D{}^{(\C,v)}=h^*\d{}^\nabla\colon W^{p,q}(A;\frak k)\ra W^{p,q+1}(A;\frak k)^\Hor,
\end{align*}
where $\nabla=\C|_\frak k$ is the induced connection on $\frak k$.
\end{definition}
\begin{example}
The most important case is $p=1$, where equations \eqref{eq:ext_cov_der_p=1} and \eqref{eq:h_p=1} yield
\begin{align}
\label{eq:D_inf}
\begin{split}
  (\D{}^{(\C,v)}c)_0(\alpha)&=\d{}^\nabla c_0(\alpha)-(\d{}^\nabla c)_1\wedgedot \C\alpha,\\
  (\D{}^{(\C,v)}c)_1(\beta)&=(\d{}^\nabla c)_1(h\beta),
\end{split}
\end{align}
where we recall that $(\d{}^\nabla c)_1(\beta)=c_0(\beta)-\d{}^\nabla (c_1(\beta))$ for any $\beta\in\Gamma(A)$, hence the second term of the leading coefficient, for any vectors $X_i\in \vf(M)$, reads
\begin{align*}
  \big((\d{}^\nabla c)_1\wedgedot \C\alpha\big)(X_0,\dots,X_q)=\textstyle\sum_i (-1)^i \big(c_0(\C\alpha (X_i))-\d{}^\nabla c_1(\C\alpha (X_i))\big)(X_0,\dots,\widehat{X_i},\dots,X_q).
\end{align*}
\end{example}
%From Proposition \ref{prop:invariant_abelian}, we already know that $\d{}^\nabla$ is a cochain map if and only if $\frak k$ is abelian. 
We now prove the infinitesimal analogue of Theorem \ref{thm:deltaD}.
\begin{theorem}
\label{thm:deltaD_inf}
Let $A$ be a Lie algebroid with an IM connection $(\C,v)\in\A(A;\frak k)$. The horizontal exterior covariant derivative is a cochain map, that is,
\begin{align*}
\delta\D{}^{(\C,v)}=\D{}^{(\C,v)}\delta.
\end{align*}
In particular, $\D{}^{(\C,v)}$ maps IM forms to IM forms. 
\end{theorem}
In other words, an IM connection yields the columns of a curved double complex, depicted in the diagram below, with the feature that $\D{}^{(\C,v)}$ does not square to zero unless $U=0$ and $R^\nabla=0$, as will become clear in \sec\ref{sec:curvature}.
\begin{align}
\label{eq:weil_ideals}
\begin{tikzcd}[ampersand replacement=\&, column sep=large,row sep=large]
	{\Omega^{q+1}(M;\frak k)} \& {W^{1,q+1}(A;\frak k)} \& {W^{2,q+1}(A;\frak k)} \& \cdots \\
	{\Omega^q(M;\frak k)} \& {W^{1,q}(A;\frak k)} \& {W^{2,q}(A;\frak k)} \& \cdots
	\arrow["{\delta}", from=1-1, to=1-2]
	\arrow["{\delta}", from=1-2, to=1-3]
	\arrow["{\delta}", from=1-3, to=1-4]
	\arrow["{\d{}^\nabla}", from=2-1, to=1-1]
	\arrow["{\delta}", from=2-1, to=2-2]
	\arrow["{\D{}^{(\C,v)}}", from=2-2, to=1-2]
	\arrow["{\delta}", from=2-2, to=2-3]
	\arrow["{\D{}^{(\C,v)}}", from=2-3, to=1-3]
	\arrow["{\delta}", from=2-3, to=2-4]
\end{tikzcd}
\end{align}
\begin{proof}
First observe that by Proposition \ref{prop:h_cochain}, $[\D{}^{(\C,v)},\delta]=h^*[\d{}^\nabla,\delta]$. As in the case of groupoids, we will use the explicit expression for $[\d{}^\nabla,\delta]$ from Lemma \ref{lem:A_invariant}, and to do so, we first have to compute the invariance form $(T,\theta)$ of connection $\nabla=\C|_{\frak k}$, given by \eqref{eq:invariance_form}. Equation \eqref{eq:diff_a_conn} already states that the tensor $\theta\colon A\ra \End\frak k$ reads
\[
\theta(\alpha)=[v\alpha,\cdot],
\]
and on the other hand, the map $T\colon\Gamma(A)\ra \Omega^1(M;\End \frak k)$ equals
\begin{align*}
T(\alpha)\cdot\xi&=\nabla(\theta(\alpha)\cdot\xi)-\theta(\alpha)\cdot\nabla\xi-\iota_{\rho(\alpha)}R^\nabla=\nabla[v\alpha,\xi]-[v\alpha,\nabla\xi]-[U(h\alpha),\xi]\\&=[\nabla (v\alpha)-U(h\alpha),\xi]=[\C\alpha,\xi],
\end{align*}
where we have used the properties \eqref{eq:s1} and \eqref{eq:s2} of the IM connection $(\C,v)$. Now, since there holds $h^*((T,\theta)\wedge c)=h^*(T,\theta)\wedge h^* c$% by Lemma \ref{lem:hstar_wedge}
, it is enough to show $h^*(T,\theta)=0$. This is a simple computation: for the leading term,
\begin{align*}
  (h^*(T,\theta))_0(\alpha)(X)\xi=T(\alpha)(X)\xi-\theta(\C\alpha(X))\xi=[\C\alpha(X),\xi]-[\C\alpha(X),\xi]=0,
\end{align*}
and for the symbol, $(h^*(T,\theta))_1(\alpha)=\theta(h\alpha)=0$.
\end{proof}

\subsection{Van Est map and \texorpdfstring{$\D{}^\omega$}{the horizontal exterior covariant derivatives}}
\label{sec:van_est_D}
In this section we inspect the relationship between the van Est map and the horizontal exterior covariant derivatives. The following theorem states they commute at the level of multiplicative forms, and the proof will also reveal that they do not commute on general  cochains.%: a part of the reason is that the van Est map commutes with the horizontal projection only up to homotopy (see Remark \ref{rem:cochain_homotopy} after the proof).
\begin{theorem}
  \label{thm:van_est_D}
  Let $\omega\in\A(G;\frak k)$ be a multiplicative Ehresmann connection on a Lie groupoid $G$. Let $A$ be its Lie algebroid, endowed with the IM connection $(\C,v)=\ve(\omega)$. The van Est map commutes with the horizontal exterior covariant derivatives at the level of multiplicative forms, that is, the following diagram commutes.
% https://q.uiver.app/#q=WzAsNCxbMCwwLCJcXE9tZWdhXlxcYnVsbGV0X20oRztzXipcXGZyYWsgaykiXSxbMSwwLCJcXE9tZWdhXntcXGJ1bGxldCsxfV9tKEc7c14qXFxmcmFrIGspIl0sWzAsMSwiXFxPbWVnYV5cXGJ1bGxldF97aW19KEE7XFxmcmFrIGspIl0sWzEsMSwiXFxPbWVnYV57XFxidWxsZXQrMX1fe2ltfShBO1xcZnJhayBrKSJdLFswLDIsIlxcbWF0aHJte0xpZX0iLDJdLFsxLDMsIlxcbWF0aHJte0xpZX0iXSxbMCwxLCJcXER7fV5cXG9tZWdhIl0sWzIsMywiXFxEe31eXFxDIiwyXV0=

\begin{align}
  \label{eq:square_D}
    \begin{tikzcd}[row sep=large,column sep=large,ampersand replacement=\&]
    {\Omega^\bullet_m(G;\frak k)} \& {\Omega^{\bullet+1}_m(G;\frak k)} \\
    {\Omega^\bullet_{im}(A;\frak k)} \& {\Omega^{\bullet+1}_{im}(A;\frak k)}
    \arrow["{\D{}^\omega}", from=1-1, to=1-2]
    \arrow["{\ve}"', from=1-1, to=2-1]
    \arrow["{\ve}", from=1-2, to=2-2]
    \arrow["{\D{}^{(\C,v)}}"', from=2-1, to=2-2]
  \end{tikzcd}
  \end{align}
\end{theorem}
\begin{lemma}
  \label{lem:cochain_homotopy}
  Let $\omega\in \A(G;\frak k)$ be a multiplicative connection on a Lie groupoid $G\rra M$. The commutator $[\ve,h^*]\colon \Omega^{p,q}(G;\frak k)\ra W^{p,q}(A;\frak k)^\Hor$ at level $p=1$ has a vanishing symbol, and its leading term equals
  \begin{align*}
    \ve (h^*\eta)_0(\alpha)(X_i)_i-(h^*\ve \eta)_0(\alpha)(X_i)_i=\deriv \lambda 0 (v^*\delta \eta)_{(g_\lambda,g_{\lambda}^{\smash{-1}})}\big(X^\lambda_{i}, \d(\inv)h X^\lambda_{i}\big)_i
    \end{align*}
    for any $\eta\in\Omega^q(G;s^*\frak k)$, $\alpha\in \Gamma(A)$ and $X_i\in T_x M$. Here, we have denoted $g_\lambda=\phi^{\alpha^L}_\lambda(1_x)$ and $X_i^\lambda=\d(\phi^{\alpha^L}_{\lambda_{\phantom L}})\d u(X_i)$, and $v^*=\id-h^*$ denotes the vertical projection of differential forms. %The symbol of the Weil cochain $[\ve,h^*]\eta$ vanishes. 
\end{lemma}
\begin{remark}
  \label{rem:cochain_homotopy}
  The lemma suggests that $\ve$ and $h^*$ commute only up to a cochain homotopy $\Psi$, as portrayed in  the (noncommutative) diagram below. Specifically, the formula for $\Psi^2$ can be read from the equation above, whereas $\Psi^1=0$. This general statement will not be further explored here.
% https://q.uiver.app/#q=WzAsMTAsWzIsMCwiXFxPbWVnYV57cCxxfShHO1xcZnJhayBrKSJdLFszLDAsIlxcT21lZ2Fee3ArMSxxfShHO1xcZnJhayBrKSJdLFsxLDEsIldee3AtMSxxfShHO1xcZnJhayBrKSJdLFsyLDEsIldee3AscX0oRztcXGZyYWsgaykiXSxbMywxLCJXXntwKzEscX0oRztcXGZyYWsgaykiXSxbMSwwLCJcXE9tZWdhXntwLTEscX0oRztcXGZyYWsgaykiXSxbNCwxLCJcXGNkb3RzIl0sWzQsMCwiXFxjZG90cyJdLFswLDAsIlxcY2RvdHMiXSxbMCwxLCJcXGNkb3RzIl0sWzAsMSwiXFxkZWx0YV5wIl0sWzIsMywiXFxkZWx0YV57cC0xfSIsMl0sWzAsMiwiXFxEZWx0YV57cCxxfSIsMV0sWzEsMywiXFxEZWx0YV57cCsxLHF9IiwxXSxbMCwzLCJcXHZlIGheKiIsMix7Im9mZnNldCI6MX1dLFswLDMsImheKlxcdmUiLDAseyJvZmZzZXQiOi0xfV0sWzUsMCwiXFxkZWx0YV57cC0xfSJdLFszLDQsIlxcZGVsdGFecCIsMl0sWzgsNV0sWzksMl0sWzEsN10sWzQsNl1d
\[\begin{tikzcd}[row sep=large, column sep=large]
	\cdots & {\Omega^{p-1,q}(G;\frak k)} & {\Omega^{p,q}(G;\frak k)} & {\Omega^{p+1,q}(G;\frak k)} & \cdots \\
	\cdots & {W^{p-1,q}(G;\frak k)} & {W^{p,q}(G;\frak k)} & {W^{p+1,q}(G;\frak k)} & \cdots
	\arrow[from=1-1, to=1-2]
	\arrow["{\delta^{p-1}}", from=1-2, to=1-3]
	\arrow["{\delta^p}", from=1-3, to=1-4]
	\arrow["{\Psi^{p}}"{description}, from=1-3, to=2-2]
	\arrow["{\ve h^*}"', shift right, from=1-3, to=2-3]
	\arrow["{h^*\ve}", shift left, from=1-3, to=2-3]
	\arrow[from=1-4, to=1-5]
	\arrow["{\Psi^{p+1}}"{description}, from=1-4, to=2-3]
	\arrow[from=2-1, to=2-2]
	\arrow["{\delta^{p-1}}"', from=2-2, to=2-3]
	\arrow["{\delta^p}"', from=2-3, to=2-4]
	\arrow[from=2-4, to=2-5]
\end{tikzcd}\]
% % https://q.uiver.app/#q=WzAsOCxbMSwwLCJcXE9tZWdhXnsxLHF9KEc7XFxmcmFrIGspIl0sWzIsMCwiXFxPbWVnYV57MixxfShHO1xcZnJhayBrKSJdLFswLDEsIlxcT21lZ2Fee3F9KE07XFxmcmFrIGspIl0sWzEsMSwiV157MSxxfShHO1xcZnJhayBrKSJdLFsyLDEsIldeezIscX0oRztcXGZyYWsgaykiXSxbMCwwLCJcXE9tZWdhXntxfShNO1xcZnJhayBrKSJdLFszLDEsIlxcY2RvdHMiXSxbMywwLCJcXGNkb3RzIl0sWzAsMSwiXFxkZWx0YV5wIl0sWzIsMywiXFxkZWx0YV57cC0xfSIsMl0sWzAsMiwiXFxQc2leezF9PTAiLDFdLFsxLDMsIlxcUHNpXnsxfSIsMV0sWzAsMywiXFx2ZSBoXioiLDIseyJvZmZzZXQiOjF9XSxbMCwzLCJoXipcXHZlIiwwLHsib2Zmc2V0IjotMX1dLFs1LDAsIlxcZGVsdGFee3AtMX0iXSxbMyw0LCJcXGRlbHRhXnAiLDJdLFsxLDddLFs0LDZdLFs1LDIsIiIsMCx7InN0eWxlIjp7ImhlYWQiOnsibmFtZSI6Im5vbmUifX19XSxbNSwyLCIiLDAseyJvZmZzZXQiOjEsInN0eWxlIjp7ImhlYWQiOnsibmFtZSI6Im5vbmUifX19XV0=
% \[\begin{tikzcd}[row sep=large,column sep=large]
% 	{\Omega^{q}(M;\frak k)} & {\Omega^{1,q}(G;\frak k)} & {\Omega^{2,q}(G;\frak k)} & \cdots \\
% 	{\Omega^{q}(M;\frak k)} & {W^{1,q}(G;\frak k)} & {W^{2,q}(G;\frak k)} & \cdots
% 	\arrow["{\delta^{p-1}}", from=1-1, to=1-2]
% 	\arrow[no head, from=1-1, to=2-1]
% 	\arrow[shift right, no head, from=1-1, to=2-1]
% 	\arrow["{\delta^p}", from=1-2, to=1-3]
% 	\arrow["{\Psi^{1}=0}"{description}, from=1-2, to=2-1]
% 	\arrow["{\ve h^*}"', shift right, from=1-2, to=2-2]
% 	\arrow["{h^*\ve}", shift left, from=1-2, to=2-2]
% 	\arrow[from=1-3, to=1-4]
% 	\arrow["{\Psi^{1}}"{description}, from=1-3, to=2-2]
% 	\arrow["{\delta^{p-1}}"', from=2-1, to=2-2]
% 	\arrow["{\delta^p}"', from=2-2, to=2-3]
% 	\arrow[from=2-3, to=2-4]
% \end{tikzcd}\]
\end{remark}
\begin{proof}
  We begin with the symbol, where the computation is simple---for any section $\beta\in\Gamma(A)$,
  \begin{align*}
  \ve (h^*\eta)_1(\beta)=J_\beta (h^*\eta)=u^*\iota_{\beta^L} (h^*\eta)=u^* h^*\iota_{h\beta^L}\eta=u^*\iota_{h\beta^L}\eta=(h^*\ve\eta)_1(\beta),
  \end{align*}
  where we have used that $E\subset TG$ is a wide subgroupoid. For the leading term, first observe that the vertical projection $v^*= \id-h^*$ enables us to write
\begin{align*}
(R_\alpha h^*\eta)_x(X_i)_i=(R_\alpha \eta)_x(X_i)_i-(R_\alpha v^*\eta)_x(X_i)_i,
\end{align*}
for any vectors $X_i\in T_x M$. Observe that 
\begin{align*}
  (v^*\eta)(X_i)_i&=\eta(vX_i+hX_i)_i-\eta(hX_i)_i=\textstyle\sum\limits_{k=1}^q\textstyle\sum\limits_{\mathclap{\hspace{3.2em}\sigma\in S_{(k,q-k)}}}\sgn(\sigma)\eta(vX_{\sigma(1)},\dots,vX_{\sigma(k)},hX_{\sigma(k+1)},\dots,h X_{\sigma(q)}),
\end{align*}
hence $(R_\alpha v^*\eta)_x(X_i)_i$ may be expressed as a sum of terms of the form (up to a sign)
\begin{align}
  \label{eq:term_in_Rv}
  \deriv\lambda 0 \Ad_{\phi^{\alpha^L}_\lambda(1_x)}\cdot\eta(vY_1^\lambda,\dots, vY_k^\lambda,hY_{k+1}^\lambda,\dots,hY_{q}^\lambda)
\end{align}
for some $k\in\set{1,\dots, q}$, where we are denoting $Y^\lambda_i=X^\lambda_{\sigma(i)}$ for a fixed permutation $\sigma\in S_{(k,q-k)}$, and $X^\lambda_i=\d(\phi^{\alpha^L}_\lambda)\d u(X_i)$. 
We now use the formula for the simplicial differential,
\begin{align*}
  \delta\eta_{(g,h)}(v_i,w_i)_i=\eta_h(w_i)_i-\eta_{gh}(\d m_{(g,h)}(v_i,w_i))_i+\Ad_{h^{-1}}\cdot\eta_g(v_i)_i,
\end{align*}
with arrows $g=g_\lambda\coloneqq \phi^{\alpha^L}_\lambda(1_x)$, $h=g^{-1}$ and vectors
\begin{align*}
  v_i=\begin{cases}
    vY^\lambda_i & \text{if }i\leq k,\\
    h Y^\lambda_i & \text{if }i>k,
  \end{cases}
  \qquad\quad
  w_i=\begin{cases}
    0 & \text{if }i\leq k,\\
    \d(\inv) h Y^\lambda_i & \text{if }i>k.
  \end{cases}
\end{align*}
Using these vectors allows us to write \eqref{eq:term_in_Rv} as
\begin{align*}
  \eqref{eq:term_in_Rv}&= \eta_{gh}(\d m_{(g,h)}(v_i,w_i))_i+\delta\eta_{(g,h)}(v_i,w_i)_i\eqqcolon(\star)+(\star\star)
\end{align*}
It is not hard to see there holds 
\begin{align*}
  \d m(v_i,w_i)=\begin{cases}
    \Ad_{g_\lambda}\cdot\omega(Y^\lambda_i) & \text{if }i\leq k,\\
    \d u (Y^0_i) & \text{if }i>k,
  \end{cases}
\end{align*}
hence the first term $(\star)$ becomes
\begin{align*}
  (\star)&=\deriv\lambda 0 \eta_{1_x}\big(\Ad_{g_\lambda}\cdot\omega(Y_1^\lambda),\dots, \Ad_{g_\lambda}\cdot\omega(Y_k^\lambda),\d u (Y^0_{k+1}),\dots, \d u (Y^0_{q})\big)\\
  &=\sum_{i=1}^k\eta_{1_x}\bigg({\omega(\d u(Y^0_1))},\dots,\underbrace{\deriv\lambda 0\Ad_{g_\lambda}\cdot\omega(Y_i^\lambda)}_{\C(\alpha)(Y^0_i)},\dots,{\omega(\d u(Y^0_k))},\d u(Y^0_{k+1}),\dots, \d u(Y^0_{q})\bigg).
\end{align*}
Since $u^*\omega=0$ by multiplicativity of the connection $\omega$, the expression above is nonzero only in the case when $k=1$. Assuming for a moment that $\eta$ is multiplicative, we thus obtain
\begin{align*}
  (R_\alpha v^*\eta)_x(X_i)_i&=\smallderiv\lambda 0\textstyle\sum_i  \phi^{\alpha^L}_\lambda(1_x)\cdot\eta(hX_1^\lambda,\dots, vX_i^\lambda,\dots,hX_{q}^\lambda)\\
  &=\textstyle\sum_i (-1)^{i+1}\eta_{1_x}(\C\alpha(X_i),\d u(X_1),\dots,\widehat{\d u(X_i)},\dots, \d u(X_q))\\
  &=(c_1\wedgedot \C\alpha)(X_i)_i,
\end{align*}
where we are denoting by $c_1$ the symbol of $\ve (\eta)$. Since the leading term is $c_0(\alpha)=R_\alpha\eta$, this proves the lemma for the case when $\eta$ is multiplicative. If $\eta$ is not multiplicative, the term $(\star\star)$ becomes
{\begin{align*}
  (\star\star)=\deriv \lambda 0 \delta\eta_{(g_\lambda,g_{\lambda}^{\smash{-1}})}\big(&(v Y_1^\lambda, 0),\dots ,(v Y_k^\lambda, 0),(h Y^\lambda_{k+1}, \d(\inv)h Y^\lambda_{k+1}),\dots,(h Y^\lambda_{q}, \d(\inv)h Y^\lambda_{q})\big).
\end{align*}
}Finally, observe that on $TG^{(2)}$ there holds
\begin{align*}
h(Y^\lambda_{i}, \d(\inv)h Y^\lambda_{i})&=(hY^\lambda_{i}, \d(\inv)h Y^\lambda_{i}),\\
v(Y^\lambda_{i}, \d(\inv)h Y^\lambda_{i})&=(vY^\lambda_i,0).
\end{align*}
This establishes the wanted formula for the leading terms. 
\end{proof}
\begin{proof}[Proof of Theorem \ref{thm:van_est_D}]
  We need to check $\ve (h^* \d{}^{\nabla}\eta) =h^*\d{}^\nabla \ve(\eta)$ holds for any multiplicative form $\eta\in\Omega^q_m(G;s^*\frak k)$. By the lemma above, we have
  \begin{align*}
    \ve(h^*\d{}^{\nabla}\eta)_0(\alpha)(X_i)&=(h^*\d{}^\nabla\ve \eta)_0(\alpha)(X_i)_i+\deriv \lambda 0 (v^*\delta \d{}^\nabla\eta)_{(g_\lambda,g_{\lambda}^{\smash{-1}})}\big(X^\lambda_{i}, \d(\inv)h X^\lambda_{i}\big)_i
  \end{align*}
  where we have already used on the first term that $\ve$ commutes with $\d{}^\nabla$ on multiplicative forms, as established in Theorem \ref{thm:van_est_G_A}. To show that the second term vanishes, first note that multiplicativity of the form $\eta$ implies
  \[
    v^*\delta\d{}^\nabla\eta=-v^*[\d{}^\nabla,\delta]\eta.
  \]
  As already observed in the lemma, $v^*$ acts on individual arguments of a given form by either the horizontal or the vertical projection, so the second component $\d(\inv)h X^\lambda_{i}$ of any vector is either untouched by $v^*$ or sent to zero. This is important since by the formula for the commutator \eqref{eq:commutator_higher}, the second components are inserted into $\Theta$, and so by the fact that $h^*\Theta=0$ from Corollary \ref{cor:diff_nabla_ts}, the second term on the right-hand side above vanishes. This proves the theorem for leading terms; that the symbols coincide follows directly from the last lemma.
\end{proof}

\subsection{VB-algebroid picture of Weil cochains and IM connections%: deriving the \texorpdfstring{formula for $h^*$}{horizontal projection formula}
}
\label{sec:homogeneous_horizontal}
The purpose of this section is to derive the formula for the horizontal projection $h^*$ of Weil cochains, induced by an IM connection (Definition \ref{def:hor_proj}). Essentially, the idea is to identify Weil cochains with special forms on a certain algebroid, where the formula for $h^*$ is simple. Independently of this purpose, the process of defining this algebroid and making this identification uncovers interesting underlying geometric structures, so it provides an important viewpoint of the theory. For instance, we will observe that although isomorphic, the two models are in practice very different to work with---the Weil cochain model offers a more hands-on, computational approach, while the alternative model offers more conceptual clarity. The alternative model comes from the study of homogeneous cochains on VB-algebroids \cite{homogeneous}, which is an interesting topic in itself. This section hence assumes some familiarity with VB-algebroids; some good references for the basics of VB-algebroids and double vector bundles are \cites{bundles_over_gpds, mackenzie_doubles, dvb, gracia-saz, vb-algebroid-morphisms, mackenzie_duality}. 

For the reader's convenience, let us state how we have organized this section. In \sec\ref{sec:alternative_model_weil}, we recall the construction of the alternative model for Weil cochains from \cite{homogeneous}. In \sec\ref{sec:im_connections_distributions}, we connect this construction with the interpretation of IM connections as VB-subalgebroids of $TA$ from \cite{mec}. The method of obtaining $h^*$ for Weil cochains is then presented in \sec\ref{sec:derivation_horproj_weil}, together with some prerequisite results which are interesting on their own. Although we have found it impossible to keep this section short, we will try to keep the discussion goal-oriented, i.e., we will primarily focus on obtaining the wanted formula for $h^*$ and its properties. 

\subsubsection{Alternative model for the Weil complex}
\label{sec:alternative_model_weil}
Let us first fix some notation regarding the tangent algebroid of $A$. It is the algebroid $TA\Ra TM$, defined in such a way that the bundle projection $TA\ra A$ is a Lie algebroid map covering the projection $TM\ra M$ \cite{bialgebroids}*{Theorem 5.1}.
% https://q.uiver.app/#q=WzAsNCxbMCwwLCJUQSJdLFsxLDAsIkEiXSxbMCwxLCJUTSJdLFsxLDEsIk0iXSxbMCwxXSxbMiwzXSxbMCwyXSxbMSwzXV0=
\[\begin{tikzcd}
	TA & A \\
	TM & M
	\arrow[from=1-1, to=1-2]
	\arrow[Rightarrow, from=1-1, to=2-1]
	\arrow[Rightarrow, from=1-2, to=2-2]
	\arrow[from=2-1, to=2-2]
\end{tikzcd}\]
This is a double vector bundle---it has two mutually compatible vector bundle structures, in the sense of \cite{mackenzie_duality}*{Definition 1.1}. The two addition maps and zero maps will be denoted by
\begin{align*}
  +_A&\colon TA\oplus_{A}TA\ra TA,  & +_{TM}&=\d{(+)}\colon TA\oplus_{TM} TA\ra TA,\\
  0_A&\colon A\ra TA, &0_{TM}&=\d{0}\colon TM\ra TA,
\end{align*}
where $+\colon A\oplus_M A\ra A$ denotes the addition in $A$ and $0\colon M\ra A$ is the zero section; we denote the respective homogeneous structures by $h_A^\lambda$ and $h_{TM}^\lambda$. Moreover, $TA$ is a VB-algebroid, meaning that the vector bundle operations fibred over $A$ are Lie algebroid morphisms over the corresponding operations in $TM\ra M$, where $TA\oplus_A TA\Ra TM\oplus_M TM$ carries the componentwise structure.

Now let us suppose $V$ is any representation of $A$. It gives rise to the \textit{prolongation VB-algebroid} $\pi^*V\Ra 0_M$, where $\pi\colon A\ra M$ is the bundle projection (see diagram below). Taking its dual with respect to the fibration over $A$, we obtain the action algebroid $\pi^*V^*\Ra V^*$ of the induced representation of $A$ on $V^*$.
\begin{equation*}
  %\label{eq:vb_algebroid_representation}
  \begin{aligned}
    [(\alpha,\xi),(\beta,\eta)]=([\alpha,\beta],\nabla^V_\alpha\eta-\nabla^V_\beta\xi)
    \\
    \rho(\alpha,\xi)=\rho(\alpha)
  \end{aligned}
  \quad
  \vcenter{\hbox{
% https://q.uiver.app/#q=WzAsNCxbMCwwLCJcXHBpXipWIl0sWzEsMCwiQSJdLFswLDEsIjBfTSJdLFsxLDEsIk0iXSxbMCwxXSxbMiwzXSxbMCwyLCIiLDEseyJsZXZlbCI6Mn1dLFsxLDMsIiIsMSx7ImxldmVsIjoyfV1d
\begin{tikzcd}
	{\pi^*V} & A \\
	{0_M} & M
	\arrow[from=1-1, to=1-2]
	\arrow[Rightarrow, from=1-1, to=2-1]
	\arrow[Rightarrow, from=1-2, to=2-2]
	\arrow[from=2-1, to=2-2]
\end{tikzcd}
  }}
  \,
  \xrightarrow{\text{dualize}}
  \,
  \vcenter{\hbox{
% https://q.uiver.app/#q=WzAsNCxbMCwwLCJcXHBpXipWIl0sWzEsMCwiQSJdLFswLDEsIjBfTSJdLFsxLDEsIk0iXSxbMCwxXSxbMiwzXSxbMCwyLCIiLDEseyJsZXZlbCI6Mn1dLFsxLDMsIiIsMSx7ImxldmVsIjoyfV1d
\begin{tikzcd}
	{\pi^*V^*} & A \\
	{V^*} & M
	\arrow[from=1-1, to=1-2]
	\arrow[Rightarrow, from=1-1, to=2-1]
	\arrow[Rightarrow, from=1-2, to=2-2]
	\arrow[from=2-1, to=2-2]
\end{tikzcd}
  }}
  \end{equation*}
Hence, we can consider the following VB-algebroid, obtained as the Whitney sum of VB-algebroids:
\begin{equation}
  \label{eq:vb_algebroid}
  \vcenter{\hbox{
    \begin{tikzcd}
      {\AA_q} & A \\
      {\MM_q} & M
      \arrow[from=1-1, to=1-2]
      \arrow[Rightarrow, from=1-1, to=2-1]
      \arrow[Rightarrow, from=1-2, to=2-2]
      \arrow[from=2-1, to=2-2]
    \end{tikzcd}
  }}
  \qquad
  \begin{aligned}
  \AA_q &= \oplus^q_A TA\oplus_A \pi^* V^*,
  \\
  \MM_q &= \oplus^q_M TM\oplus_M V^*.
  % \AA_q &= \underbrace{\makebox[0pt][l]{$TA$}\phantom{TM}\oplus_{\makebox[0pt][l]{\scriptsize $A$}\phantom{M}}\dots\oplus_{\makebox[0pt][l]{\scriptsize $A$}\phantom{M}} \makebox[0pt][l]{$TA$}\phantom{TM}}_{q \text{ copies}}\oplus_A \pi^* V^*
  % \\
  % \MM_q &= \overbrace{TM\oplus_M\dots\oplus_M TM}\oplus_M V^*
  \end{aligned}
  \end{equation}
The algebroid structure is defined componentwise. This is again a double vector bundle---the vector bundle structure fibred over $A$ is just the usual Whitney sum structure, while the one fibred over $\MM_q$ is induced componentwise by $TA\ra TM$ and $\pi^*V^*\ra V^*$, that is,
\begin{align*}
  (X_1,\dots,X_q,(a,\zeta))+_{\MM_q}(X_1',\dots,X_q',(a',\zeta))&=(X_1 +_{TM} X_1',\dots,X_q +_{TM} X_q',(a+a',\zeta)),\\
  0_{\MM_q}(w_1,\dots,w_q,\zeta)&=(0_{TM}(w_1),\dots,0_{TM}(w_q),(0_x,\zeta)),
\end{align*}
where the vectors $X_i\in T_a A$, $X_i'\in T_{a'}A$ satisfy $\d\pi(X_i)=\d\pi(X_i')$, and $\zeta\in V_{\pi(a)}^*$. With the VB-algebroid \eqref{eq:vb_algebroid} in mind, we now define the alternative model to Weil cochains from \cite{homogeneous}. 

\begin{definition}
  Given any section $\omega\in\Gamma(\MM_q,\Lambda^p\AA_q^*)$,\footnote{The dual and the wedge are with respect to the bundle structure fibred over $\MM_q$.} we say that:
  \begin{enumerate}[label={(\roman*)}]
    \item $\omega$ is \textit{skew-symmertic} with respect to $\AA_q\ra A$, if for any $\sigma\in S_q$,
    \[
    (\sigma_A)^*\omega=\sgn(\sigma)\omega,
    \]
    where $\sigma_A\colon\AA_q\ra\AA_q$ permutes the $q$ components in $\oplus^q_A TA$ according to $\sigma$.
    \item $\omega$ is \textit{multilinear} with respect to $\AA_q\ra A$, if it is
    \begin{align*}
      &\textit{$(q+1)$-homogeneous}\text{:}\hspace{-6em}&(h^\lambda_A)^*\omega&=\lambda^{q+1}\omega,\\
      &\text{and }\textit{simple}\text{:}&(0^i_A)^*\omega&=0,\quad (i=1,\dots,q+1),
    \end{align*}
    where $h_\lambda^A\colon \AA_q\ra\AA_q$ is the homogeneous structure of $\AA_q\ra A$, and the maps $0^i_A\colon\AA_{q-1}\ra\AA_q$ and $0^{q+1}_A\colon \oplus^q_A TA\ra \AA_q$ insert a zero at the $i$-th and $(q+1)$-th factor, respectively.
  \end{enumerate}
  A skew-symmetric and multilinear section  with respect to $\AA_q\ra A$ will be called an \textit{exterior cochain}, and the set of all exterior cochains will be denoted by
  \[
  \Gamma_{\ext}(\MM_q,\Lambda^p\AA_q^*).
  \]
\end{definition}
It is shown in \cite{homogeneous}*{Proposition 4.10} that $\Gamma_{\ext}(\MM_q,\Lambda^\bullet\AA_q^*)\subset \Gamma(\MM_q,\Lambda^\bullet\AA_q^*)$ is a subcomplex of the usual cochain complex used to define the algebroid cohomology of $\AA_q\Ra \MM_q$. The differential there is just the standard one: for any sections $X^i\in\Gamma(\MM_q,\AA_q)$,
\begin{align}
  \label{eq:delta_vb}
\begin{split}
    \delta\omega(X^0,\dots, X^p)&=\textstyle\sum_i(-1)^i\L_{\rho_{\AA_q}(X^i)}\omega(X^0,\dots,\widehat{X^i},\dots,X^p)\\
    &+\textstyle\sum_{i<j}\omega([X^i,X^j]_{\AA_q},X^0,\dots,X^i,\dots,X^j,\dots,X^p).
\end{split}
\end{align}
Here, $\L$ denotes the usual directional derivative of a function. 
\begin{remark}
  To demystify the abstract definition of multilinearity with respect to $\AA_q\ra A$, let us try to see it as multilinearity in the usual sense. Denoting the exterior cochains at the level $p=0$ by $C^\infty_\ext(\MM_q)\coloneqq\Gamma_{\ext}(\MM_q,\Lambda^0\AA_q^*)$, they form a subset of $C^\infty(\oplus^q_M TM\oplus_M V^*)$. If $\omega\in C^\infty_\ext(\MM_q)$ is $(q+1)$-homogeneous and simple, it is also linear in each component by a standard argument as in Euler's homogeneous function theorem. Hence, it is multilinear in the usual sense, and we conclude 
  \[C^\infty_\ext(\MM_q)=\Gamma(\Lambda^q(T^*M)\otimes (V^*)^*)\cong\Omega^q(M;V). \]
  Similarly, at level $p=1$, multilinearity of $\omega\in \Gamma_\ext(\MM_q,\Lambda^1\AA_q^*)$ with respect to $\AA_q\ra A$ is just a terse way of expressing
  \[
  \omega(X_1+_A h_A^\lambda X_1',X_2\dots,X_q,(a,\zeta))=\omega(X_1,\dots,X_q,(a,\zeta))+\lambda \omega( X_1',\dots,X_q,(a,\zeta)),
  \]
  for any $\lambda\in \R$ and vectors $X_1',X_i\in T_a A, \zeta\in V^*_{\pi(a)}$, and similarly for the other arguments and for linearity in $(a,\zeta)$. Going one level further, at $p=2$ multilinearity with respect to $\AA_q\ra A$ reads
\begin{align*}
  &\omega\big((X_1+_A h_A^\lambda X_1',X_2\dots,X_q,(a,\zeta)),(Y_1+_A h_A^\lambda Y_1',Y_2\dots,Y_q,(b,\zeta))\big)\\
  &=\omega\big((X_1,\dots,X_q,(a,\zeta)),(Y_1,Y_2\dots,Y_q,(b,\zeta))\big)+\lambda \omega\big( (X_1',\dots,X_q,(a,\zeta)),(Y_1',Y_2\dots,Y_q,(b,\zeta))\big),
\end{align*}
where now the vectors $Y_1',Y_i\in T_b A, b\in\pi^{-1}(a)$ must satisfy $\d\pi(Y_1')=\d\pi(X_1')$ and $\d\pi(Y_i)=\d\pi(X_i)$.
\end{remark}

We now describe the canonical identification of the model of exterior cochains with the model of Weil cochains, called the \textit{evaluation map} and denoted by
\begin{align}
  \label{eq:evaluation_map}
  \ev\colon\Gamma_{\ext}(\MM_q,\Lambda^p\AA_q^*)\xrightarrow{} W^{p,q}(A;V).
\end{align}
To do so, we first note that by \cite{mackenzie_doubles}*{Proposition 2.2}, the space of sections $\Gamma(\MM_q,\AA_q)$ is generated as a $C^\infty(\MM_q)$-module by the linear and core sections, i.e., sections of the following form: for $\alpha\in\Gamma(A)$, 
\begin{align}
  \label{eq:linear_and_core_sections}
\begin{split}
    \T\alpha{(w_1,\dots,w_q,\zeta)}&=(\d\alpha(w_1),\dots,\d\alpha(w_q),\chi_\alpha(\zeta)),\\
    \Z_i\alpha{(w_1,\dots,w_q,\zeta)}&=\Big(0_{TM}(w_1),\dots,0_{TM}(w_i)+_A\deriv\lambda 0 \lambda \alpha_x,\dots,0_{TM}(w_q),0_\zeta\Big).
\end{split}
\end{align}
Here, $w_j\in T_x M, \zeta\in V^*_x$ and we have introduced the following vectors in $\pi^*V^*$,
\[
  \chi_\alpha(\zeta)=(\alpha_x,\zeta),\quad\text{and}\quad 0_\zeta=(0_x,\zeta).
\]
Importantly, the wanted map $\ev$ is defined by evaluating $\omega$ on the generating sections; this is done  via auxiliary maps. More precisely, for any $\omega\in\Gamma_{\ext}(\MM_q,\Lambda^p\AA_q^*)$, we define
\begin{align}
  \label{eq:def_exterior_weil_generators}
\begin{split}
    &\tilde c_k(\omega)\colon{\times}^p\Gamma(A)\ra C^\infty(\MM_q),\\
    &\tilde c_k(\omega)(\alpha_1,\dots,\alpha_{p-k}\|\beta_1,\dots,\beta_k)=\omega(\Z_1\beta_1,\dots,\Z_k\beta_k,\T\alpha_1,\dots,\T\alpha_{p-k}),
\end{split}
\end{align}
if $k\leq q$, and zero for all $k>q$. In turn, these maps define functions
\begin{align}
  \label{eq:def_exterior_weil_ck}
  \begin{split}
    &c_k(\omega)(\ul\alpha\|\ul\beta)\in C^\infty(\oplus^{q-k}_MTM\oplus_M V^*)\\
    &c_k(\omega)(\ul\alpha\|\ul\beta)(w_1,\dots,w_{q-k},\zeta)=\tilde c_k(\omega)(\ul\alpha\|\ul\beta)(0_x,\dots,0_x,w_1,\dots,w_{q-k},\zeta),%\quad  w_i\in T_x M,\, \zeta\in V^*_x,
  \end{split}
\end{align}
which are shown to be $(q-k+1)$-homogeneous, simple and skew-symmetric, so they correspond to differential forms $c_k(\omega)(\ul\alpha\|\ul\beta)\in\Omega^{q-k}(M;V)$.
The maps $c_k(\omega)$ can thus be stacked in a sequence as terms of the wanted Weil cochain corresponding to $\omega$,
\[\ev(\omega)\coloneqq\big(c_0(\omega),\dots,c_p(\omega)\big)\in W^{p,q}(A;V).\]
%We again emphasize that the key idea here was to view differential forms as functions on the Whitney sum of vector bundles---for a precise way of doing so, see \cite{homogeneous}*{Section 4.2}. 
For any fixed $q$, this map provides the wanted isomorphism of the complex of exterior cochains with the Weil complex, as shown in \cite{homogeneous}*{Proposition A.3}. We observe that the complexity of the Koszul differential \eqref{eq:delta_inf} of Weil cochains is now encoded entirely in the Lie algebroid $\AA_q\Ra \MM_q$, since the differential on exterior cochains \eqref{eq:delta_vb} is just the standard one.
\begin{remark}
  \label{rem:any_vectors_ck}
  It is not hard to see that we can actually pick any vectors instead of $0_x$'s in the defining equation \eqref{eq:def_exterior_weil_ck}, by multilinearity of $\omega$ over $\AA_q\ra A$. That is,
  \[
    c_k(\omega)(\ul\alpha\|\ul\beta)(w_1,\dots,w_{q-k},\zeta)=\tilde c_k(\omega)(\ul\alpha\|\ul\beta)(\tilde w_1,\dots,\tilde w_k,w_1,\dots,w_{q-k},\zeta),
  \]
  for any vectors $\tilde w_i\in T_x M$. For example, in the case $p=1$, we have
  \begin{align*}
    \tilde c_1(\omega)(\beta)(\tilde w,w_1,\dots,w_{q-1},\zeta)&=\omega\big(0_{TM}(\tilde w)+_A\smallderiv\lambda 0\lambda \beta_x,0_{TM}(w_1),\dots 0_{TM}(w_{q-1}),0_\zeta\big)\\
    &=\omega\big(0_{\MM_q}(\tilde w,w_1,\dots,w_{q-1},\zeta)\big)+\tilde c_1(\omega)(\beta)(0_x,w_1,\dots,w_{q-1},\zeta),
  \end{align*}
  where multilinearity over $A$ was used in the second equality. The first term then vanishes by multilinearity over $\MM_q$, since $\omega\in\Gamma(\MM_q,\Lambda^p\AA_q^*)$.
\end{remark}

\subsubsection{IM connections as distributions}
\label{sec:im_connections_distributions}
In this subsection, we will observe that the viewpoint of IM connections as VB-subalgebroids of $TA$, known from \cite{mec}, is just a manifestation of the model isomorphism \eqref{eq:evaluation_map} between Weil and exterior cochains.

First note that the differential of the projection $\phi\colon A\ra B=A/\frak k$ (with $\phi$ viewed as a surjective submersion between smooth manifolds) defines a short exact sequence of vector bundles, which is, moreover, a sequence of VB-algebroids.
% https://q.uiver.app/#q=WzAsMTAsWzAsMCwiMCJdLFsxLDAsIksiXSxbMiwwLCJUQSJdLFszLDAsIlRBL0siXSxbNCwwLCIwIl0sWzAsMSwiMCJdLFsxLDEsIjBfTSJdLFsyLDEsIlRNIl0sWzMsMSwiVE0iXSxbNCwxLCIwIl0sWzAsMV0sWzEsMiwiIiwwLHsic3R5bGUiOnsidGFpbCI6eyJuYW1lIjoiaG9vayIsInNpZGUiOiJ0b3AifX19XSxbMiwzLCJcXGRcXHBoaSJdLFszLDRdLFs1LDZdLFs2LDcsIiIsMCx7InN0eWxlIjp7InRhaWwiOnsibmFtZSI6Imhvb2siLCJzaWRlIjoidG9wIn19fV0sWzcsOF0sWzgsOV0sWzMsOF0sWzIsN10sWzEsNl1d
\begin{align}
  \label{eq:ses_vb_algebroids}
  \begin{tikzcd}[ampersand replacement=\&]
  	0 \& \ker\d\phi \& TA \& {TB} \& 0. \\
  	0 \& {0_M} \& TM \& TM \& 0
  	\arrow[from=1-1, to=1-2]
  	\arrow[hook, from=1-2, to=1-3]
  	\arrow[Rightarrow,from=1-2, to=2-2]
  	\arrow["{\d\phi}", from=1-3, to=1-4]
  	\arrow[Rightarrow,from=1-3, to=2-3]
  	\arrow[from=1-4, to=1-5]
  	\arrow[Rightarrow,from=1-4, to=2-4]
  	\arrow[from=2-1, to=2-2]
  	\arrow[hook, from=2-2, to=2-3]
  	\arrow[from=2-3, to=2-4]
  	\arrow[from=2-4, to=2-5]
  \end{tikzcd}
\end{align}
Importantly, and as was observed in \cite{mec}*{Lemma 5.2}, there is a canonical identification of $K\coloneqq \ker\d\phi$ with $A\oplus_M\frak k=\pi^*\frak k$, given by the injective map
\begin{align}
  \label{eq:identification_i_K}
  &i\colon A\oplus_M\frak k\ra TA,\quad i(\alpha,\xi)=\deriv\lambda 0 (\alpha+\lambda\xi)
\end{align}
whose image is precisely $K$ due to dimensional reasons and since $\frak k=\ker\phi$. It is easy to see that the formula for $i$ can be rewritten in the following way, which will be more useful later:
\begin{align}
  \label{eq:identification_i_plus_tm}
  i(\alpha,\xi)=0_A(\alpha)+_{TM}\deriv\lambda 0 \lambda\xi.
\end{align}

Now, to obtain the correspondence of IM connections $\A(A;\frak k)$ with VB-subalgebroids of $TA$ complementary to $K$, suppose we are given a splitting $v^{TA}\colon TA\ra K$ of the sequence of VB-algebroids \eqref{eq:ses_vb_algebroids} covering the identity on $A\Ra M$. Such a splitting can be identified via the isomorphism $i\colon A\oplus_M\frak k\ra K$ with a Lie algebroid morphism $\nu\colon TA\ra A\oplus_M\frak k$ satisfying $\nu\circ i=\id_{\pi^*\frak k}$.
% https://q.uiver.app/#q=WzAsOCxbMCwwLCJUQSJdLFsyLDAsIkFcXG9wbHVzX01cXGZyYWsgayJdLFswLDIsIlRNIl0sWzIsMiwiMF9NIl0sWzEsMSwiQSJdLFszLDEsIkEiXSxbMSwzLCJNIl0sWzMsMywiTSJdLFswLDEsInZee1RBfSJdLFswLDIsIiIsMix7ImxldmVsIjoyfV0sWzEsMywiIiwwLHsibGV2ZWwiOjJ9XSxbMCw0XSxbMSw1XSxbNCw1LCJcXGlkX0EiLDAseyJsYWJlbF9wb3NpdGlvbiI6MjB9XSxbMiwzXSxbMiw2XSxbMyw3XSxbNiw3XSxbNSw3LCIiLDEseyJsZXZlbCI6Mn1dLFs0LDYsIiIsMSx7ImxldmVsIjoyfV1d
\[\begin{tikzcd}[row sep=small, column sep=small]
	TA && {A\oplus_M\frak k} \\
	& A && A \\
	TM && {0_M} \\
	& M && M
	\arrow["{\nu}", from=1-1, to=1-3]
	\arrow[from=1-1, to=2-2]
	\arrow[Rightarrow, from=1-1, to=3-1]
	\arrow[from=1-3, to=2-4]
	\arrow[Rightarrow, from=1-3, to=3-3]
	\arrow["{\id_A}"{pos=0.2}, from=2-2, to=2-4, crossing over]
	\arrow[Rightarrow, from=2-4, to=4-4]
	\arrow[from=3-1, to=3-3]
	\arrow[from=3-1, to=4-2]
	\arrow[from=3-3, to=4-4]
	\arrow[from=4-2, to=4-4]
  \arrow[Rightarrow, from=2-2, to=4-2, crossing over]
\end{tikzcd}\]
Since it is $C^\infty(A)$-linear, $\nu$ can thus be viewed as an element $\nu\in \Gamma_\ext(\MM_1,\Lambda^1\AA_1^*)$. Using the identification of models \eqref{eq:evaluation_map}, one then obtains an IM form
\[
(\C,v)=\ev(\nu)\in \Omega_{im}^1(A;\frak k),
\]
whose multiplicativity corresponds to $\nu$ being a VB-algebroid morphism, and  the condition $v|_{\frak k}=\id_{\frak k}$ corresponds to $\nu\circ i=\id_{\pi^*\frak k}$. Explicitly, the pair $(\C,v)$ reads
\begin{align}
  \label{eq:C_vb_picture}
  \C(\alpha)(w)&=\pr_{\frak k}\big(\nu(\d\alpha(w))\big),\quad (w\in T_xM)\\
  v(\alpha)&=\pr_{\frak k}(\nu(\alpha_c)),\label{eq:v_vb_picture}
\end{align}
where $\alpha_c\in\Gamma(TM,TA)$ is the core section induced by $\alpha\in\Gamma(A)$. It is implicit here that the vector $\nu(\alpha_c(w))$ is independent of the choice $w\in T_x M$, since $v^{TA}(0_{TM}(w))$ equals the zero of the core $0_A(0_x)=0_{TM}(0_x)$, so we obtain
\begin{align*}
  v^{TA}(\alpha_c(w))=v^{TA}\Big(0_{TM}(w)+_A\deriv\lambda 0 \lambda \alpha_x\Big)=v^{TA}\Big(\deriv\lambda 0 \lambda \alpha_x\Big).
\end{align*}
In other words,
\begin{align}
  \label{eq:v_core_section}
  v^{TA}\Big(\deriv\lambda 0\lambda\alpha_x\Big)=\deriv\lambda 0 \lambda v(\alpha_x).
\end{align}
Since any exterior cochain $\nu\in\Gamma_\ext(\MM_1,\AA_1^*)$ is determined by its values on the linear and core sections of $\AA_1$, identities \eqref{eq:C_vb_picture} also tell us how an IM connection $(\C,v)$ induces $\nu$.
Summing up, any IM connection can alternatively be viewed as a wide VB-subalgebroid $E\subset TA$:
% https://q.uiver.app/#q=WzAsNCxbMCwwLCJFIl0sWzEsMCwiQSJdLFswLDEsIlRNIl0sWzEsMSwiTSJdLFswLDIsIiIsMCx7ImxldmVsIjoyfV0sWzEsMywiIiwwLHsibGV2ZWwiOjJ9XSxbMCwxXSxbMiwzXV0=
\[\begin{tikzcd}
  E & A \\
  TM & M
  \arrow[from=1-1, to=1-2]
  \arrow[Rightarrow, from=1-1, to=2-1]
  \arrow[Rightarrow, from=1-2, to=2-2]
  \arrow[from=2-1, to=2-2]
\end{tikzcd}\]
The core of $E$ is $H=\ker v$, so we now view $v\colon A\ra \frak k$ as the induced splitting of the core $A\ra M$ of the tangent algebroid $TA\Ra TM$.\footnote{The fact that $\Gamma(A)\ra \Gamma(TM,TA), \alpha\mapsto\alpha_c$ is not a morphism of Lie algebras gives another perspective on why the induced core splitting $v\colon A\ra M$ is, in general, not a splitting of Lie algebroids, but merely of vector bundles.}

% \begin{remark}
%   Let us expand on what it means for an IM connection $E\subset TA$ to be a (wide) Lie subalgebroid. By definition, the Lie bracket on $TA$ is determined on the linear and core sections \cite{bialgebroids}*{Eq.\ (27)}, so the requirement that $h\colon TA\ra E\hookrightarrow TA$ is a Lie algebroid map covering $\id_A$ means:
%   \begin{enumerate}[label={(\roman*)}]
%     \item The bracket $[\cdot,\cdot]_{TA}$ of two linear sections of $E$ is a linear section of $E$: for any $\alpha,\beta\in\Gamma(A)$,
%       \[[h(\d\alpha),h(\d\beta)]_{TA}=h(\d[\alpha,\beta]_A).\footnote{We point out the obvious pitfall: the equation $h(\d\alpha)=\d(h\alpha)$ is nonsensical.}\]
%     \item The bracket of a linear section of $E$ with a core section of $E$ is a core section of $E$:
%     \[
%     [h(\d\alpha),(h\beta)_c]_{TA}=(h[\alpha,\beta])_c.
%     \]
%   \end{enumerate}
% \end{remark}
%By abuse of notation, we henceforth denote the maps $v^{TA}$ and $\nu$ also by the letter $v$. 
\subsubsection{Horizontal projection of exterior cochains}
\label{sec:derivation_horproj_weil}
Let us now identify the notion of horizontality (Definition \ref{defn:horizontal_inf}) for the alternative model of exterior cochains.
\begin{definition}
  Suppose $\frak k\subset A$ is a bundle of ideals of a Lie algebroid $A\Ra M$. A cochain $\omega\in \Gamma_{\ext}(\MM_q,\Lambda^p\AA_q^*)$ is said to be \textit{horizontal}, if for any $\xi\in\Gamma(\frak k)$ and $i=1,\dots,q$, there holds
  \[
  \iota_{\Z_i(\xi)}\omega=0,
  \]
  where $\Z_i(\xi)$ is a core section, see \eqref{eq:linear_and_core_sections}.
  We denote the space of horizontal exterior cochains by 
  \[\Gamma_{\ext}(\MM_q,\Lambda^p\AA_q^*)^\Hor\subset \Gamma_{\ext}(\MM_q,\Lambda^p\AA_q^*).\]
\end{definition}
That this forms a subcomplex (for a fixed $q$) is a consequence of the fact that the evaluation map is an isomorphism of cochain complexes which clearly maps horizontal exterior cochains to horizontal Weil cochains, which themselves form a subcomplex. For clarity, we prove this directly.
\begin{proposition}
  $\Gamma_{\ext}(\MM_q,\Lambda^p\AA_q^*)^\Hor$ is a subcomplex of $\Gamma_{\ext}(\MM_q,\Lambda^p\AA_q^*)$.
\end{proposition}
\begin{proof}
Assuming $\iota_{\Z_i(\xi)}\omega =0$, use the definition \eqref{eq:delta_vb} of $\delta$ to get
\begin{align*}
  (\iota_{\Z_i(\xi)}\delta\omega)(X^1,\dots,X^p)=\L_{\rho_{\AA_q}(\Z_i(\xi))}\omega(X^1,\dots,X^p)+\textstyle\sum_{j=1}^p\omega([\Z_i(\xi),X^j]_{\AA_q},X^1,\dots,\widehat{X^j},\dots,X^p),
\end{align*}
for any sections $X^i\in\Gamma(\MM_q,\AA_q)$. To see the first term vanishes, note that the anchor $\smash{\rho_{\AA_q}}$ is defined componentwise, and the anchor $\rho_{TA}$ is defined as the composition of the canonical involution on $T(TM)$ with $\d\rho\colon TA\ra T(TM)$. Hence, it suffices to compute:
\[
\d\rho\big({\d 0(w)}+_A \smallderiv\lambda 0\lambda\xi\big)=\d\rho(\smallderiv\lambda 0\lambda\xi)=\smallderiv\lambda 0 \lambda\rho(\xi)=0,
\]
since $\frak k\subset\ker\rho$. For the remaining terms, using the fact that the Lie bracket on $\AA_q$ is defined componentwise, together with the definition of the Lie bracket on $TA$, we first observe
\[
[\Z_i(\xi),\Z_j(\alpha)]_{\AA_q}=0,\quad [\Z_i(\xi),\T(\alpha)]_{\AA_q}=\Z_i([\xi,\alpha]),
\]
for any $\alpha\in\Gamma(A)$. Notice that $[\xi,\alpha]\in\Gamma(\frak k)$ on account of $\frak k$ being a bundle of ideals, hence by multilinearity of $\omega$ with respect to $\AA_q\ra \MM_q$ and the fact that $\Gamma(\MM_q,\AA_q)$ is generated by the linear and core sections as a $C^\infty(\MM_q)$-module, the second term also vanishes.
\end{proof}
There is now a straightforward way of defining the horizontal projection of exterior cochains, given an IM connection for $\frak k$.
\begin{definition}
  \label{def:hor_proj_ext}
  Let $E\subset TA$ be an IM connection for a bundle of ideals $\frak k$ on $A\Ra M$, and let us denote by 
  $h\colon TA\ra E\hookrightarrow TA$ 
  the VB-algebroid morphism over $\id_A$ and $\id_{TM}$, 
  \begin{align}
    \label{eq:horizontal_proj}
    h(X)=X-_A v^{TA}(X).
  \end{align}
  The \textit{horizontal projection} of exterior cochains is the map
\begin{align*}
  &h^*\colon \Gamma_{\ext}(\MM_q,\Lambda^p\AA_q^*)\ra \Gamma_{\ext}(\MM_q,\Lambda^p\AA_q^*)^\Hor,\\
  &(h^*\omega)_{(w_1,\dots,w_q,\zeta)}(X^1,\dots,X^p)=\omega_{(w_1,\dots,w_q,\zeta)}(hX^1,\dots,h X^p),
\end{align*}
for any vectors $X^i=(X^i_1,\dots,X^i_q,(a^i,\zeta))$ with $\d\pi(X^i_j)=w_j\in T_x M$ and $a^i\in A_x$, where $h\colon \AA_q\ra \AA_q$ is the VB-algebroid morphism covering $\id_A$ and $\id_{\MM_q}$, induced by \eqref{eq:horizontal_proj}. Namely, $h\colon \AA_q\ra \AA_q$ is defined on any vector $X^i$ as above by
\[hX^i=(hX_1^i,\dots,hX_q^i,(a^i,\zeta)).\] 
\end{definition}
\begin{remark}
  We will use the same letter $h$ (resp., $v$) for all three horizontal (resp., vertical) projections of vectors in $\AA_q$, $TA$, and $A$, as it will always be contextually clear which one is used.
\end{remark}
What follows is the most important property of $h^*$, the proof of which is now almost trivial. Theorem \ref{thm:derivation_hor_proj} will imply that the same holds for Weil cochains, where proving this is  tedious.
\begin{proposition}
  \label{prop:h_cochain_ext}
  Let $E\subset TA$ be an IM connection for a bundle of ideals $\frak k$ on $A\Ra M$. The horizontal projection of exterior cochains is a cochain map, that is:
  \[\delta h^*=h^*\delta.\]
\end{proposition}
\begin{proof}
  Inspecting the defining equation \eqref{eq:delta_vb} of $\delta$, we observe that this is a direct consequence of $h\colon  \AA_q\ra \AA_q$ being a Lie algebroid morphism, since this just means
  \begin{align*}
    \rho_{\AA_q}(h X)=\rho_{\AA_q}(X),\quad h[X,Y]_{TA}=[hX,hY]_{TA},
  \end{align*}
  for any sections $X,Y\in\Gamma(\MM_q,\AA_q)$.
\end{proof}
% \begin{remark}
%   In fact, we observe that a (wide) double vector subbundle $E\subset TA$ is a subalgebroid of $TA$ if and only if $h^*$ is a cochain map.
% \end{remark}

We can now start deriving the formula for $h^*$ on Weil cochains. Looking at equation \eqref{eq:def_exterior_weil_generators}, we see that to do so, we need to see how the horizontal projection acts on the generators of the module $\Gamma(\MM_q,\AA_q)$. The following lemma shows how to express the horizontal projection of any generator as a linear combination of the generators, with respect to the structure on $\AA_q\ra\MM_q$.
\begin{lemma}
  \label{lem:h_on_generators}
  Let $E\subset TA$ be an IM connection for a bundle of ideals $\frak k$ on $A$. For any section $\alpha\in\Gamma(A)$, the linear and core sections $\T\alpha,\Z_i\alpha\in\Gamma(\MM_q,\AA_q)$ satisfy:
  \begin{align*}
    h(\T\alpha)(w_1,\dots,w_q,\zeta)&=\T\alpha(w_1,\dots,w_q,\zeta)-_{\MM_q}\textstyle\sum_i^{\MM_q}\Z_i(\C\alpha(w_i))(w_1,\dots,w_q,\zeta),\\
    h(\Z_i\alpha)&=\Z_i(h\alpha),
  \end{align*}
  for any vectors $w_j\in T_x M$ and $\zeta\in V_x^*$. Here, the superscript on the sum indicates that the summation is with respect to $\AA_q\ra {\MM_q}$.
\end{lemma}
\begin{proof}
  The formula for the core sections is clear from the definition of $\Z_i\alpha$ and equation \eqref{eq:v_core_section}. For the linear sections, we first write
  \begin{align*}
    h(\T\alpha)(w_1,\dots,w_q,\zeta)&=(h\d\alpha(w_1),\dots,h\d\alpha(w_q),\chi_\alpha(\zeta)),
  \end{align*}
  so we have to compute $h\d\alpha(w_i)$ for any $w_i\in T_xM$. Observe there holds
  \begin{align}
    h\d\alpha(w_i)&=\d\alpha(w_i)-_A v\d\alpha(w_i)\nonumber\\
    &=\big({\d\alpha(w_i)}+_{TM} 0_{TM}(w_i)\big)-_A \big(0_A(\alpha_x)+_{TM}\smallderiv\lambda0 \lambda \C(\alpha)(w_i)\big)\nonumber\\
    &=\big({\d\alpha(w_i)}-_{A} 0_{A}(\alpha_x)\big)+_{TM}\big(0_{TM}(w_i)-_A \smallderiv\lambda0 \lambda \C(\alpha)(w_i)\big)\nonumber\\
    &=\d\alpha(w_i)-_{TM}\big(0_{TM}(w_i)+_A \smallderiv\lambda0 \lambda \C(\alpha)(w_i)\big)\label{eq:horizontal_proj_dalpha_TM}
  \end{align}
  where we have used identities \eqref{eq:C_vb_picture} and \eqref{eq:identification_i_plus_tm} on the second equality and compatibility of the two vector bundle structures on $TA$ on the third. On the fourth, we have used that the two bundle structures coincide on the core $A\ra M$, so in particular, the two scalar multiplications coincide on the vector $\smallderiv\lambda 0 \lambda\C(\alpha)(w_i)$. Now, since the addition (over $\MM_q$) is componentwise, we obtain
  {\begin{align*}
    &h(\T\alpha)(w_1,\dots,w_q,\zeta)\\
    &=\T\alpha(w_1,\dots,w_q,\zeta)
    -_{\MM_q}\big(0_{TM}(w_1)+_A \smallderiv\lambda0 \lambda \C(\alpha)(w_1),\dots,0_{TM}(w_q)+_A \smallderiv\lambda0 \lambda \C(\alpha)(w_q), 0_\zeta \big)\\
    &=\T\alpha(w_1,\dots,w_q,\zeta)
    -_{\MM_q}\textstyle\sum_i^{\MM_q}\big(0_{TM}(w_1),\dots, 0_{TM}(w_i)+_A \smallderiv\lambda0 \lambda \C(\alpha)(w_i),\dots, 0_{TM}(w_q),0_\zeta\big),
  \end{align*}
  }from which the claimed formula for linear sections follows.
\end{proof}
\begin{remark}
  Intuitively, equation \eqref{eq:horizontal_proj_dalpha_TM} should be seen as a formula for how to horizontally project $\d\alpha$ within the fibre of $TA\ra TM$ instead of $TA\ra A$. As a consequence, it tells us how to horizontally project $\T\alpha$ within the fibre of $\AA_q\ra \MM_q$ instead of $\AA_q\ra A$.
  Henceforth, we will denote the so-called \textit{$\MM_q$-vertical component} of $\T\alpha$ by
  \begin{align*}
    u(\T\alpha)\in\Gamma(\MM_q,\AA_q),\quad u(\T\alpha)(w_1,\dots,w_q,\zeta)=\textstyle\sum_i^{\MM_q}\Z_i(\C\alpha(w_i))(w_1,\dots,w_q,\zeta),
  \end{align*}
  for any vectors $w_i\in T_x M$ and $\zeta\in V^*_x$.
\end{remark}
  To obtain some insight for deriving the general formula for $h^*$ on Weil cochains, let us first deal with the simplest nontrivial case: $p=1$. By the last lemma, for any $\omega\in \Gamma_\ext(\MM_q,\Lambda^1\AA_q^*)$,
  \begin{align*}
    \tilde c_0(h^*\omega)(\alpha)=\omega(h\T\alpha)=\omega(\T\alpha)-\omega(u\T\alpha),
  \end{align*}
  where the first term is just $\tilde c_0(\omega)(\alpha)$. The second term reads
  \begin{align}
    \omega(u\T\alpha)(w_1,\dots,w_q,\zeta)&=\textstyle\sum_i\omega\big(\Z_i(\C\alpha(w_i))\big)(w_1,\dots,w_q,\zeta)\nonumber\\
    &=\textstyle\sum_i (-1)^{i+1}\omega\big(\Z_1(\C\alpha(w_i))\big)(w_i,w_1,\dots,\widehat{w_i},\dots, w_q,\zeta)\label{eq:zi_expression_minus}\\
    &=\tfrac 1{(q-1)!}\textstyle\sum_{\sigma\in S_q} \sgn(\sigma)\omega\big(\Z_1(\C\alpha(w_{\sigma(1)}))\big)(w_{\sigma(1)},\dots, w_{\sigma(q)},\zeta),\label{eq:zi_to_z1}
  \end{align}
  where we have used the skew-symmetry of $\omega$ with respect to $\AA_q\ra A$ in the second equality. In the last line, we have merely rewritten the expression \eqref{eq:zi_expression_minus} with permutations because this form will be useful in the general proof. In any case, by Remark \ref{rem:any_vectors_ck}, the line \eqref{eq:zi_expression_minus} equals 
  \[
    \omega(u\T\alpha)(w_1,\dots,w_q,\zeta)=\textstyle\sum_i (-1)^{i+1}c_1(\omega)(\C\alpha(w_i))(w_1,\dots,\widehat{w_i},\dots,w_q,\zeta),
  \]
  which finally yields the formula \eqref{eq:h_p=1}, that is, \[c_0(h^*\omega)(\alpha)=c_0(\omega)(\alpha)-c_1(\omega)\wedgedot\C\alpha.\] At the level of symbols, $c_1(h^*\omega)(\beta)=c_1(\omega)(h\beta)$ is also clear from the last lemma. We now provide the desired derivation for $h^*$ on Weil cochains for the general case $p\geq 1$, preliminarily noting that exactly the same ideas will be used as above, though the procedure will be combinatorially heavier.
\begin{theorem}
  \label{thm:derivation_hor_proj}
  Let $E\subset TA$ be an IM connection for a bundle of ideals $\frak k$ on $A$. The model isomorphism $\ev\colon \Gamma_\ext(\MM_q,\Lambda^p\AA_q^*)\ra W^{p,q}(A;\frak k)$ commutes with the horizontal projections,
  \begin{align*}
    \ev\circ h^*=h^*\circ \ev,
  \end{align*}
  where $h^*$ on the left and the right side correspond to Definitions \ref{def:hor_proj_ext} and \ref{def:hor_proj}, respectively.
\end{theorem}
\begin{proof}
  Let us first deal with the leading term. For $\omega\in\Gamma_\ext(\MM_q,\Lambda^p\AA_q^*)$, we begin by computing
\begin{align}
  \label{eq:grouped_rearranged}
\begin{split}
  &\tilde c_0(h^*\omega)(\alpha_1,\dots,\alpha_p)=\omega(h\T\alpha_1,\dots,h\T\alpha_p)=\omega\big(\T\alpha_1-_{\MM_q}u\T\alpha_1,\dots,\T\alpha_p-_{\MM_q}u\T\alpha_p\big)\\
  &=\textstyle\sum_{j=0}^p(-1)^j\textstyle\sum\limits_{\mathclap{\qquad\sigma\in S_{(j,p-j)}}}\sgn(\sigma)\,\omega(u\T\alpha_{\sigma(1)},\dots, u\T\alpha_{\sigma(j)},\T\alpha_{\sigma(j+1)},\dots,\T\alpha_{\sigma(p)}),
\end{split}
\end{align}
  where we have used multilinearity of $\omega$ with respect to $\AA_q\ra\MM_q$ and grouped the terms with the same number of $\MM_q$-vertical arguments. The arguments have also been rearranged using skew-symmetry with respect to $\AA_q\ra\MM_q$, so that the $\MM_q$-vertical arguments appear first. To establish the theorem at the level of leading terms, we thus have to show
  \begin{align}
    \label{eq:omega_cj_wedgedot}
    \omega\big(u\T\alpha_1,\dots,u\T\alpha_j,\T\alpha_{j+1},\dots,\T\alpha_p\big)=c_j(\omega)(\alpha_{j+1},\dots,\alpha_p)\wedgedot (\C\alpha_1,\dots,\C\alpha_j).
  \end{align}
  The plan now is to transform each of the sections $u\T\alpha_i$ on the left-hand side into a core section of $\AA_q\Ra\MM_q$. We do this step by step, starting with $u\T\alpha_1$ and proceeding towards the right. For any vectors $w_1,\dots,w_q\in T_x M$ and $\zeta\in V^*_x$, we have
  \begin{align}
    &\omega\big(u\T\alpha_1,\dots,u\T\alpha_j,\T\alpha_{j+1},\dots,\T\alpha_p\big)(w_1,\dots,w_q,\zeta)\label{eq:omega_u_Talpha}\\
    &=\omega\big(\textstyle\sum_i\Z_i(\C\alpha_1(w_i)),u\T\alpha_2,\dots,u\T\alpha_j,\T\alpha_{j+1},\dots,\T\alpha_p\big)(w_1,\dots,w_q,\zeta)\nonumber\\
    &=\tfrac 1{(q-1)!}\textstyle\sum_{{\sigma\in S_q}}\sgn(\sigma) \omega\big(\Z_1(\C\alpha_1(w_{\sigma(1)})),\underbrace{u\T\alpha_2}_{\mathclap{\textstyle\sum_{i\geq 2}\Z_i(\C\alpha_2(w_{\sigma(i)}))}},\dots,u\T\alpha_j,\T\alpha_{j+1},\dots,\T\alpha_p\big)(w_{\sigma(1)},\dots,w_{\sigma(q)},\zeta),\nonumber
  \end{align}
  where we have used Lemma \ref{lem:h_on_generators} on the first equality and a similar identity as \eqref{eq:zi_to_z1} on the second equality (see Intermezzo \ref{intermezzo:intermediate_eq} below). Importantly, the under-brace follows from
  \begin{align}
    \label{eq:Z_Z_omega}
    \iota_{\Z_i\alpha}\iota_{\Z_i\beta}\omega=0,
  \end{align}
  for any $\alpha,\beta\in\Gamma(A)$, as a consequence of multilinearity of $\omega$ with respect to both vector bundle structures, and the general fact that linear differential forms on vector bundles vanish whenever two vertical vectors are inserted.
 \begin{intermezzo}
  \label{intermezzo:intermediate_eq}
   We observe that the following generalization of the equation \eqref{eq:zi_to_z1} holds:
   for any $\omega\in\Gamma_\ext(\MM_q,\Lambda^p\AA_q^*)$ and for any given $\ell\geq 1$, skew-symmetry of $\omega$ with respect to $\AA_q\ra A$ implies
 \begin{align*}
 \begin{split}
       &\omega\big(\textstyle\sum_{i\geq \ell}^{\MM_q}\Z_i(\C\alpha(w_i)),X^2,\dots,X^p\big)(w_1,\dots,w_q,\zeta)\\
       &=\textstyle\sum_{i\geq \ell}(-1)^{i-\ell}\omega\big(\Z_\ell(\C\alpha(w_i)),X^2,\dots,X^p\big)(w_1,\dots,w_{\ell-1},w_i,w_\ell,\dots,\widehat{w_i},\dots,w_q,\zeta)\\
       &=\tfrac 1{(q-\ell)!}\textstyle\sum_{{\sigma\in S_q^{\ell-1}}}\sgn(\sigma)\omega\big(\Z_\ell(\C\alpha(w_{\sigma(\ell)})),X^2,\dots,X^p\big)(w_{\sigma(1)},\dots,w_{\sigma(q)},\zeta),
 \end{split}
 \end{align*}
 for any vectors $X^i\in\AA_q$ over $(w_1,\dots,w_q,\zeta)$, where $S_q^{\ell-1}\subset S_q$ denotes the permutations which restrict to the identity on $\set{1,\dots, \ell-1}$. In the second line, we have used Lemma \ref{lem:h_on_generators}.
 \end{intermezzo}
\noindent By the intermezzo, the computation \eqref{eq:omega_u_Talpha} continues as
\begin{align*}
  &\eqref{eq:omega_u_Talpha}=\tfrac 1{(q-1)!(q-2)!}
    \textstyle\sum_{{\sigma\in S_q}}
    \textstyle\sum_{{\tilde\sigma\in S_q^1}}
    \sgn(\sigma)\sgn(\tilde\sigma)\cdot
    \\
    %phantom
    &\hspace{1em}
    \cdot\omega\big(\Z_1(\C\alpha_1(w_{\sigma(1)})),\Z_2(\C\alpha_2(w_{\sigma\tilde \sigma(2)})),u\T\alpha_3,\dots,\T\alpha_{j+1},\dots,\T\alpha_p\big)(w_{\sigma\tilde\sigma(1)},w_{\sigma\tilde\sigma(1)},\dots,w_{\sigma\tilde\sigma(q)},\zeta).
\end{align*}
By introducing $\tau=\sigma\tilde\sigma$, the condition $\tilde\sigma(1)=1$ becomes $\tau(1)=\sigma(1)$, so we can replace the second sum $\sum_{{\tilde\sigma\in S_q^1}}$ with $\sum_{\tau\in S_q, \tau(1)=\sigma(1)}$. The summand is then independent of $\sigma$, and since there are $(q-1)!$ permutations $\sigma$ in $S_q$ with a fixed $\sigma(1)$, we get
\begin{align*}
  \eqref{eq:omega_u_Talpha}=\tfrac 1{(q-2)!}\textstyle\sum\limits_{\smash[b]{\mathclap{\tau\in S_q}}}\sgn(\tau)\omega\big(\Z_1(\C\alpha_1(w_{\tau(1)})),\Z_2(\C\alpha_2(w_{\tau(2)})),u\T\alpha_3,\dots,\T\alpha_p\big)(w_{\tau(1)},\dots,w_{\tau(q)},\zeta).
\end{align*}
%We now repeat this procedure on each of the remaining sections $u\T\alpha_i$, where in the $i$-th step, we use \eqref{eq:Z_Z_omega} and Intermezzo \ref{intermezzo:intermediate_eq}, the sum $\sum_{\tilde\sigma\in S^i_q}$ is replaced with $\sum_{\tau\in S_q,\tau|_{\set{1,\dots,i}}=\sigma|_{\set{1,\dots,i}}}$, and we similarly note that there are $(q-i)!$ many permutations in $S_q$ with fixed $i$ entries. We are left with
Repeating this procedure on each of the remaining sections $u\T\alpha_i$, we are left with
\begin{align*}
  \eqref{eq:omega_u_Talpha}&=\tfrac 1{(q-j)!}\textstyle\sum\limits_{\mathclap{\tau\in S_q}}\sgn(\tau)\omega\big(\Z_1(\C\alpha_1(w_{\tau(1)})),\dots,\Z_j(\C\alpha_j(w_{\tau(j)})),\T\alpha_{j+1}\dots,\T\alpha_p\big)(w_{\tau(1)},\dots,w_{\tau(q)},\zeta)\\
  &=\tfrac 1{(q-j)!}\textstyle\sum\limits_{\mathclap{\tau\in S_q}}\sgn(\tau)c_j(\omega)(\alpha_{j+1},\dots,\alpha_{p}\| \C\alpha_{1}(w_{\tau(1)}),\dots,\C\alpha_{j}(w_{\tau(j)}))(w_{\tau(j+1)},\dots,w_{\tau(p)}),
\end{align*}
where we have used Remark \ref{rem:any_vectors_ck} in the second equality. Comparing this expression with \eqref{eq:wedgedot_multiple} proves the wanted identity \eqref{eq:omega_cj_wedgedot}, in turn proving the theorem at the level of leading terms. 

Finally, for correction terms, first observe that a similar identity as in \eqref{eq:grouped_rearranged} holds:
\begin{align*}
  &\tilde c_k(h^*\omega)(\alpha_1,\dots,\alpha_{p-k}\|\beta_1,\dots,\beta_k)=\omega(h\Z_1\beta_1,\dots,h\Z_k\beta_k, h\T\alpha_1,\dots,h\T\alpha_{p-k})\\
  &=\textstyle\sum_{j=k}^p(-1)^{j-k}\textstyle\sum\limits_{\mathclap{\qquad\sigma\in_{(j-k,p-j)}}}\sgn(\sigma)\,\omega(h\Z_1\beta_1,\dots,h\Z_k \beta_k,u\T\alpha_{\sigma(1)},\dots, u\T\alpha_{\sigma(j)},\T\alpha_{\sigma(j+1)},\dots,\T\alpha_{\sigma(p-k)}).
\end{align*}
We now have to prove the following equality for any sections $\alpha_i,\beta_j\in\Gamma(A)$:
\begin{align*}
  \omega(h\Z_1\beta_1,\dots,h\Z_k\beta_k,&u\T\alpha_1,\dots,u\T\alpha_j,\T\alpha_{j+1},\dots,\T\alpha_{p-k})\\
  &=c_j(\omega)(\alpha_{j+1},\dots,\alpha_{p-k}\|h\beta_1,\dots,h\beta_k)\wedgedot (\C\alpha_1,\dots,\C\alpha_j).
\end{align*}
This holds since $h(\Z_i\beta)=\Z_i(h\beta)$ for any $\beta\in\Gamma(A)$ by Lemma \ref{lem:h_on_generators}, by using the identical procedure as for the leading term.
\end{proof}

\section{Applications}
\label{sec:applications}
\subsection{Obstruction to existence of multiplicative connections}
In this section, we develop the necessary and sufficient condition for the existence of (infinitesimal) multiplicative Ehresmann connections. This generalizes the case for groupoid extensions \cite{gerbes}*{Proposition 6.13}, and establishes the infinitesimal analogue. As we will see, the obstruction is fairly simple to obtain as soon as one considers the right cohomology---that of the horizontal subcomplexes of the Bott--Shulman--Stasheff and Weil complex which have been introduced in the previous sections.
\subsubsection{Global case}
\begin{definition}
Let $\frak k$ be a bundle of ideals on a Lie groupoid $G\rra M$. The \textit{horizontal cohomology} of $\frak k$-valued forms on $G$ is the cohomology of the horizontal subcomplex from Definition \ref{defn:horizontal}, that is,
\begin{align*}
H^{p,q}(G;\frak k)^\Hor\coloneqq H^p\big(\Omega^{\bullet,q}(G;\frak k)^\Hor,\delta\big).
\end{align*}
\end{definition}
We now construct the obstruction class for the existence of multiplicative connections for a bundle of ideals $\frak k$. Suppose $E\subset TG$ is any (not necessarily multiplicative) distribution on $G$ complementing $K$, and let $\omega\in \Omega^1(G;s^*\frak k)$ be the corresponding 1-form defined by equation \eqref{eq:omega}, i.e., $\omega$ is the vertical projection under the isomorphism $K\cong s^*\frak k$. Since $\delta\circ\delta=0$, we obtain a cocycle $\delta\omega$, which is horizontal, i.e., it vanishes on $K^{(2)}$:
\begin{align*}
\delta\omega(X,Y)&=\omega(Y)-\omega(\d m(X,Y))+\Ad_{h^{-1}}\omega(X)\\
&=\d(L_{h^{-1}})Y-\d{(L_{h^{-1}g^{-1}})}(\d(L_g)(Y)+\d(R_h)(X))+\Ad_{h^{-1}}\d(L_{g^{-1}})X=0,
\end{align*}
for any $X\in K_g$ and $Y\in K_h$ where $s(g)=t(h)$. Hence, we may define
\begin{align}
\label{eq:obs}
\obs_{\A(G;\frak k)}=[\delta \omega] \in H^{2,1}(G;\frak k)^\Hor.
\end{align}
The difference of two 1-forms $\omega$ and $\tilde\omega$, corresponding to different distributions, is horizontal since their restrictions to $K$ equal the Maurer--Cartan form, $\tilde\omega|_K=\omega|_K=\Theta_{MC}|_K$. Therefore, the class above is independent of the choice of the distribution $E$.
\begin{proposition}
\label{prop:existence}
A multiplicative Ehresmann connection on a Lie groupoid $G$ for a bundle of ideals $\frak k$ exists if and only if the class $\mathrm{obs}_{\A(G;\frak k)}$ vanishes.
\end{proposition}
\begin{proof}
Clearly, the existence of a multiplicative connection implies the vanishing of the obstruction class. Conversely, if the class \eqref{eq:obs} vanishes, there is a horizontal form $\alpha\in \Omega^1(G;s^*\frak k)^\Hor$ such that
\begin{align*}
\delta(\omega+\alpha)=0,
\end{align*}
where $\omega$ corresponds to a fixed distribution $E$, as above. This means that the form $\alpha$ corrects the connection $\omega$, making $\omega+\alpha$ a multiplicative Ehresmann connection.
\end{proof}
\begin{remark}
  This obstruction class can be used to provide alternative proofs of the following two results from \cite{mec}; this is currently a work in progress.
  \begin{itemize}
    \item Morita invariance of existence of multiplicative Ehresmann connections \cite{mec}*{Theorem 4.1}. This follows directly from the to-be-proved Morita invariance of horizontal cohomology. For Lie groupoid extensions, this was already proved in \cite{gerbes}*{Theorem 6.9}; our work in progress concerns generalizing this result to arbitrary bundles of ideals.
    \item Existence of multiplicative Ehresmann connections on proper Lie groupoids for arbitrary bundles of ideals \cite{mec}*{Theorem 4.2}. This statement is a direct consequence of a to-be-proved version of \textit{vanishing cohomology theorem} for (horizontal!)\ cohomology on proper Lie groupoids, with values in a representation. 
  \end{itemize}
\end{remark}

\subsubsection{Infinitesimal case}
\begin{definition}
Let $\frak k$ be a bundle of ideals on a Lie algebroid $A\Ra M$. The \textit{horizontal cohomology} of $\frak k$-valued Weil cochains on $A$ is the cohomology of the horizontal subcomplex from Definition \ref{defn:horizontal_inf}, that is,
\begin{align*}
H^{p,q}(A;\frak k)^\Hor\coloneqq H^p\big(W^{\bullet,q}(A;\frak k)^\Hor,\delta\big).
\end{align*}
\end{definition}
The construction of the obstruction class now proceeds as follows. Pick any triple of the following form: a splitting $v\colon A\ra \frak k$ of the short exact sequence \eqref{eq:splitting}, a linear connection on $\frak k\ra M$ and a tensor $U\in\Gamma(H^*\otimes T^*M\otimes \frak k)$, where $H=\ker v$ (for example, $U=0$). Combine the connection and the tensor to obtain a map $\C\colon \Gamma(A)\ra \Omega^1(M;\frak k)$,
\[
\C(\alpha)=\nabla(v\alpha)-U(h\alpha),
\]
which clearly satisfies $\C(f\alpha)=f\C(\alpha)+\d f\otimes v\alpha$, so it defines a Weil cochain $(\C,v)\in W^{1,1}(A;\frak k)$. Since $\delta\circ\delta=0$, we obtain a cocycle $\delta(\C,v)$, which is horizontal---indeed, by equation \eqref{eq:delta_p=1} we get
\begin{align*}
\delta(\C,v)_1(\alpha\|\xi)&=-[\alpha,v(\xi)]+v[\alpha,\xi]=0,
\end{align*}
and $\delta(\C,v)_2$ vanishes automatically since $2=k\geq q=1$. Hence, we may define
\begin{align}
\label{eq:obs_inf}
\obs_{\A(A;\frak k)}=[\delta(\C,v)]\in H^{2,1}(A;\frak k)^\Hor.
\end{align}
If $(\tilde \C,\tilde v)$ is another cochain obtained as above, the difference $(\tilde \C,\tilde v)-(\C,v)$ is clearly horizontal since both $\tilde v$ and $v$ are splittings of \eqref{eq:splitting}, hence the class above is independent of the choice of the triple $(v,\nabla,U)$. The proof of the following is analogous to that of Proposition \ref{prop:existence}.
\begin{proposition}
  \label{prop:existence_inf}
An IM connection on a Lie algebroid $A$ for a bundle of ideals $\frak k$ exists if and only if the class $\obs_{\A(A;\frak k)}$ vanishes. Moreover, if $A\Ra M$ is the algebroid of $G\rra M$ and $\frak k$ is a bundle of ideals on $G$, then the van Est map relates the two obstruction classes:
\[
\ve (\obs_{\A(G;\frak k)})=\obs_{\A(A;\frak k)}.
\]
\end{proposition}
\begin{proof}
  The proof of the first part of the proposition is analogous to the groupoid case. For the second part, note that if a form $\omega\in\Omega^1(G;s^*\frak k)$ satisfies $\omega|_K=\Theta_{MC}$, then the symbol of $(\C,v)=\ve(\omega)$ is given on any $\xi\in\Gamma(\frak k)$ by $v(\xi)=\omega(\xi^L)|_M=(s^*\xi)|_M=\xi$, so $v|_\frak k=\id_\frak k$, and hence
\begin{align}
    \ve(\obs_{\A(G;\frak k)})=\ve[\delta\omega]=[\ve(\delta\omega)]=[\delta(\ve(\omega))]=[\delta(\C,v)]=\obs_{\A(A;\frak k)}.\tag*\qedhere
\end{align}
\end{proof}

\subsection{Curvature and affine deformations}
\label{sec:curvature}
\subsubsection{Global case}
\begin{definition}
The \textit{curvature} $\Omega^\omega\in\Omega^2(G;s^*\frak k)$ of a multiplicative Ehresmann connection $\omega\in \A(G;\frak k)$ on a Lie groupoid $G\rra M$ is given by \[\Omega^\omega=\D{}^\omega\omega.\]
\end{definition}
In what follows we state several important properties of the curvature form, generalizing the ones already known from the theory of principal bundles. The following result is already established in \cite{mec}*{Propositions 2.22 and 2.24}, however, a direct proof of (i) below is now possible due to our results from \sec\ref{sec:hor_ext_cov_der}, and also a simpler proof of the structure equation in (iii) on account of our construction of $\nabla$ in Proposition \ref{prop:conn}. 
\begin{proposition}
	Let $\omega\in\A(G;\frak k)$ be a multiplicative Ehresmann connection on $G\rra M$. Its curvature satisfies the following properties:
	\begin{enumerate}[label={(\roman*)}]
		\item $\Omega^\omega$ is a multiplicative form.
		\item For any horizontal vector fields $X,Y\in \Gamma(E)$ there holds
		\begin{align}
		\label{eq:curv_horizontal}
			\d (L_g)_{1_{s(g)}}\Omega^\omega_g(X,Y)=h([X,Y]_g)-[X,Y]_g.
		\end{align}
		In particular, the curvature $\Omega^\omega$ vanishes if and only if $E$ is involutive.
		\item The structure equation for $\Omega^\omega$ holds:
		\begin{align}
		\label{eq:structure}
		\Omega^\omega=\d{}^{\nabla^s}\omega+\frac 12 [\omega,\omega]_{s^*\frak k}.
		\end{align}
		\item The \textit{Bianchi identity} for $\Omega^\omega$ holds:
		\begin{align}
		\label{eq:bianchi}
			\D{}^\omega \Omega^\omega = 0.
		\end{align}
	\end{enumerate}
\end{proposition}
\begin{remark}
	The bracket on $\Omega^\bullet(G;s^*\frak k)$ which appears in equation \eqref{eq:structure} is the bracket of forms induced by the Lie bracket on the bundle of Lie algebras $s^*\frak k\ra G$. Skew-symmetry of the
	 latter and graded commutativity of the wedge product ensure there holds
   \[[\alpha,\beta]_{s^*\frak k}=(-1)^{kl+1}[\beta,\alpha]_{s^*\frak k}.\] Moreover, the graded Leibniz rule applies: 
\begin{align}
\label{eq:d_omega_graded_leibniz}
  	 \d{}^{\nabla^s}[\alpha,\beta]_{s^*\frak k}=[\d{}^{\nabla^s}\alpha,\beta]_{s^*\frak k}+(-1)^{k}[\alpha,\d{}^{\nabla^s}\beta]_{s^*\frak k},
\end{align} 
which follows from the fact that the induced connection $\nabla$ preserves the Lie bracket on $\frak k$. %We omit the proof of equation \eqref{eq:d_omega_graded_leibniz} and just note that it is not hard to show it in local coordinates.
\end{remark}
\begin{proof}
Multiplicativity of $\Omega^\omega$ is a direct consequence of Theorem \ref{thm:deltaD}. To show (ii), we note that if $X,Y\in \Gamma(E)$, then there holds
\begin{align*}
\Omega^\omega(X,Y)&=\nabla^s_X\omega(Y)-\nabla^s_Y\omega(X)-\omega([X,Y])=-\omega([X,Y])
\end{align*}
and now use the defining equation $\eqref{eq:omega}$ of $\omega$ to conclude (ii). %Multiplicativity of ! stated in (ii) is a consequence of d! mapping multiplicative forms to multiplicative forms, together with the fact that the horizontal projection h: TG!E is a functor. 

To prove (iii), first write out the right-hand side of equation \eqref{eq:structure}:
\begin{align*}
\Big({\d{}}^{\nabla^s}\omega&+\frac 12[\omega,\omega]_{s^*\frak k}\Big)(X,Y)=\nabla^s_X\omega(Y)-\nabla^s_Y\omega(X)-\omega([X,Y])+[\omega(X),\omega(Y)]_{s^*\frak k},
\end{align*}
for any vector fields $X,Y\in \vf (G)$. Note that both sides of \eqref{eq:structure} are $C^\infty(G)$-linear and moreover, any vertical vector can be extended to a vertical left-invariant vector field, and any horizontal vector can be extended to a horizontal $s$-projectable vector field on $G$. Hence, it is enough to consider all possible combinations of $X,Y$ being either vertical and left-invariant, or horizontal and $s$-projectable. If both $X$ and $Y$ are horizontal and $s$-projectable, both sides of \eqref{eq:structure} evaluate to $-\omega([X,Y])$, as already seen above. Furthermore, if both are vertical and left-invariant, i.e., $X=\xi^L$ and $Y=\eta^L$ for some $\xi,\eta\in \Gamma(\frak k)$, then both sides of \eqref{eq:structure} vanish since $[X,Y]=[\xi,\eta]^L$ and so $\omega([X,Y])=s^*[\xi,\eta]_{s^*\frak k}=[\omega(X),\omega(Y)]_{s^*\frak k}$. Finally, suppose that $X=\xi^L$ for some $\xi\in \Gamma(\frak k)$ and $Y\in \Gamma(E)$ is $s$-projectable to some vector field $s_*Y=U\in \vf(M)$. We need to show that the following identity holds:
\begin{align*}
\nabla^s_Y\omega(X)=\omega([Y,X]).
\end{align*}	
Notice that at any $g\in G$, we have
\begin{align*}
	\nabla^s_{Y}\omega(X)|_g&=\nabla^s_Y(s^*\xi)|_g=s^*(\nabla_U\xi)|_g=(g,v[Y,\xi^L]_{1_{s(g)}}),\\
	\omega_g([Y,X])&=(g,\d(L_{g^{-1}})_g(v[Y,X]_g)),
\end{align*}
where we have used the defining equation \eqref{eq:nabla} of $\nabla$ in the first line, and the defining equation \eqref{eq:omega} of $\omega$ in the second. These expressions coincide by left-invariance of $v[Y,\xi^L]$ from Proposition \ref{prop:conn} (ii).

The proof of the Bianchi identity is a matter of applying the structure equation:
%\begin{align*}
%	\D{}^\omega \Omega^\omega &(X,Y,Z)=\d{}^\omega\Omega^\omega (X,Y,Z)\\
%	&=\d{}^\omega (\d{}^\omega \omega+\frac 12[\omega,\omega]_{s^*\frak k})(X,Y,Z)=\big((\d{}^\omega)^2\omega + [\d{}^\omega\omega,\omega]_{s^*\frak k}\big)(X,Y,Z)\\
%	&=R^{\nabla^s}(X,Y)\omega(Z)+R^{\nabla^s}(Y,Z)\omega(X)+R^{\nabla^s}(Z,X)\omega(Y)=0,
%\end{align*}
\begin{align}
\label{eq:bianchi_intermediate}
  \d{}^{\nabla^s}\Omega^\omega=\d{}^{\nabla^s} (\d{}^{\nabla^s} \omega+\frac 12[\omega,\omega]_{s^*\frak k})=R^{\nabla^s}\wedge \omega+ [\d{}^{\nabla^s}\omega,\omega]_{s^*\frak k}
\end{align}
where we have denoted by $R^{\nabla^s}\in\Gamma(\Lambda^2(T^*G)\otimes \End(s^*\frak k))$ the usual curvature tensor of the pullback connection $\nabla^s$ on $s^*\frak k\ra G$. Inserting any three horizontal vectors into this equation now yields zero since $E=\ker \omega$.
\end{proof}
\begin{remark}
If we insert the structure equation into \eqref{eq:bianchi_intermediate} once more, we see that the Bianchi identity can also be written in the alternative form
\[
 \d{}^{\nabla^s}\Omega^\omega+[\omega,\Omega^\omega]_{s^*\frak k}=R^{\nabla^s}\wedge\omega.
\]
\end{remark}

\subsubsection{Infinitesimal case}
\label{sec:im_curvature}
In the rest of the paper, we use the simplified notation 
\[c(\alpha)=(c_0(\alpha),c_1(\alpha))\in \Omega^q(M;\frak k)\times \Omega^{q-1}(M;\frak k),\] for any Weil cochain $c=(c_0,c_1)\in W^{1,q}(A;\frak k)$, where $\alpha\in \Gamma(A)$.
\begin{definition}
Given a Lie algebroid $A$ with an IM connection $(\C,v)\in\A(A;\frak k)$, its \textit{curvature} is the horizontal IM form $\Omega^{(\C,v)}\in\Omega^2_{im}(A;\frak k)^\Hor$, defined as 
\[\Omega^{(\C,v)}=\D{}^{(\C,v)}(\C,v).\]
\end{definition} 
To obtain an explicit formula for the curvature, simply note that $(\d{}^\nabla(\C,v))_1=\C-\d{}^\nabla v$ restricts on $\frak k$ to zero, hence we get
\begin{align}
\begin{split}
  \label{eq:im_curv}
  \Omega^{(\C,v)}\alpha&=(\d{}^\nabla\C\alpha,\C h\alpha)=(R^\nabla\cdot v\alpha-\d{}^\nabla U(h\alpha),-U(h\alpha))
\end{split}
\end{align}
for any $\alpha\in\Gamma(A)$, where we have used the expression \eqref{eq:split_C} for $\C$ in terms of $\nabla$ and $U$. We observe that this coincides with \cite{mec}*{equation 5.9}, where a different approach for obtaining the formula for the curvature was used.

\begin{remark}
By definition, all the compatibility conditions for $\Omega^{(\C,v)}$ must follow from those for $(\C,v)$. More precisely, one can show that the condition \eqref{eq:c2} for $\Omega^{(\C,v)}$ is equivalent to \eqref{eq:s3} for $(\C,v)$; moreover, the condition \eqref{eq:c1} for $\Omega^{(\C,v)}$ may be split into three pieces, by considering $\Omega^{(\C,v)}[\alpha,\beta]$ for $\alpha$ and $\beta$ either horizontal or vertical sections:
\begin{itemize}[]%[label={$\circ$}]
\item When both are vertical, the condition translates to: the curvature tensor $R^\nabla$ takes values in the bundle $\Der(\frak k)\subset\End(\frak k)$ consisting of fibrewise derivations of $\frak k$. In other words, for any $\xi,\eta\in\frak k$ there holds
\begin{align}
\label{eq:R_derivation}
R^\nabla\cdot [\xi,\eta]=[R^\nabla\cdot\xi,\eta]+[\xi,R^\nabla\cdot\eta].
\end{align}
This is a direct consequence of condition \eqref{eq:s1}.
\item  When one is horizontal and one vertical, the condition reads $\iota_{\rho\alpha}\d{}^{\nabla^{\End\frak k}} R^\nabla=0$, where $\nabla^{\End\frak k}$ denotes the induced connection on $\End \frak k\ra M$. However, the Bianchi identity $\smash{\d{}^{\nabla^{\End\frak k}}} R^\nabla=0$ already holds for any connection $\nabla$, so this case is trivial.
\item When both are horizontal, one straightforwardly checks that the obtained condition is equivalent to the equation one gets when applying  $\d{}^\nabla$ to both sides of identity \eqref{eq:s3}.
\end{itemize}
\end{remark}

Since we have discovered the horizontal exterior covariant derivative of IM forms in equation \eqref{eq:D_inf}, it now becomes easy to establish the infinitesimal version of the Bianchi identity \eqref{eq:bianchi} for the curvature of an IM connection.
\begin{theorem}[Infinitesimal Bianchi identity] 
  \label{thm:bianchi_inf}
The horizontal exterior covariant derivative of the curvature of any IM connection $(\C,v)\in \A(A;\frak k)$ on a Lie algebroid $A$ vanishes:
\begin{align}
\label{eq:bianchi_inf}
  \D{}^{(\C,v)}\Omega^{(\C,v)}=0.
\end{align}
\end{theorem}
\begin{proof}
Using the explicit expression \eqref{eq:im_curv} for the curvature, we first compute:
\begin{align*}
  (\d{}^\nabla\Omega^{(\C,v)})_1(\alpha)=\d{}^\nabla\C\alpha-\d{}^\nabla\C (h\alpha)=\d{}^\nabla \C(v\alpha)=R^\nabla\cdot v\alpha,
\end{align*}
for any $\alpha\in \Gamma(A)$. Hence,
\begin{align*}
  \D{}^{(\C,v)}\Omega^{(\C,v)}\alpha=(\underbrace{\d{}^\nabla\d{}^\nabla \C\alpha}_{R^\nabla\wedge\:\C\alpha}-R^\nabla\wedgedot\C\alpha,R^\nabla\cdot v(h\alpha))=0,
\end{align*}
where we have used the relation \eqref{eq:wedge_wedgedot} between $\dot\wedge$ and $\wedge$.
\end{proof}
%\todo[inline]{Compute the square of D.}

\SkipTocEntry\subsection*{Affine deformations of multiplicative connections}
\subsubsection{Global case}
The set $\A(G;\frak k)$ of multiplicative Ehresmann connections is an affine space modelled on the vector space $\Omega^1_m(G;s^*\frak k)^\Hor$ of horizontal multiplicative 1-forms, so a natural question is how the curvature changes as we make an affine deformation \[\omega\ra \omega+\lambda \alpha\]
of a multiplicative  connection $\omega \in\A(G;\frak k)$ by a horizontal multiplicative 1-form $\alpha\in \Omega_m^1(G;\frak k)^\Hor$, scaled by a factor $\lambda\in \R$. That is, we aim to obtain an expansion of $\Omega^{\omega+\lambda\alpha}$ in terms of the scalar $\lambda$. 
\begin{theorem}
\label{thm:expansion}
	For any multiplicative connection $\omega\in \A(G;\frak k)$ on a Lie groupoid $G\rra M$, we have
	\begin{align}
	\label{eq:expansion}
		\Omega^{\omega+\lambda\alpha}=\Omega^\omega+\lambda \D{}^\omega\alpha+\lambda^2\mathbf c_2(\alpha),
	\end{align}
	where $\alpha\in \Omega^1_m(G;\frak k)^\Hor$ is any horizontal multiplicative form and $\lambda\in\R$. Here, the map \[\mathbf c_2\colon \Omega^1_m(G;\frak k)^\Hor\rightarrow \Omega^2_m(G;\frak k)^\Hor\] is homogeneous of degree two, and %defines an intrinsic operator on horizontal multiplicative 1-forms, i.e.\ it is 
independent of the connection $\omega\in\A(G;\frak k)$.
\end{theorem}
In the proof, an explicit formula for the coefficient $\mathbf c_2$ will also be given. The theorem is proved in several steps, starting with the following observation.
\begin{lemma}
\label{lem:tau_alpha}
Let $\frak k$ be a bundle of ideals on a Lie groupoid $G\rra M$. Any horizontal 1-form $\alpha\in\Omega^1(G;K)^\Hor$ determines a 1-form $\tau^\alpha\in\Omega^1(G;\End(K))$,
\[
\tau^\alpha(X)Y=\alpha[X,Y]+\nabla^s_Y\alpha(X),
\]
for any $X\in\vf(G)$ and $Y\in\Gamma(K)$. In turn, the form $\tau^\alpha$ defines an operator $\Omega^\bullet (G;K)\ra \Omega^{\bullet+1}(G;K)$, $\beta\mapsto \tau^\alpha\wedge\beta$, given on any $q$-form $\beta$ as
\[
(\tau^\alpha\wedge \beta)(X_i)_{i=0}^q=\sum_i(-1)^i\tau^\alpha(X_i)\cdot\beta(X_0,\dots,\hat{X_i},\dots,X_q).
\]
\end{lemma}
\begin{proof}
First note that $\nabla^s$ denotes the intrinsic $K$-connection on $K$ (see Remark \ref{rem:nabla_s}), so we are not fixing any multiplicative Ehresmann connection. The only thing to show is $C^\infty(G)$-linearity in both arguments of $\tau^\alpha$, which is simple and left to the reader. 
\end{proof}
What follows is a lemma relating the exterior covariant derivatives induced by two different multiplicative Ehresmann connections.
\begin{lemma}
\label{lem:diff_ext_cov}
	Let $G\rra M$ be a Lie groupoid with a multiplicative Ehresmann connection $\omega\in \A(G;\frak k)$. For any horizontal multiplicative 1-form $\alpha\in \Omega^1_m(G,\frak k)^\Hor$ and any form $\beta\in \Omega^\bullet(G;s^*\frak k)$, there holds
	\begin{align}
	\begin{split}
		\d{}^{\omega+\alpha}\beta-\d{}^\omega\beta=\tau^\alpha\wedge\beta\label{eq:diff_ext_cov}
		- [\alpha,\beta]_{s^*\frak k},
	\end{split}
	\end{align}
	with $\tau^\alpha\in \Omega^1(G;s^*\End(\frak k))$ under the identification $s^*\frak k\cong K$ of vector bundles.
\end{lemma}
\begin{remark}
	We are denoting $\d{}^\omega\coloneqq\d{}^{\nabla^s}$ where $\nabla$ is the linear connection on $\frak k$ from Proposition \ref{prop:conn} induced by $\omega$. In fact, in the proof below, the connection $\nabla$ induced by $\omega$ will be denoted $\nabla^\omega$. Moreover, we will denote the corresponding multiplicative distribution by $E^\omega=\ker \omega$, and the corresponding horizontal and vertical projections by $h_\omega\colon TG\ra E^\omega$ and $v_\omega\colon TG\ra K$, respectively. Throughout, we will be using the isomorphism $s^*\frak k\cong K$ of vector bundles as defined by \eqref{eq:iso_pullback_k}.
\end{remark}
\begin{proof}
	By the definitions of the exterior covariant derivatives, we have
	\begin{align}
	\begin{split}
	\label{eq:diff_ext_der_1}
	(\d{}^{\omega+\alpha}\beta&-\d{}^\omega\beta)(X_i)_{i=0}^q=\textstyle\sum_{i}(-1)^i(s^*\nabla^{\omega+\alpha}-s^*\nabla^\omega)_{X_i}\beta(X_0,\dots,\hat{X_i},\dots,X_q)
	\end{split}
	\end{align}
	for any vector fields $X_i\in \vf(G)$. First note that the difference of linear connections is an endomorphism-valued 1-form, in our case:
	\[
	s^*\nabla^{\omega+\alpha}-s^*\nabla^\omega\in \Omega^1(G;\End(s^*\frak k)).
	\]
	So, let us evaluate the expression 
	\begin{align*}
		(s^*\nabla^{\omega+\alpha}-s^*\nabla^\omega)_YX,
	\end{align*}
	for $Y\in \vf(G)$ an $s$-projectable vector field to $s_*Y=U\in \vf(M)$, and $X=s^*\xi$ the pullback of a section $\xi\in \Gamma(\frak k)$. In this case, by definition of the pullback of a linear connection, we have
	\[
	(s^*\nabla^{\omega+\alpha}-s^*\nabla^\omega)_YX=s^*(\nabla^{\omega+\alpha}_U\xi - \nabla^\omega_U\xi),
	\]
	so we inspect the difference of the covariant derivatives on the right-hand side. By Proposition \ref{prop:conn} (ii), there holds
	\begin{align*}
	(&\nabla^{\omega+\alpha}_U\xi - \nabla^\omega_U\xi)^L\\
  &=v_{\omega+\alpha}[h_{\omega+\alpha}(Y),\xi^L]-v_{\omega}[h_{\omega}(Y),\xi^L]=v_{\omega+\alpha}[Y-v_{\omega+\alpha}(Y),\xi^L]-v_{\omega}[Y-v_{\omega}(Y),\xi^L]\\
	&=(v_{\omega+\alpha}-v_\omega)[Y,\xi^L]-[(v_{\omega+\alpha}-v_\omega)Y,\xi^L]=\bar\alpha[Y,\xi^L]-[\bar\alpha(Y),\xi^L],
	\end{align*}
	where we have denoted by $\bar\alpha\in\Omega^1(G;K)$ the form $\alpha$ under the identification $s^*\frak k\cong K$, i.e., $\bar\alpha=v_{\omega+\alpha}-v_\omega$. We now want to write this as a section of $s^*\frak k$, and rewrite the second term using the bracket $[\cdot,\cdot]_{s^*\frak k}$. To do so, we first note that the canonical isomorphism $s^*\frak k\cong K$ of vector bundles, given in \eqref{eq:iso_pullback_k}, is not an isomorphism of Lie algebroids---in fact, the anchor of $K$ is the inclusion and the anchor of $s^*\frak k$ is zero. The relation between their brackets is the following: $[\cdot,\cdot]_{s^*\frak k}$ is the torsion of the canonical $K$-connection $\nabla^s$ on $K$ (up to a sign). Indeed, note that on left-invariant sections, $[\cdot,\cdot]$ on $K$ agrees with $[\cdot,\cdot]_{s^*\frak k}$ on $s^*\frak k$ (up to the identification $s^*\frak k\cong K$ as vector bundles), hence it is an application of the Leibniz rule of $[\cdot,\cdot]$ to see that
	\[
	\left[\xi_1,\xi_2\right]=[\xi_1,\xi_2]_{s^*\frak k}+\nabla^s_{\xi_1} \xi_2-\nabla^s_{\xi_2} \xi_1,
	\]
	for any $\xi_1,\xi_2\in \Gamma(K)$.  Using this relation, we can rewrite $(\nabla^{\omega+\alpha}_U\xi - \nabla^\omega_U\xi)^L$ as:
	\begin{align*}
  s^*(\nabla^{\omega+\alpha}_U\xi - \nabla^\omega_U\xi)&=\alpha[Y,\xi^L]-[\alpha(Y),s^*\xi]_{s^*\frak k}+\nabla^s_{\xi^L}\alpha(Y)\\
  &=\tau^\alpha(Y)s^*\xi-[\alpha(Y),s^*\xi]_{s^*\frak k},
\end{align*}
where we have used that $\nabla^s$ vanishes on pullback sections in the first line, and Lemma \ref{lem:tau_alpha} in the second. Due to $C^\infty(G)$-linearity of the difference $s^*\nabla^{\omega+\alpha}-s^*\nabla^\omega$, this allows us to rewrite \eqref{eq:diff_ext_der_1} as
	\begin{align*}
		(\d{}^{\omega+\alpha}\beta-\d{}^\omega\beta)(X_i)_{i=0}^q&=\textstyle\sum_{i}(-1)^i\tau^\alpha(X_i)\beta(X_0,\dots,\hat{X_i},\dots,X_q)\\
		&- \textstyle\sum_{i}(-1)^i[\alpha(X_i),\beta(X_0,\dots,\hat{X_i},\dots,X_q)]_{s^*\frak k}.
	\end{align*}
The first term is just $\tau^\alpha\wedge \beta$ in equation \eqref{eq:diff_ext_cov}, and the second term is $[\alpha,\beta]_{s^*\frak k}$
with a bit of combinatorics: since $\alpha$ is a 1-form, the bracket of forms $[\alpha,\beta]_{s^*\frak k}$ reads
	\begin{align*}
  [\alpha,\beta]_{s^*\frak k}(X_0,\dots,X_k)&=\frac1{k!}\textstyle\sum_{{\sigma\in S_{k+1}}}\sgn(\sigma)[\alpha(X_{\sigma(0)}),\beta(X_{\sigma(1)},\dots,X_{\sigma(k)})]_{s^*\frak k}\\
	&=\frac 1{k!}\textstyle\sum_{i}\textstyle\sum_{{\substack{\sigma\in S_{k+1} \\ \sigma(0)=i}}}\sgn(\sigma)[\alpha(X_i),\beta(X_{\sigma(1)},\dots,X_{\sigma(k)})]_{s^*\frak k}\\
  &=\textstyle\sum_{i}(-1)^i[\alpha(X_i),\beta(X_0,\dots,\hat{X_i},\dots,X_k)]_{s^*\frak k},
	\end{align*}
	where we have used that the number of permutations from $S_{k+1}$ with $\sigma(0)=i$ equals $k!$, and observed that given any such $\sigma$, ordering the arguments of $\beta(X_{\sigma(1)},\dots,X_{\sigma(k)})$ yields an additional factor of $(-1)^i\sgn\sigma$, concluding our proof.
\end{proof}
\begin{proof}[Proof of Theorem \ref{thm:expansion}]
	We first consider the case $\lambda=1$. The structure equation gives us
	\begin{align}
	\begin{split}
	\Omega^{\omega+\alpha}&=\d{}^{\omega+\alpha}(\omega+\alpha)+\frac12[\omega+\alpha,\omega+\alpha]_{s^*\frak k}
	\label{eq:structure_perturbed}\\
	&=\d{}^{\omega+\alpha}(\omega+\alpha)+\frac12[\omega,\omega]_{s^*\frak k}+[\omega,\alpha]_{s^*\frak k}+\frac {1}2[\alpha,\alpha]_{s^*\frak k},
	\end{split}
	\end{align}
	so we need to compute $\d{}^{\omega+\alpha}\omega$ and $\d{}^{\omega+\alpha}\alpha$. By Lemma \ref{lem:diff_ext_cov}, we have
	\begin{align*}
	\d{}^{\omega+\alpha}\omega-\d{}^\omega\omega&=\tau^\alpha\wedge \omega-[\alpha,\omega]_{s^*\frak k}=\D{}^\omega\alpha-\d{}^\omega\alpha-[\alpha,\omega]_{s^*\frak k},
	\end{align*}
	where we have used the equality
	\begin{align}
	\label{eq:diff_D_d}
		\D{}^\omega\alpha=\d{}^\omega\alpha+\tau^\alpha\wedge\omega
	\end{align}
	which holds due to horizontality of $\alpha$ and is a straightforward computation (see Remark \ref{rem:diff_D_d} after the proof). %\footnote{Equation \eqref{eq:diff_D_d} is a correction of the incorrect guess $\D{}^\omega\alpha=\d{}^\omega\alpha+[\omega,\alpha]_{s^*\frak k}$.} 
  On the other hand, using the lemma on $\alpha$ yields
		\begin{align*}
	\d{}^{\omega+\alpha}\alpha-\d{}^\omega\alpha=\tau^\alpha\wedge\alpha-[\alpha,\alpha]_{s^*\frak k}.
	\end{align*}
	When both equalities are plugged into the right-hand side of \eqref{eq:structure_perturbed}, the two terms with $\d{}^\omega\alpha$ cancel out, and likewise the terms containing $[\omega,\alpha]_{s^*\frak k}$, so we are left with
	\begin{align}
	\Omega^{\omega+\alpha}=\Omega^\omega+ \D{}^\omega\alpha+\big(\tau^\alpha\wedge\alpha-\frac 12[\alpha,\alpha]_{s^*\frak k}\big),
	\end{align}
	which proves the case $\lambda=1$. It is clear that the second-order coefficient 
\begin{align}
\label{eq:secondorder}
  \mathbf c_2(\alpha)=\tau^\alpha\wedge\alpha-\frac 12[\alpha,\alpha]_{s^*\frak k}
\end{align}
is independent of $\omega\in \A(G;\frak k)$, and since the forms $\Omega^{\omega+\alpha}, \Omega^\omega$ and $\D{}^\omega\alpha$ are horizontal and multiplicative by Theorem \ref{thm:deltaD}, so is $\mathbf c_2(\alpha)$. Since $\D{}^\omega$ is linear and $\mathbf c_2$ is clearly homogeneous of degree two, the general case $\lambda\in\R$ now also follows.
\end{proof}
\begin{remark}
\label{rem:diff_D_d}
More generally than in equation \eqref{eq:diff_D_d} above, it is straightforward to show (using the definition of the exterior covariant derivative $\d{}^\omega$) that for any horizontal form $\alpha\in\Omega^q(G;s^*\frak k)^\Hor$, there holds
	\begin{align*}
	(\D{}^\omega\alpha&-\d{}^\omega\alpha)(X_0,\dots,X_k)=-\textstyle\sum_i(-1)^i\nabla^s_{v_\omega(X_i)}\alpha(X_0,\dots,\hat{X_i},\dots,X_q)\\
	&-\textstyle\sum_{i<j}(-1)^{i+j}\alpha([X_i,v_\omega(X_j)]+[v_\omega(X_i),X_j],X_0,\dots,\hat{X_i},\dots,\hat{X_j},\dots,X_q),
	\end{align*}
	for any $X_i\in \vf (G)$. If $\alpha$ is a horizontal 1-form, this is just
\begin{align*}
  (\D{}^\omega\alpha-\d{}^\omega\alpha)(X,Y)&=\nabla^s_{v_\omega(Y)}\alpha(X)-\nabla^s_{v_\omega(X)}\alpha(Y)+\alpha[X,v_\omega (Y)]-\alpha[Y,v_\omega (X)]=(\tau^\alpha\wedge\omega)(X,Y),
\end{align*}
for any $X,Y\in \vf(G)$.
\end{remark}
\begin{remark}
Roughly speaking (since the spaces are infinite dimensional), we can read from the first order coefficient that the derivative of the map% (between infinite-dimensional spaces)
\[
\kappa\colon\A(G;\frak k)\ra \Omega^2_m(G;\frak k)^\Hor,\quad \kappa(\omega)=\Omega^\omega
\]
is given for any $\alpha\in\Omega^1_m(G;\frak k)^\Hor$ by
\[
\d \kappa_\omega(\alpha)=\D{}^\omega\alpha.
\]
\end{remark}

\subsubsection{Infinitesimal case}
Similarly to the global case, the set $\A(A;\frak k)$ of all IM connections is an affine space modelled on the vector space $\Omega^1_{im}(A;\frak k)^\Hor$ of horizontal IM 1-forms. We hence now discuss the infinitesimal analogue of Theorem 
\ref{thm:expansion}, i.e., we are going to inspect how the curvature of an IM connection changes as we make an affine deformation 
\begin{align*}
  (\C,v)\rightarrow (\C,v)+\lambda (L,l)
\end{align*}
of an IM connection $(\C,v)$ by a horizontal IM form $(L,l)\in\Omega^1_{im}(A,\frak k)^\Hor$ scaled by $\lambda\in\R$.

\begin{theorem}
\label{thm:expansion_inf}
On a Lie algebroid $A\Ra M$ with an IM connection $(\C,v)\in\A(A;\frak k)$, there holds
\begin{align}
\label{eq:expansion_inf}
  \Omega^{(\C,v)+\lambda(L,l)}=\Omega^{(\C,v)}+\lambda \D{}^{(\C,v)}(L,l)+\lambda^2\mathbf c_2(L,l),
\end{align}
for any horizontal IM form $(L,l)\in\Omega^1_{im}(A,\frak k)^\Hor$ and $\lambda\in \R$, where the second-order coefficient in the expansion is given by the map
\begin{align}
\begin{split}
\label{eq:secondorder_inf}
    &\mathbf c_2\colon\Omega^1_{im}(A,\frak k)^\Hor\rightarrow \Omega^2_{im}(A,\frak k)^\Hor,\\
  &\mathbf c_2(L,l)(\alpha)=-(L|_{\frak k}\wedgedot L\alpha, L|_{\frak k}\cdot l\alpha).
\end{split}
\end{align}
\end{theorem}
\begin{proof}
It is enough to show the identity 
\eqref{eq:expansion_inf} for the case $\lambda=1$, since $\mathbf c_2$ is homogeneous of degree two. Let us first fix some notation. Denote \[(\tilde\C,\tilde v)=(\C,v)+(L,l),\] so the relation between the respective horizontal projections is given by $\tilde h=h-l$, and the relation between induced linear connections is just 
$
\tilde \nabla = \nabla+ L|_{\frak k}.
$
The induced exterior derivatives are thus related by 
\[
\d{}^{\tilde\nabla}=\d{}^\nabla+L|_{\frak k}\wedge\cdot
\] 
where the wedge product is as in Remark \ref{rem:wedges}. For any $(J,j)\in\Omega^{1}_{im}(A;\frak k)$, let us simplify the notation for $(\d{}^\nabla(J,j))_1$ and simply denote it by $\ul J^\nabla$. It is not hard to see there holds
\[
  \ul J^{\tilde\nabla}\alpha = \ul J^\nabla\alpha - L|_{\frak k}\wedge j\alpha, 
\] 
for any $\alpha \in \Gamma(A)$. The task at hand now is to expand the expression \[
\D{}^{(\tilde \C,\tilde v)}(\tilde \C,\tilde v)=\D{}^{(\tilde \C,\tilde v)}(\C,v)+\D{}^{(\tilde \C,\tilde v)}(L,l),
\]
and we do this by means of a straightforward computation. The first term reads
\begin{align*}
  \D{}^{(\tilde \C,\tilde v)}(\C,v)\alpha &=(\d{}^{\tilde\nabla} \C\alpha-\ul\C^{\tilde\nabla}\wedgedot\tilde\C\alpha,\ul\C^{\tilde\nabla}\tilde h\alpha)\\
  &=(\d{}^\nabla\C\alpha+\bcancel{L|_{\frak k}\wedge\C\alpha} -\ul\C^{\nabla}\wedgedot \C\alpha + \bcancel{L|_{\frak k}\wedgedot\C\alpha}-\cancel{\ul\C^{\nabla}\wedgedot L\alpha}+L|_{\frak k}\wedgedot L\alpha,\\
  & \begingroup \color{white} =( \endgroup \ul\C^{\nabla}h\alpha-\cancel{\ul\C^{\nabla}l\alpha}-L|_{\frak k}\cdot \cancel{v(h\alpha)}+L|_{\frak k}\cdot l\alpha)\\
  &=\Omega^{(\C,v)}\alpha+ (L|_{\frak k}\wedgedot L\alpha,L|_{\frak k}\cdot l\alpha),
\end{align*}
where we have used the relation \eqref{eq:wedge_wedgedot} and observed that the restriction of $\ul\C^{\nabla}$ to $\frak k$ vanishes. On the other hand, the second term becomes
\begin{align*}
\D{}^{(\tilde \C,\tilde v)}(L,l)(\alpha)
&=(\d{}^{\tilde\nabla}L\alpha- \ul L^{\tilde\nabla}\wedgedot \tilde \C\alpha, \ul L^{\tilde\nabla}\tilde h\alpha)\\
&=(\d{}^\nabla L\alpha + L|_{\frak k}\wedge L\alpha - \ul L^{\nabla}\wedgedot \C\alpha - L|_{\frak k}\wedgedot L\alpha,%\\&\begingroup \color{white} =( \endgroup 
\ul L^\nabla h\alpha -L|_{\frak k}\cdot l\alpha - \ul L^{\nabla} l\alpha)\\
&=\D{}^{(\C,v)}(L,l)(\alpha)-2(L|_{\frak k}\wedgedot L\alpha,L|_{\frak k}\cdot l\alpha),
\end{align*}
where we have observed that $\ul L^{\tilde\nabla}$ coincides with $\ul L^{\nabla}$ when restricted to $\frak k$. Adding the two expressions together, we obtain the wanted identity
\begin{align*}
  \Omega^{(\tilde \C,\tilde v)}\alpha=\Omega^{(\C,v)}\alpha + \D{}^{(\C,v)}(L,l)\alpha-(L|_{\frak k}\wedgedot L\alpha, L|_{\frak k}\cdot l\alpha).\tag*\qedhere
\end{align*}
\end{proof}
The theorem may also be restated using the components $R^\nabla$ and $U$, appearing in the explicit expression \eqref{eq:im_curv} for the curvature $\Omega^{(\C,v)}$ of an IM connection $(\C,v)$. They transform with an affine deformation $(\tilde \C,\tilde v)=(\C,v)+\lambda(L,l)$ as
\begin{align}
R^{\tilde\nabla}\cdot \xi &=R^{\nabla}\cdot \xi+\lambda\d{}^{\nabla^{\End\frak k}}(L|_{\frak k})\cdot \xi+\lambda^2 L|_{\frak k}\wedge L\xi,\\
\tilde U\tilde h\alpha &=Uh\alpha+\lambda \ul L^\nabla h\alpha-\lambda^2 L|_{\frak k}\cdot l\alpha,
\end{align}
for any $\alpha \in \Gamma(A)$ and $\xi\in\Gamma(\frak k)$.

\begin{remark}
The second-order coefficient \eqref{eq:secondorder_inf} is just the infinitesimal counterpart of the one from Theorem \ref{thm:expansion}, defined in equation
\eqref{eq:secondorder}. More precisely, if a Lie groupoid $G$ integrates $A$, then the following diagram commutes. 
% https://q.uiver.app/#q=WzAsNCxbMCwxLCJcXE9tZWdhXjFfe2ltfShBLFxcZnJhayBrKV5cXEhvciJdLFsxLDEsIlxcT21lZ2FeMl97aW19KEEsXFxmcmFrIGspXlxcSG9yIl0sWzAsMCwiXFxPbWVnYV4xX20oRztzXipcXGZyYWsgayleXFxIb3IiXSxbMSwwLCJcXE9tZWdhXjJfbShHO3NeKlxcZnJhayBrKV5cXEhvciJdLFsyLDAsIlxcTGllIiwyXSxbMywxLCJcXExpZSJdLFswLDEsIlxcbWF0aGJmIGNfMiJdLFsyLDMsIlxcbWF0aGJmIGNfMiJdXQ==
\[\begin{tikzcd}[row sep=large]
	{\Omega^1_m(G;\frak k)^\Hor} & {\Omega^2_m(G;\frak k)^\Hor} \\
	{\Omega^1_{im}(A,\frak k)^\Hor} & {\Omega^2_{im}(A,\frak k)^\Hor}
	\arrow["{\mathbf c_2}", from=1-1, to=1-2]
	\arrow["\ve"', from=1-1, to=2-1]
	\arrow["\ve", from=1-2, to=2-2]
	\arrow["{\mathbf c_2}", from=2-1, to=2-2]
\end{tikzcd}\]
This follows from both theorems on affine deformations and the diagram \eqref{eq:square_D}. 
\end{remark}

\subsection{Primitive multiplicative connections and their curvings}
\label{sec:primitive}
%Motivated by the previous section, 
We now study the role of cohomological triviality for the theory of multiplicative connections. More precisely, we will focus on multiplicative connections with cohomologically trivial curvature. %In the case of groupoids, this means precisely that the cohomological class $[\Omega^\omega]\in H^{2,1}_\Hor(G;\frak k)$ vanishes, i.e., $[\omega]\in\ker \D{}^\omega$, where $\D{}^\omega$ is the map in cohomology induced by the horizontal exterior covariant derivative. 
This has already been briefly studied in \cite{gerbes} for the particular case of Lie groupoid extensions; we show that these results extend to the infinitesimal realm, and actually hold for arbitrary bundles of ideals. We will only develop the respective results for algebroids, and the respective results for groupoids will be omitted. The reason for this decision is that the restriction of the van Est map to the cohomologically trivial forms, \[\ve\colon \im\delta_G^0\ra\im\delta_A^0,\] is a bijection when $G$ is source-connected, since in this case the $G$-invariant forms on the base coincide with $A$-invariant forms---that is, $\ker\delta_G^0=\ker\delta_A^0$. 

\begin{definition}
Let $\frak k\subset A$ be a bundle of ideals of a Lie algebroid $A\Rightarrow M$. An IM connection $(\C,v)\in\A(A;\frak k)$ is said to be \textit{primitive} if $\Omega^{(\C,v)}$ is cohomologically trivial, i.e., if the
cohomological class $[\Omega^{(\C,v)}]\in H^{1,2}(A;\frak k)^\Hor$ vanishes.
We denote the set of primitive connections by \[\arc(A;\frak k)\subset \A(A;\frak k).\]
Given a primitive connection $(\C,v)$, a form $F\in\Omega^2(M;\frak k)$ satisfying
\begin{align}
\label{eq:curving}
\Omega^{(\C,v)}=\delta^0 F
\end{align}
is then called a \emph{curving} of the connection $(\C,v)$. The differential form \[G\coloneqq \d{}^\nabla F\] is then called the \textit{curvature 3-form} of $F$, where $\nabla=\C|_{\frak k}$ is the induced linear connection on $\frak k$.
\end{definition}
\begin{example}
If $V\ra M$ is a vector bundle, seen as a Lie algebroid with the trivial structure, and $\frak k=V$, then primitive connections are in a 1-to-1 correspondence with flat linear connections on $V$. This will follow as a particular case of Proposition 
\ref{prop:abelian_primitive}, however, the reader is encouraged to see this only by referencing equation \eqref{eq:R_ad_F}.
\end{example}
\begin{example}
If $A\Ra M$ is a transitive algebroid and $\frak k=\ker\rho$, IM connections are in a bijective correspondence with splittings of the abstract Atiyah sequence by condition \eqref{eq:c2}, and every IM connection admits a unique curving---the curvature of the splitting $v\colon A\ra\frak k$. It is defined as
\begin{align}
  \label{eq:curvature_splitting}
  F^v\in\Omega^2(M;\frak k),\quad F^v(X,Y)=\sigma[X,Y]-[\sigma(X),\sigma(Y)],
\end{align}
where $\sigma\colon TM\ra A$ is the horizontal lift along the anchor, associated to the splitting $v$. The fact that $F^v$ is the unique curving of the IM connection associated to $v\colon A\ra \frak k$ follows from the fact that $\delta^0\colon \Omega^q(M;\frak k)\ra \Omega^q_{im}(A;\frak k)^\Hor$ is canonically an isomorphism, whose inverse is \[(\delta^0)^{-1}(c_0,c_1)(v_1,\dots,v_q)=c_1(\tilde v_1)(v_2,\dots,v_q),\] where $\tilde v_1$ is an arbitrary lift of $v_1$ along the anchor. By combining this fact with the observation from Example \ref{ex:horizontal_p=1}, we also obtain a canonical isomorphism $\Omega^q_{im}(A;\frak k)\cong \Omega^q(M;\frak k)$ for all $q\geq 2$. We observe that the curvature 3-form of the curving $F^v$ vanishes by the Bianchi identity for the curvature of any splitting, $\d{}^\nabla F^v=0$.
\end{example}
Notice that a curving $F$ of a fixed IM connection $(\C,v)$ is in general not unique---it is only determined up to an invariant 2-form. In other words, the space of curvings $(\delta^0)^{-1}(\Omega^{(\C,v)})$ of a connection $(\C,v)$ is an affine space modelled on the vector space $\Omega^2_\inv(M;\frak k)=\ker\delta^0$ of invariant 2-forms on $M$. This means that any two curvings $\tilde F$ and $F$ must differ by a 2-form $\beta\in\Omega^2(M;\frak k)$ satisfying the following conditions:
\begin{align}
\label{eq:invariant}
\L^A_\alpha\beta=0,\quad\iota_{\rho(\alpha)}\beta=0, 
\end{align}
for any $\alpha\in \Gamma(A)$. The associated curvature 3-forms are then clearly related by
\begin{align}
\label{eq:3curv_diff}
\tilde G-G =\d{}^\nabla\beta.
\end{align}
\begin{remark}
\label{rem:invariant}
By the first condition in \eqref{eq:invariant}, any invariant form $\beta\in\Omega^\bullet(M;\frak k)$ is center-valued (take $\alpha\in\Gamma(\frak k)$). Furthermore, the second condition means that $\beta$ is \textit{transversal}, i.e.,\ it is a section of $\Lambda^\bullet(T\F)^\circ\otimes z(\frak k)$, so it must vanish wherever its degree is greater than the codimension of $\F$: for any $x\in M$, $\deg\beta>\operatorname{corank}\rho_x$ implies $\beta_x=0$. In particular, if $\codim\F= 1$ then the curving of any multiplicative connection is unique, if it exists.
\end{remark}
The following is the infinitesimal analogue of \cite{gerbes}*{Theorem 6.33}.
\begin{lemma}
\label{lem:curving}
Let $(\C,v)\in \arc(A;\frak k)$ be a primitive IM connection on a Lie algebroid $A\Ra M$, and let $F$ be its curving. The following statements hold.
\begin{enumerate}[label={(\roman*)}]
\item The curvature tensors $U$ and $R^\nabla$ are determined by the curving $F$ as
\begin{align}
\label{eq:R_ad_F}
R^\nabla\cdot\xi&=[\xi,F],\\
U(\alpha)&=-\iota_{\rho(\alpha)}F
\end{align}
for any $\xi\in\Gamma(\frak k)$ and $\alpha\in H$. %Moreover, $U(h\alpha)=-\iota_{\rho(\alpha)}F$.
\item Bianchi identities: the curvature 3-form $G$ satisfies
\begin{align}
\label{eq:3curv}
\delta^0 G=0 \quad\text{and}\quad \d{}^\nabla G=0.
\end{align}
Hence, $G$ is center-valued and transversal, so it vanishes wherever $\codim\F\leq 2$.
\end{enumerate}
\end{lemma}
\begin{proof}
The point (i) follows directly from using \eqref{eq:curving} with the explicit expression \eqref{eq:im_curv} for the curvature. For the second point, note that
\[
\delta^0 G=\delta^0 \d{}^\nabla F=\D{}^{(\C,v)}\delta^0F=\D{}^{(\C,v)}\Omega^{(\C,v)}=0,
\]
where we have used Theorem \ref{thm:deltaD_inf} and the Bianchi identity. Finally,
\[
\d{}^\nabla G=(\d{}^\nabla)^2 F=R^\nabla\wedge F=[F,F]=0,
\]
since $F$ is a 2-form, where we have used the point (i).
\end{proof}

Equation \eqref{eq:R_ad_F} of item (i) of the last lemma will henceforth simply be written as 
\[
R^\nabla=-\operatorname{ad}F,
\]
and it is stronger than the condition \eqref{eq:s2} for the coupling data $(\nabla,U)$ of a primitive connection. In a concealed way, the last lemma also tells us what \eqref{eq:s3} should be replaced with for the primitive case, as we will now see. In light of the coupling construction from \cite{mec}, we shall reverse the process, and start with a linear connection and a curving to construct a primitive IM connection.
\begin{proposition}
\label{prop:splitting_coh_triv}
Let $B\Ra M$ be a Lie algebroid and let $(\frak k,[\cdot,\cdot]_{\frak k})$ be a bundle of Lie algebras over $M$. Suppose that a connection $\nabla$ on $\frak k$ and a form $F\in \Omega^2(M;\frak k)$ are given, such that the following conditions are satisfied:
\begin{enumerate}[label={(\roman*)}]
\item The connection $\nabla$ preserves the Lie bracket on $\frak k$.
\item The tensors $R^\nabla$ and $F$ are related by $R^\nabla=-\ad F$.
\item The 3-form $\d{}^\nabla F$ is transversal, i.e., $\iota_{\rho_B(\alpha)}\d{}^\nabla F=0$ for any $\alpha \in B$.
\end{enumerate}
Then the direct sum $A=B\oplus \frak k$ has a structure of a Lie algebroid, which admits a primitive IM connection $(\C,v)$, given by
\[
v(\alpha,\xi)=\xi,\quad\C(\alpha,\xi)=\nabla\xi+\iota_{\rho_B(\alpha)}F,
\]
with $F$ as its curving. Conversely, any Lie algebroid $A$ with a bundle of ideals $\frak k\subset A$ that admits a primitive IM connection, is isomorphic to one of this type.
\end{proposition}
\begin{remark}
  This proposition simplifies considerably when $\frak k$ is a semisimple Lie algebra bundle; see \sec\ref{sec:semisimple} and Corollary \ref{cor:semisimple_extension}.
\end{remark}
\begin{proof}
Assumptions (i) and (ii) imply that conditions \eqref{eq:s1} and \eqref{eq:s2} are satisfied for the pair $(\nabla,U)$, where $U\in \Gamma (B^*\otimes T^*M\otimes \frak k)$ is given by $U(\alpha)=-\iota_{\rho_B(\alpha)}F$. Moreover, it is straightforward to see that \eqref{eq:s3} for our case  amounts to saying that $\d{}^\nabla F$ vanishes when evaluated on two vectors tangent to the orbit foliation of $B$. Hence, (iii) implies \eqref{eq:s3}, and we can apply \cite{mec}*{Proposition 5.13} to conclude $A$ is a Lie algebroid, with the anchor given by the composition $A\ra B\xrightarrow{\rho_B}TM$, and the bracket by
\begin{align}
\label{eq:bracket_construction}
[(\alpha,\xi),(\beta,\eta)]=\big([\alpha,\beta]_B, \nabla_{\rho_B(\alpha)}\eta-\nabla_{\rho_B(\beta)}\xi+[\xi,\eta]_{\frak k}-F(\rho_B(\alpha),\rho_B(\beta))\big). 
\end{align}
The obtained IM connection $(\C,v)$ is indeed primitive since condition (iii) used with Cartan's magic formula yields 
\[
-\d{}^\nabla U(\alpha)=\L^\nabla_{\rho_B(\alpha)}F,
\]
for any $\alpha\in\Gamma(B)$, hence also 
\[
\Omega^{(\C,v)}(\alpha,\xi)=(R^\nabla\cdot\xi-\d{}^\nabla U(\alpha),-U(\alpha))=(\L^A_{(\alpha,\xi)} F,\iota_{\rho(\alpha,\xi)}F)=\delta^0 F.\qedhere
\]
\end{proof}
\begin{corollary}
  \label{cor:primitive_bijective_correspondence}
  Let $\frak k$ be a bundle of ideals of a Lie algebroid $A$. There is a bijective correspondence between primitive IM connections for $\frak k$ together with a choice of curving, and triples $(v,\nabla,F)$, where $v\colon A\ra \frak k$ is a splitting, the pair $(\nabla,F)$ satisfies conditions (i)--(iii) of Proposition \ref{prop:splitting_coh_triv}, and the splitting is compatible with $(\nabla,F)$, that is:
  \begin{align}
    \label{eq:primitive_additional_conditions}
    \nabla^A_{h(\alpha)}=\nabla^{}_{\rho(\alpha)},\quad F^v=(\rho_B)^* F
  \end{align}
  for any $\alpha\in A$, where $B=A/\frak k$ and $F^v\in\Omega^2(B;\frak k)$ is the curvature of the splitting.
\end{corollary}
\begin{proof}
  One direction is clear from the  identities \eqref{eq:U_along_orbits}, \eqref{eq:conn_orb2} and Lemma \ref{lem:curving}. For the other direction we use the previous proposition, where the conditions \eqref{eq:primitive_additional_conditions} ensure that the obtained algebroid structure on $B\oplus \frak k$ is isomorphic to $A$. 
\end{proof}
\begin{remark}
  If $F'\in \Omega^2(M;\frak k)$ is another 2-form satisfying the same properties as $F$ in the corollary above, then $(v,\nabla,F')$ defines the same IM connection as $(v,\nabla, F)$ if and only if the difference $F-F'$ is transversal, that is, $\iota_X(F-F')=0$ whenever $X\in T\F$. The equivalence follows by observing that for any $\alpha\in \Gamma(A)$,
  \[
  \L^A_\alpha(F-F')=\L^\nabla_{\rho\alpha}(F-F')+[v\alpha,F-F']=0,
  \]
  using Cartan's formula together with properties (ii) and (iii) from the proposition. 
\end{remark}

We now examine the affinity of the space of primitive multiplicative connections. We emphasize that this should not be confused with the aforementioned affinity of $(\delta^0)^{-1}(\C,v)$ for a fixed IM connection $(\C,v)$.
\begin{proposition}
   The deformation of a primitive IM connection by a cohomologically trivial IM 1-form is again a primitive IM connection. More precisely, if $F$ is a curving of IM connection $(\C,v)$, then a curving for the IM connection $(\C,v)+\delta^0\gamma$ is given by
\begin{align}
   F^\gamma\coloneq F+\d{}^\nabla\gamma-\tfrac 12 [\gamma,\gamma]_{\frak k}.
\end{align}
   for any $\gamma\in \Omega^1(M;\frak k)$. Moreover, under any such affine deformation, the 3-curvature form $G=\d{}^\nabla F$ remains unchanged.
\end{proposition}
\begin{proof}
By Theorems \ref{thm:expansion_inf} and \ref{thm:deltaD_inf}, we only need to check that $\mathbf c_2(L,l)$ is cohomologically trivial. We show this by proving 
\begin{align}
\label{eq:c2_cohtriv}
\mathbf c_2(\delta^0\gamma)=-\frac 12\delta^0[\gamma,\gamma],
\end{align}
i.e., that the following diagram commutes.
% https://q.uiver.app/#q=WzAsNCxbMCwwLCJcXE9tZWdhXjFfe2ltfShBO1xcZnJhayBrKV5cXEhvciJdLFsxLDAsIlxcT21lZ2FeMl97aW19KEE7XFxmcmFrIGspXlxcSG9yIl0sWzAsMSwiXFxPbWVnYV4xKE07XFxmcmFrIGspIl0sWzEsMSwiXFxPbWVnYV4yKE07XFxmcmFrIGspIl0sWzIsMywiXFxnYW1tYVxcbWFwc3RvLVxcZnJhYyAxMltcXGdhbW1hLFxcZ2FtbWFdIiwyXSxbMywxLCJcXGRlbHRhXjAiLDJdLFsyLDAsIlxcZGVsdGFeMCJdLFswLDEsIlxcbWF0aGJmIGNfMiJdXQ==
\[
\begin{tikzcd}[row sep=large]
	{\Omega^1_{im}(A;\frak k)^\Hor} & {\Omega^2_{im}(A;\frak k)^\Hor} \\
	{\Omega^1(M;\frak k)} & {\Omega^2(M;\frak k)}
	\arrow["{\mathbf c_2}", from=1-1, to=1-2]
	\arrow["{\delta^0}", from=2-1, to=1-1]
	\arrow["{\gamma\mapsto-\frac 12[\gamma,\gamma]}"', from=2-1, to=2-2]
	\arrow["{\delta^0}"', from=2-2, to=1-2]
\end{tikzcd}
\]
Let us denote $(L,l)=\delta^0\gamma$ and $(J,j)= \mathbf c_2(\delta^0\gamma)$, that is, 
\[
(J,j)\alpha=-(L|_{\frak k}\wedgedot L\alpha,L|_{\frak k}\cdot l\alpha)=-(L|_{\frak k}\wedgedot\L^A_\alpha\gamma, [\gamma(\rho\alpha),\gamma]),
\]
for any $\alpha\in\Gamma(A)$, where $L|_{\frak k}=[-,\gamma]$. To check that the symbols coincide, just note that $[\gamma,\gamma](X,Y)=2[\gamma(X),\gamma(Y)]$ holds for any $X,Y\in\vf(M)$, hence
\[
j(\alpha)=-\frac 12\big(\iota_{\rho(\alpha)}[\gamma,\gamma]\big).
\]
For the leading term, we compute 
\begin{align*}
-J\alpha(X,Y)&=[(\L^A_\alpha\gamma)(X),\gamma(Y)]-[(\L^A_\alpha\gamma)(Y),\gamma(X)]\\
&=[[\alpha,\gamma(X)],\gamma(Y)]-[\gamma[\rho\alpha,X],\gamma(Y)]-[[\alpha,\gamma(Y)],\gamma(X)]+[\gamma[\rho\alpha,Y],\gamma(X)]\\
&=[\alpha,[\gamma(X),\gamma(Y)]]-[\gamma[\rho\alpha,X],\gamma(Y)]-[\gamma(X),\gamma[\rho\alpha,Y]],
%&=\frac 12[\alpha,[\gamma,\gamma](X,Y)]-\frac 12[\gamma,\gamma]([\rho\alpha,X],Y)-\frac 12[\gamma,\gamma](X,[\rho\alpha,Y])
\end{align*}
where  the first and third term in the second line were combined using the Jacobi identity. Now using the identity for $[\gamma,\gamma]$ on all three terms shows that this equals
\[
-J\alpha(X,Y)=\frac 12\L^A_\alpha[\gamma,\gamma](X,Y),
\]
which concludes the proof of the first part of the proposition. The second part of the proposition is a simple computation:
      \begin{align}
\begin{split}
  \label{eq:3_curvature_invariant}
  G^\gamma&=\d{}^{\nabla^\gamma}F^\gamma=(\d{}^\nabla+[\cdot,\gamma]_{\frak k})\big(F+\d{}^\nabla\gamma-\tfrac 12[\gamma,\gamma]_{\frak k}\big)\\
  &=G+R^\nabla\wedge\gamma-[\d{}^\nabla\gamma,\gamma]_{\frak k}+[F,\gamma]_{\frak k}+[\d{}^\nabla\gamma,\gamma]_{\frak k}-\tfrac 12[[\gamma,\gamma]_{\frak k},\gamma]_{\frak k}=G.
\end{split}
\end{align}
Here, we denoted by $\nabla^\gamma$ the connection on $\frak k$ associated to the IM connection $(\C,v)+\delta^0\gamma$, 
\begin{align}
\label{eq:nabla_gamma}
   \nabla^\gamma=\nabla+[\cdot,\gamma],
\end{align} 
noted that the following identity holds on forms of degree $k$,
\begin{align}
   \label{eq:d_nabla_gamma}
   \d{}^{\nabla^\gamma}=\d{}^\nabla+(-1)^k[\cdot,\gamma]_{\frak k},
\end{align}
and observed that the Jacobi identity for $[\cdot,\cdot]_{\frak k}$ implies $[[\gamma,\gamma]_{\frak k},\gamma]_{\frak k}=0$.
\end{proof}
\begin{remark}
The remarkable fact that the 3-curvature does not change with this deformation is crucial in the development of Yang--Mills theory for primitive IM connections; see \cite{ym_mecs}*{\sec 3.3}.
\end{remark}
\subsubsection{Lie algebroids of principal type}
Let $A\Ra M$ be an arbitrary algebroid and suppose we are given a primitive connection $(\C,v)\in\A(A;\frak k)$ with a curving $F$. Observe that the curvature 3-form $G=\d{}^\nabla F$ vanishes if and only if the direct sum $A'\coloneqq TM\oplus \frak k$ is a Lie algebroid, with anchor $\pr_1\colon A'\ra TM$ and the Lie bracket given by
\begin{align}
  \label{eq:principal_bracket}
  [(X,\xi),(Y,\eta)]=([X,Y],\nabla_X\eta-\nabla_Y\xi+[\xi,\eta]_{\frak k}-F(X,Y)).
\end{align}
The vanishing of $G$ amounts precisely to the vanishing Jacobiator for this bracket. In this case, $A'$ is clearly a transitive algebroid, and furthermore, the algebroid $A\Ra M$ is isomorphic to a Lie algebroid \textit{of principal type} \cite{mec}*{\sec 6.6},
\[
A'\times_{TM} B \Ra M,\quad A'\times_{TM} B \coloneqq\set{(\alpha,\beta)\in A'\times B\given \rho_{A'}(\alpha)=\rho_B(\beta)}.
\]
Its algebroid structure is determined by the condition that the inclusion $A'\times_{TM}B\hookrightarrow A'\times B$ is a Lie algebroid monomorphism into the direct product algebroid. The isomorphism reads \[\alpha\mapsto (\rho(\alpha),v(\alpha), \phi(\alpha)),\] where $\phi\colon A\ra B=A/\frak k$ is the canonical projection. Hence, the existence of a primitive connection for $\frak k$, with a curving whose curvature 3-form vanishes, forces the algebroid to be of principal type. 

This observation is closely related to the construction of IM connections on algebroids of principal type from \cite{mec}*{\sec 6.6}. To elaborate, let $A\coloneqq A'\times_{TM}B$ be an algebroid of principal type for some pair of Lie algebroids $A'$ and $B$, with $A'$ transitive. Denote by $\phi\colon A\ra B$ the projection and consider the bundle of ideals $\frak k=\ker\phi$, which can be identified with $\ker\rho_{A'}$ under the projection $\pr_{A'}\colon A\ra A'$. Recall that a splitting of the Atiyah sequence of $A'$,
\begin{align}
  \label{eq:splitting_principal_type}
    \begin{tikzcd}[ampersand replacement=\&, column sep=large]
    0 \& {\frak k} \& {A'} \& {TM} \& 0,
    \arrow[from=1-1, to=1-2]
    \arrow[from=1-2, to=1-3]
    \arrow["v'"{pos=0.45}, bend left=30, from=1-3, to=1-2]
    \arrow["\rho_{A'}", from=1-3, to=1-4]
    \arrow[from=1-4, to=1-5]
  \end{tikzcd}
  \end{align}
is the same thing as an IM connection $(\C',v')\in\A(A';\frak k)$ by condition \eqref{eq:c2}, which can then be pulled back along $\pr_{A'}$ to an IM connection $(\C,v)\in\A(A;\frak k)$:
\begin{align}
\label{eq:principal_type_IM_connection_induced_by_splitting}
    \C(\alpha,\beta)=\C'(\alpha),\quad v(\alpha,\beta)=v'(\alpha),
\end{align}
for any $\alpha\in\Gamma(A')$ and $\beta\in\Gamma(B)$ with $\rho_{A'}(\alpha)=\rho_B(\beta)$. In this way, the authors of \cite{mec} conclude that Lie algebroids of principal type admit IM connections. As the following shows, the obtained IM connection is actually primitive.
\begin{proposition}
  Let $A=A'\times_{TM} B$ be a Lie algebroid of principal type, with the bundle of ideals $\frak k=\ker\phi$, where $\phi\colon A\ra B$ is the canonical projection. The IM connection \eqref{eq:principal_type_IM_connection_induced_by_splitting} induced by a splitting $v'\colon A'\ra \frak k$ of the transitive algebroid $A'$ is primitive with curving $F^{v'}$, whose curvature 3-form vanishes.
\end{proposition}
\begin{proof}
  First observe that the curving of the IM connection \eqref{eq:principal_type_IM_connection_induced_by_splitting} is given by the curvature $F^{v'}$ of the splitting $v'\colon A'\ra \frak k$. Indeed, we have
\[
\Omega^{(\C,v)}(\alpha,\beta)=\Omega^{(\C',v')}(\alpha)=\delta^0_{A'}(F^{v'})(\alpha)=(\L^{A'}_{\alpha}F^{v'},\iota_{\rho_{A'}(\alpha)}F^{v'})
\]
where the first equality holds since $(\C,v)=(\pr_{A'})^*(\C',v')$. Now observe that $\rho(\alpha,\beta)=\rho_{A'}(\alpha)$ and the representation of $A$ on $\frak k$ is given for any $\xi\in\Gamma(\frak k)$ as
\[
[(\alpha,\beta),(\xi,0)]=([\alpha,\xi],0),
\]
so we conclude $\smash{\L^A_{(\alpha,\beta)}=\L^{A'}_\alpha}$ and hence $\Omega^{(\C,v)}=\delta^0_A(F^{v'})$. In conclusion, any IM connection on a Lie algebroid of principal type $A'\times_{TM}B$, induced by a splitting $v'$ of \eqref{eq:splitting_principal_type}, is primitive with a canonical curving $F^{v'}$ satisfying $\d{}^\nabla F^{v'}=0$ by the Bianchi identity.
\end{proof}
The results of this subsection thus culminate in the following statement.
\begin{corollary}
  A Lie algebroid $A$ with a bundle of ideals $\frak k$ is of principal type if and only if it admits a primitive IM connection for $\frak k$ with a curving whose 3-curvature vanishes.
\end{corollary}

% \begin{remark}
%     It is clear that a primitive connection on $A'\times_{TM}B$ does not necessarily admit a curving with vanishing 3-curvature---for example, we can change $F^{v'}$ by an invariant 2-form $\beta$ satisfying $\d{}^\nabla\beta\neq 0$ if the codimension of the foliation $\F$ in $M$ is large enough.
% \end{remark}

\subsubsection{The semisimple case}
\label{sec:semisimple}
A particularly well-behaved scenario for the theory of multiplicative connections is when the typical fibre of $\frak k$ is a semisimple Lie algebra, in which case an IM connection on $A$ for $\frak k$ always exists. This was established in \cite{mec}*{Corollary 6.6} by showing that $A$ is then isomorphic to an algebroid of principal type, with $\frak k$ the isotropy of a transitive algebroid $A'$, whose sections are bracket-preserving derivations of $\frak k$. We now establish some stronger properties.

In what follows, we will refer to a locally trivial bundle of Lie algebras $\frak k$ with a semisimple typical fibre as a \textit{semisimple Lie algebra bundle}.\footnote{When the base $M$ is connected and all the fibres are semisimple, local triviality is automatic by the rigidity results for bundles of semisimple Lie algebras.} In the case when $\frak k\subset A$ is also a bundle of ideals, it will be called a \textit{semisimple bundle of ideals} of $A$. As concerns cohomological triviality, we have the following.
\begin{proposition}
  \label{prop:semisimple}
Let $\frak k$ be a semisimple bundle of ideals of a Lie algebroid $A\Ra M$. There holds:
\begin{enumerate}[label={(\roman*)}]
  \item The differential at level zero $\delta^0\colon \Omega^\bullet(M;\frak k)\ra \Omega^\bullet_{im}(A;\frak k)^\Hor$ is an isomorphism.
  \item Any multiplicative connection for $\frak k$ is uniquely primitive, with vanishing  3-curvature.
\item There is a bijective correspondence:\\
\begin{minipage}{\linewidth}
\vspace{2pt}
\begin{align*}
\left\{
\parbox{3cm}{\centering IM connections $\A(A;\frak k)$}
\right\}
\longleftrightarrow
%\set*{(v,\nabla)\given \parbox{6cm}{$v\colon A\ra \frak k$ is a splitting and $\nabla$ is a bracket-preserving connection on $\frak k$, with $\ad(F^v)=\rho^*R^\nabla$}}
\left\{
\parbox{7.9cm}{\centering Pairs $(v,
\nabla)$, where $v\colon A\ra \frak k$ is a splitting and\\$\nabla$ is a bracket-preserving connection on $\frak k$,\\such that $\nabla^A_{h(\alpha)}=\nabla^{}_{\rho(\alpha)}$ for all $\alpha\in A$}
\right\}
\end{align*}
\vspace{0pt}
\end{minipage}
%Here, $F^v\in\Omega^2(B;\frak k)$ denotes the curvature of a splitting $v\colon A\ra \frak k$.
\end{enumerate}
\end{proposition}
\begin{remark}
As will be clear from the proof, it is actually enough to assume that \[\ad\colon \frak g\rightarrow \Der(\frak g)\] is an isomorphism, where $\frak g$ is the typical fibre of $\frak k$. In other words, (i) states that the vanishing of the zeroth and first Chevalley--Eilenberg cohomology groups of the typical fibre $\frak g$, with values in the adjoint representation of $\frak g$, implies the vanishing of the zeroth and first horizontal Weil cohomology groups of $\frak k$-valued IM forms on $A\Ra M$,
\begin{align*}
  H^{0,\bullet}(A;\frak k)^\Hor=H^{1,\bullet}(A;\frak k)^\Hor=0.
\end{align*}
\end{remark}
\begin{proof}
The center $z(\frak k)$ of a semisimple Lie algebra bundle $\frak k$ is trivial, so $\delta^0$ is injective by Remark \ref{rem:invariant}. For surjectivity, we take any $(L,l)\in\Omega^\bullet_{im}(A;\frak k)^\Hor$ and observe that $L|_{\frak k}$ is tensorial by horizontality. By condition \eqref{eq:c1}, it is moreover a derivation-valued form $L|_{\frak k}\in\Omega^\bullet(M;\Der\frak k)$, hence semisimplicity implies there is a unique $\gamma\in\Omega^\bullet(M;\frak k)$ which satisfies $L|_{\frak k}=-\ad \gamma$. That is,
\begin{align}
\label{eq:semisimple_inverse}
[\xi,\gamma]=L|_{\frak k}\cdot \xi,
\end{align}
%\footnote{That $F$ is smooth follows from the fact that $\frak k$ is a locally trivial bundle of Lie algebras.} 
for any $\xi\in\Gamma(\frak k)$. To see that $\delta\gamma=(L,l)$, we let $\alpha\in\Gamma(A)$ and apply $\iota_{\rho(\alpha)}$ to both sides of \eqref{eq:semisimple_inverse}. By condition \eqref{eq:c2}, there holds
\[
[\xi,\iota_{\rho(\alpha)}\gamma]=\iota_{\rho(\alpha)}L\xi=\L_\xi^A l\alpha-l[\xi,\alpha]=[\xi,l\alpha],
\]
where we have used horizontality of $(L,l)$. This shows that the symbols coincide. Finally, for the leading term we use condition \eqref{eq:c1} to get
\begin{align*}
[\xi,L\alpha]=\L_\xi^A L\alpha=L[\xi,\alpha]+\L_\alpha^A L\xi=[[\xi,\alpha],\gamma]+\L^A_\alpha[\xi,\gamma].
\end{align*}
It is now easy to see that the right-hand side equals $[\xi,\L_\alpha^A\gamma]$ by Jacobi identity, concluding part (i). Part (ii) is clearly implied by (i): applying equation \eqref{eq:semisimple_inverse} to the curvature $\Omega^{(\C,v)}$ of an IM connection $(\C,v)$ yields the unique curving $F$, implicitly defined by 
\begin{align}
  \label{eq:semisimple_def_F}
  R^\nabla=-\ad F.
\end{align} 
For part (iii), if we are given a pair $(v,\nabla)$, there is a unique form $F$ satisfying \eqref{eq:semisimple_def_F}, so we would now like to apply Corollary \ref{cor:primitive_bijective_correspondence}. The condition (iii) there is automatically fulfilled since $\d{}^\nabla F$ is $z(\frak k)$-valued:
\begin{align*}
[\xi,\d{}^\nabla F]=\d{}^\nabla[\xi,F]-[\nabla\xi,F]=\d{}^\nabla(R^\nabla\cdot \xi)-R^\nabla\wedge\nabla\xi=(\d{}^{\nabla^{\End\frak k}}R^\nabla)\cdot \xi=0
\end{align*}
where we used the Bianchi identity for $R^\nabla$ in the last step. Moreover, the condition $\nabla^A_{h\alpha}=\nabla_{\rho\alpha}$ implies $(\rho_B)^*R^\nabla=-\ad F^v$ already in the non-semisimple case, and now inserting \eqref{eq:semisimple_def_F} into this equation yields $(\rho_B)^*F=F^v$, which is what we needed to show.
\end{proof}
The last part of the proof above shows that the coupling construction from Proposition \ref{prop:splitting_coh_triv} significantly simplifies in the semisimple case: we just need to assume that $\nabla$ is bracket-preserving, and the other two conditions are automatically satisfied by defining $F$ with $R^\nabla$. However, the existence of a bracket-preserving connection $\nabla$ is equivalent to local triviality of $\frak k$, which is already assumed, so we obtain the following.
\begin{corollary}
\label{cor:semisimple_extension}
If $B\Ra M$ is a Lie algebroid and $(\frak k,[\cdot,\cdot]_{\frak k})$ is a semisimple Lie algebra bundle, then $A=B\oplus \frak k$ has a structure of a Lie algebroid defined by \eqref{eq:bracket_construction}, where $\nabla$ is any bracket-preserving linear connection on $\frak k$ and the form $F$ is implicitly defined by $R^\nabla=-\ad F$.  Conversely, any Lie algebroid $A$ with a semisimple bundle of ideals $\frak k$ is isomorphic to one of this type.
\end{corollary}
%\begin{remark}
%To reiterate, any Lie algebroid can be (non-uniquely) extended by a semisimple Lie algebra bundle. Non-uniqueness of the extension is reflected precisely by the fact that the space of bracket-preserving linear connections on $\frak k$ is an affine space over $\Omega^1(M;\Der\frak k)\cong \Omega^1(M;\frak k)$.
%\end{remark}

\subsubsection{The abelian case}
We now establish a criterion for the existence of primitive IM connections for an abelian bundle of ideals $\frak k\subset A$. Preliminarily, observe that when $\frak k$ is abelian, the representation $\nabla^A$ descends to a  representation $\nabla^B$ of $B$ on $\frak k$, given by $\nabla^B_\alpha\xi=[\tilde\alpha,\xi]$, where $\tilde\alpha\in\Gamma(A)$ denotes any lift of $\alpha\in\Gamma(B)$ along the anchor, and $\xi\in\Gamma(\frak k)$.

Now suppose $(\C,v)\in \arc(A;\frak k)$ is a primitive connection and $F$ is its curving. By identity \eqref{eq:R_ad_F}, $\nabla$ is now necessarily flat, and moreover, equation \eqref{eq:conn_orb2} now just says that $\nabla$ induces $\nabla^B$, that is, 
\begin{align}
\label{eq:nabla_induces}
\nabla^B_\alpha=\nabla_{\rho(\alpha)}.
\end{align}
 Let us consider the set of 2-forms $F$ with \textit{transverse coboundary} $\d{}^\nabla F$, that is,
\[
\Omega^2_{tc,\nabla}(M;\frak k)\coloneqq \set{F\in \Omega^2(M;\frak k)\given \iota_X\d{}^\nabla F=0\text{ for any $X\in T\F$}}.
\]
Let $(\rho_B)^*\colon \Omega^2(M;\frak k)\ra \Omega^2(B;\frak k)$ denote the pullback along the anchor of $B$, and observe that there holds $(\rho_B)^*\d{}^\nabla=\d{}^{\nabla^B}(\rho_B)^*$ by equation \eqref{eq:nabla_induces}, so that the pullback restricts to
\[
(\rho_B)^*_{tc,\nabla}\colon \Omega_{tc,\nabla}^2(M;\frak k)\ra Z^2(B;\frak k)
\]
where $Z^2(B;\frak k)$ denotes $\frak k$-valued 2-cocycles on $B$ with respect to the differential $\d{}^{\nabla^B}$. Clearly, the map above maps $F$ to the curvature of the splitting $F^v$. Moreover, in the usual Lie algebroid cohomology  of $B$ with values in $\frak k$, the class 
\[c(A)\coloneqq [F^v]\in H^2(B;\frak k)\] 
is independent of the splitting $v\colon A\ra \frak k$ since $\frak k$ is abelian. It is thus in the image of the map $\Omega_{tc,\nabla}^2(M;\frak k)\ra H^2(B;\frak k)$ induced by $(\rho_B)^*_{tc,\nabla}$, which we denote by the same symbol. With this in mind, we formulate the following.
\begin{proposition}
\label{prop:abelian_primitive}
An abelian bundle of ideals $\frak k\subset A$ admits a primitive connection if and only if there exists a flat connection $\nabla$ on $\frak k$, which induces $\nabla^B$ and satisfies \[c(A)\in\im(\rho_B)^*_{tc,\nabla}.\] Moreover, we have the following bijective correspondence.
\begin{align*}
\left\{
\parbox{4.8cm}{\centering Primitive connections for $\frak k$\\with a choice of curving}
\right\}
\longleftrightarrow
%\set*{(v,\nabla)\given \parbox{6cm}{$v\colon A\ra \frak k$ is a splitting and $\nabla$ is a bracket-preserving connection on $\frak k$, with $\ad(F^v)=\rho^*R^\nabla$}}
\left\{
\parbox{7.8cm}{\centering Triples $(v,
\nabla,F)$, where $v\colon A\ra \frak k$ is a splitting,\\$\nabla$ is a  flat connection on $\frak k$ inducing $\nabla^B$, and\\$F\in\Omega^2_{tc,\nabla}(M;\frak k)$ satisfies $F^v=(\rho_B)^*F$.}
\right\}
\end{align*}
\end{proposition}
\begin{proof}
One direction for the statement about existence is already proved above. For the other direction, suppose there is a flat connection $\nabla$ that induces $\nabla^B$ and satisfies $c(A)\in\im(\rho_B)^*_{tc,\nabla}$. Since $\frak k$ is abelian, the latter means precisely that there exists a splitting $v\colon A\ra \frak k$ with curvature 
\begin{align}
\label{eq:temp_abelian_prop}
F^v=(\rho_B)^* F,
\end{align}
for some $F\in\Omega^2(M;\frak k)$ with $\iota_{\rho_B(\alpha)}\d{}^\nabla F=0$ for any $\alpha\in B$. The pair $(\nabla, F)$ satisfies the assumptions of Proposition \eqref{prop:splitting_coh_triv}, and equation \eqref{eq:temp_abelian_prop} ensures that the Lie algebroid $A$ is isomorphic to the obtained algebroid $B\oplus \frak k$, which has an IM connection with curving $F$. The statement regarding the bijective correspondence is clear from this argument. %If there are two 2-forms $F,\tilde F\in\Omega_{tc,\nabla}^2(M;\frak k)$ such that $F^v=(\rho_B)^* F=(\rho_B)^* \tilde F$ for a given splitting $v\colon A\ra \frak k$, then they must differ by an invariant form, hence they both define a curving of the same IM connection.
\end{proof}
\begin{remark}
The last proposition is not unrelated to the criterion for existence of kernel flat connections on an abelian bundle of ideals $\frak k$ from \cite{mec}*{Proposition 5.21}. In fact, Proposition \ref{prop:abelian_primitive} is its refinement, and their relation can be precisely summarized with the following diagram.
\[\begin{tikzcd}[row sep=large]
	{H_{im}^2(B;\frak k)} & {H^2(B;\frak k)} \\
	{\Omega_{tc,\nabla}^2(M;\frak k)}
	\arrow["\operatorname{sym}", from=1-1, to=1-2]
	\arrow["{\delta^0}", from=2-1, to=1-1]
	\arrow["{(\rho_B)^*_{tc,\nabla}}"', from=2-1, to=1-2]
\end{tikzcd}\]
Here, $H_{im}^\bullet(B;\frak k)$ denotes the cohomology of the cochain complex $(\Omega^\bullet_{im}(B;\frak k),\d{}^\nabla_{im})$, where the differential $\d{}^\nabla_{im}$ is the restriction to IM forms of $\d{}^\nabla$ as in equation \eqref{eq:ext_cov_der_p=1}, and \[\operatorname{sym}(c_0,c_1)(\alpha,\beta)=c_1(\alpha)(\rho \beta).\]
\end{remark}

\appendix
\section{Auxiliary computations}
\begin{lemma}
\label{lemma:wd_dnabla}
Let $\nabla$ be a connection on a representation $V$ of a Lie algebroid $A\Ra M$. The map $\d{}^\nabla\colon W^{p,q}(A;V)\ra W^{p,q+1}(A;V)$ given by \eqref{eq:dnabla} is well-defined.
\end{lemma}
\begin{proof}
  Given $c=(c_0,\dots,c_p)\in W^{p,q}(A;V)$, we would like to show that \[\d{}^\nabla c=((\d{}^\nabla c)_0,\dots,(\d{}^\nabla c)_p)\] defines a Weil cochain. Letting $\ul\alpha=(\alpha_1,\dots,\alpha_{p-k})$, $\ul\beta=(\beta_1,\dots,\beta_k)$ and $f\in C^\infty(M)$, we straightforwardly compute:
\begin{align*}
  &(-1)^k (\d{}^\nabla c)_k(f\alpha_1,\dots,\alpha_{p-k}\|\ul\beta)=\d{}^\nabla\big(fc_k(\ul\alpha\|\ul\beta)+\d f\wedge c_{k+1}(\alpha_2,\dots,\alpha_{p-k}\|\alpha_1,\beta)\big)\\
  &-\textstyle\sum_i\big(f c_{k-1}(\beta_i,\ul\alpha\|\beta_1,\dots,\widehat{\beta_i},\dots,\beta_k)-\d f\wedge c_k(\beta_i,\alpha_2,\dots,\alpha_{p-k}\|\alpha_1,\beta_1,\dots,\widehat{\beta_i},\dots,\beta_k)\big)\\
  &=f\big({\d{}}^\nabla c_k(\ul\alpha\|\ul\beta)-\textstyle\sum_i c_{k-1}(\beta_i,\ul\alpha\|\beta_1,\dots,\widehat{\beta_i},\dots,\beta_k)\big)\\
  &+\d f\wedge \big(c_k(\ul\alpha\|\ul\beta)+\textstyle\sum_i c_k(\beta_i,\alpha_2,\dots,\alpha_{p-k}\|\alpha_1,\beta_1,\dots,\widehat{\beta_i},\dots,\beta_k)-\d{}^\nabla c_{k+1}(\alpha_2,\dots,\alpha_{p-k}\|\alpha_1,\ul\beta)\big)\\
  &=f (-1)^k (\d{}^\nabla c)_k(\ul\alpha\|\ul\beta)-\d f\wedge (-1)^{k+1}(\d{}^\nabla c)_{k+1}(\alpha_2,\dots,\alpha_{p-k}\|\alpha_1,\ul\beta),
\end{align*}
where we have used the Leibniz identity for $c$ in the first and second line, and the graded Leibniz rule for $\d{}^\nabla$ in the third and fourth. In the last line, we recognized that the expressions in the parentheses in the third and fourth line are precisely the defining expressions for $(-1)^k(\d{}^\nabla c)_k$ and $-(-1)^{k+1}(\d{}^\nabla c)_{k+1}$, respectively. Cancelling the factor $(-1)^k$ on both sides finishes the proof.
\end{proof}

\begin{lemma}
  \label{lem:T_theta_wedge}
  Let $V$ be a representation of a Lie algebroid $A\Ra M$. For any $c\in W^{p,q}(A;V)$ and $(T,\theta)\in W^{1,1}(A;\End V)$, the following expression defines the coefficients of a Weil cochain $(T,\theta)\wedge c\in W^{p+1,q+1}(A;V)$.
\begin{align*}
\big((T,\theta)\wedge c\big)_k(\alpha_0,\dots,\alpha_{p-k}\|\ul\beta)&=\textstyle\sum_{i=0}^{p-k} (-1)^i T(\alpha_i)\wedge c_k(\alpha_0,\dots,\widehat{\alpha_i},\dots,\alpha_{p-k}\|\ul\beta)\\
&+\textstyle\sum_{j=1}^k\theta(\beta_j)\cdot c_{k-1}(\ul\alpha\|\beta_1,\dots,\widehat{\beta_j},\dots,\beta_k),
\end{align*}
\end{lemma}
\begin{proof}
  For any $\ul\alpha=(\alpha_0,\dots,\alpha_{p-k})$, $\ul\beta=(\beta_1,\dots,\beta_k)$ and $f\in C^\infty(M)$, we compute:
\begin{align*}
\big((T&,\theta)\wedge c\big)_k(f\alpha_0,\dots,\alpha_{p-k}\|\ul\beta)=\underbrace{T(f\alpha_0)}_{\mathclap{fT(\alpha_0)+\d f\otimes\theta(\alpha_0)}}\wedge\, c_k(\alpha_1,\dots,\alpha_{p-k}\|\ul\beta)\\
&+\textstyle\sum_{i=1}^{p-k}(-1)^i T(\alpha_i)\wedge \underbrace{c_k(f\alpha_0,\alpha_1,\dots,\widehat{\alpha_i},\dots,\alpha_{p-k}\|\ul\beta)}_{\mathclap{fc_k(\alpha_0,\dots,\widehat{\alpha_i},\dots,\alpha_{p-k}\|\ul\beta)+\d f\wedge c_{k+1}(\alpha_1,\dots,\widehat{\alpha_i},\dots,\alpha_{p-k}\|\alpha_0,\ul\beta)}}\\
&+\textstyle\sum_{j=1}^k\theta(\beta_j)\cdot \underbrace{c_{k-1}(f\alpha_0,\dots,\alpha_{p-k}\|\beta_1,\dots\widehat{\beta_i},\dots,\beta_k)}_{\mathclap{f c_{k-1}(\alpha_0,\dots,\alpha_{p-k}\|\beta_1,\dots\widehat{\beta_i},\dots,\beta_k)+\d f\wedge c_k(\alpha_1,\dots,\alpha_{p-k}\|\alpha_0,\beta_1,\dots\widehat{\beta_i},\dots,\beta_k)}}
\end{align*}
Exchanging $\d f$ with $T(\alpha_i)$ yields an additional minus, and separating the terms with $f$ and $\d f$ yields
\begin{align*}
  \big((T&,\theta)\wedge c\big)_k(f\alpha_0,\dots,\alpha_{p-k}\|\ul\beta)=f\big((T,\theta)\wedge c\big)_k(\alpha_0,\dots,\alpha_{p-k}\|\beta)\\
  &+\d f\wedge \big({-}\textstyle\sum_{i=1}^{p-k}(-1)^i T(\alpha_i)\wedge c_{k+1}(\alpha_1,\dots,\widehat{\alpha_i},\dots,\alpha_{p-k}\|\alpha_0,\ul\beta)\\
  &\phantom{+\d f\wedge \big(}+\theta(\alpha_0)c_k(\alpha_1,\dots,\alpha_{p-k}\|\ul\beta)+\textstyle\sum_{j=1}^k\theta(\beta_j)c_k(\alpha_1,\dots,\alpha_{p-k}\|\alpha_0,\beta_1,\dots,\widehat{\beta_j},\dots,\beta_k)\big).
  \end{align*}
Now just observe that the terms in the parentheses in the second and third line add up to 
\[\big((T,\theta)\wedge c\big)_{k+1}(\alpha_1,\dots,\alpha_{p-k}\|\alpha_0,\ul\beta).\qedhere\]
\end{proof}
\begin{proposition}
  \label{prop:conn_global_inf}
  Let $\omega\in\A(G;\frak k)$ be a multiplicative Ehresmann connection on a Lie groupoid $G\rra M$. Let $A$ be its Lie algebroid, endowed with the IM connection $(\C,v)=\ve(\omega)$. The two linear connections induced by $\omega$ and $(\C,v)$, respectively defined by equations \eqref{eq:nabla} and \eqref{eq:nabla2}, coincide.
\end{proposition}

\begin{proof}
  We need to show that for any $\xi\in\Gamma(\frak k)$ and $X\in\vf(M),$ there holds
  \[
  \C(\xi)_x(X)=v[Y,\xi^L]_{1_x}
  \]
  for any $x\in M$, where $Y$ is any horizontal $s$-lift of $X$. Denoting $g_\lambda=\phi^{\xi^L}_{-\lambda}(1_x)$, the right-hand side equals
  \begin{align*}
    v[Y,\xi^L]_{1_x}=\deriv\lambda 0 v\big({\d(\phi^{\xi^L}_\lambda)_{g_{\lambda}}}(Y_{g_{\lambda}})\big).
  \end{align*}
  The flow of a left-invariant vector field is given by the right translation along the bisection $\exp(\lambda\xi)$, i.e., $\phi^{\xi^L}_\lambda=R_{\exp(\lambda\xi)}$ and differentiating 
  $R_{\exp(\lambda\xi)}=m\circ(\id,\phi^{\xi^L}_\lambda\circ u\circ s),$ we get
  \[
  \d{}\big(R_{\exp(\lambda\xi)}\big)_{g_\lambda}(Y)=\d m \big(Y_{g_\lambda},\d{}(\phi^{\xi^L}_\lambda)_{1_x}\d u(X_x)\big)
  \]
  Now observe that since $E$ is multiplicative, $v\colon TG\ra K$ is a groupoid morphism. But since $Y$ is horizontal, we obtain
   \begin{align}
    \label{eq:aux_v_Y_xiL}
     v[Y,\xi^L]_{1_x}=\deriv\lambda 0 \d{}({L_{g_{\lambda}}}) v\big({\d(\phi^{\xi^L}_\lambda)_{1_x}\d u(X_x)}\big).
   \end{align}
  On the other hand, we have
  \begin{align*}
    \C(\xi)_x(X)=\deriv\lambda 0 \underbrace{\Ad_{g_{\lambda}^{-1}}}_{\frak k_x\ra\frak k_x}\big(\underbrace{\omega(\d(\phi^{\xi^L}_\lambda)\d u(X_x))}_{\text{in }\frak k_x\text{ for all }\lambda}\big),
  \end{align*}
  so we can apply the chain rule to differentiate this expression. But since $u^*\omega=0$, this yields precisely the expression \eqref{eq:aux_v_Y_xiL}, concluding the proof.
\end{proof}

\Addresses


\begin{bibdiv}
\begin{biblist}

\bib{weil}{article}{
   author={Arias Abad, Camilo},
   author={Crainic, Marius},
   title={The Weil algebra and the Van Est isomorphism},
   language={English, with English and French summaries},
   journal={Ann. Inst. Fourier (Grenoble)},
   volume={61},
   date={2011},
   number={3},
   pages={927--970},
   issn={0373-0956},
   review={\MR{2918722}},
   doi={10.5802/aif.2633},
}


\bib{im_forms}{article}{
   author={Bursztyn, Henrique},
   author={Cabrera, Alejandro},
   title={Multiplicative forms at the infinitesimal level},
   journal={Math. Ann.},
   volume={353},
   date={2012},
   number={3},
   pages={663--705},
   issn={0025-5831},
   review={\MR{2923945}},
   doi={10.1007/s00208-011-0697-5},
}
\bib{bundles_over_gpds}{article}{
   author={Bursztyn, Henrique},
   author={Cabrera, Alejandro},
   author={del Hoyo, Matias},
   title={Vector bundles over Lie groupoids and algebroids},
   journal={Adv. Math.},
   volume={290},
   date={2016},
   pages={163--207},
   issn={0001-8708},
   review={\MR{3451921}},
   doi={10.1016/j.aim.2015.11.044},
}
\bib{linear_mult}{article}{
   author={Bursztyn, Henrique},
   author={Cabrera, Alejandro},
   author={Ortiz, Cristi\'an},
   title={Linear and multiplicative 2-forms},
   journal={Lett. Math. Phys.},
   volume={90},
   date={2009},
   number={1-3},
   pages={59--83},
   issn={0377-9017},
   review={\MR{2565034}},
   doi={10.1007/s11005-009-0349-9},
}

\bib{twisted_dirac}{article}{
   author={Bursztyn, Henrique},
   author={Crainic, Marius},
   author={Weinstein, Alan},
   author={Zhu, Chenchang},
   title={Integration of twisted Dirac brackets},
   journal={Duke Math. J.},
   volume={123},
   date={2004},
   number={3},
   pages={549--607},
   issn={0012-7094},
   review={\MR{2068969}},
   doi={10.1215/S0012-7094-04-12335-8},
}

\bib{homogeneous}{article}{
   author={Cabrera, Alejandro},
   author={Drummond, Thiago},
   title={Van Est isomorphism for homogeneous cochains},
   journal={Pacific J. Math.},
   volume={287},
   date={2017},
   number={2},
   pages={297--336},
   issn={0030-8730},
   review={\MR{3632890}},
   doi={10.2140/pjm.2017.287.297},
}

\bib{local}{article}{
   author={Cabrera, A.},
   author={M\u arcu\c t, I.},
   author={Salazar, M. A.},
   title={Local formulas for multiplicative forms},
   journal={Transform. Groups},
   volume={27},
   date={2022},
   number={2},
   pages={371--401},
   issn={1083-4362},
   review={\MR{4431167}},
   doi={10.1007/s00031-020-09607-y},
}

\bib{diff_cohomology}{article}{
   author={Crainic, Marius},
   title={Differentiable and algebroid cohomology, van Est isomorphisms, and
   characteristic classes},
   journal={Comment. Math. Helv.},
   volume={78},
   date={2003},
   number={4},
   pages={681--721},
   issn={0010-2571},
   review={\MR{2016690}},
   doi={10.1007/s00014-001-0766-9},
}

\bib{jacobi_structures}{article}{
   author={Crainic, Marius},
   author={Salazar, Mar\'ia Amelia},
   title={Jacobi structures and Spencer operators},
   language={English, with English and French summaries},
   journal={J. Math. Pures Appl. (9)},
   volume={103},
   date={2015},
   number={2},
   pages={504--521},
   issn={0021-7824},
   review={\MR{3298367}},
   doi={10.1016/j.matpur.2014.04.012},
}

\bib{spencer}{article}{
   author={Crainic, Marius},
   author={Salazar, Maria Amelia},
   author={Struchiner, Ivan},
   title={Multiplicative forms and Spencer operators},
   journal={Math. Z.},
   volume={279},
   date={2015},
   number={3-4},
   pages={939--979},
   issn={0025-5874},
   review={\MR{3318255}},
   doi={10.1007/s00209-014-1398-z},
}
\bib{vb-valued}{article}{
   author={Drummond, Thiago},
   author={Egea, Leandro},
   title={Differential forms with values in VB-groupoids and its Morita
   invariance},
   journal={J. Geom. Phys.},
   volume={135},
   date={2019},
   pages={42--69},
   issn={0393-0440},
   review={\MR{3872622}},
   doi={10.1016/j.geomphys.2018.08.019},
}

\bib{vb-algebroid-morphisms}{article}{
   author={Drummond, T.},
   author={Jotz Lean, M.},
   author={Ortiz, C.},
   title={${\mathcal{VB}}$-algebroid morphisms and representations up to
   homotopy},
   journal={Differential Geom. Appl.},
   volume={40},
   date={2015},
   pages={332--357},
   issn={0926-2245},
   review={\MR{3333112}},
   doi={10.1016/j.difgeo.2015.03.005},
}

\bib{dvb}{article}{
   author={Gracia-Saz, Alfonso},
   author={Mehta, Rajan Amit},
   title={Lie algebroid structures on double vector bundles and
   representation theory of Lie algebroids},
   journal={Adv. Math.},
   volume={223},
   date={2010},
   number={4},
   pages={1236--1275},
   issn={0001-8708},
   review={\MR{2581370}},
   doi={10.1016/j.aim.2009.09.010},
}
\bib{gracia-saz}{article}{
   author={Gracia-Saz, Alfonso},
   author={Mehta, Rajan Amit},
   title={$\mathcal{VB}$-groupoids and representation theory of Lie groupoids},
   journal={J. Symplectic Geom.},
   volume={15},
   date={2017},
   number={3},
   pages={741--783},
   issn={1527-5256},
   review={\MR{3696590}},
   doi={10.4310/JSG.2017.v15.n3.a5},
}

\bib{ym_mecs}{article}{
   author={Grad, Žan},
   year={2026},
   title={Yang--Mills theory for multiplicative Ehresmann connections},
   note={Preprint available at \href{https://arxiv.org/abs/2607.10904}{arXiv:2607.10904}.}, 
}

\bib{ideals}{article}{
   author={Jotz, M.},
   title={Obstructions to representations up to homotopy and ideals},
   journal={Asian J. Math.},
   volume={26},
   date={2022},
   number={2},
   pages={137--166},
   issn={1093-6106},
   review={\MR{4557078}},
   doi={10.4310/ajm.2022.v26.n2.a1},
}

\bib{actions}{article}{
   author={Kosmann-Schwarzbach, Y.},
   author={Mackenzie, K. C. H.},
   title={Differential operators and actions of Lie algebroids},
   conference={
      title={Quantization, Poisson brackets and beyond},
      address={Manchester},
      date={2001},
   },
   book={
      series={Contemp. Math.},
      volume={315},
      publisher={Amer. Math. Soc., Providence, RI},
   },
   isbn={0-8218-3201-8},
   date={2002},
   pages={213--233},
   review={\MR{1958838}},
   doi={10.1090/conm/315/05482},
}


\bib{mecs_for_fibrations}{article}{
   title={Multiplicative Ehresmann connections for Lie groupoid fibrations}, 
   author={Lau, Matthijs},
   author={M\u arcu\c t, Ioan},
   year={2026},
   archivePrefix={arXiv},
   primaryClass={math.DG},
   note={Preprint available at \href{https://arxiv.org/abs/2604.22394}{ 	arXiv:2604.22394}.}, 
}


\bib{gerbes}{article}{
   author={Laurent-Gengoux, Camille},
   author={Sti\'enon, Mathieu},
   author={Xu, Ping},
   title={Non-abelian differentiable gerbes},
   journal={Adv. Math.},
   volume={220},
   date={2009},
   number={5},
   pages={1357--1427},
   issn={0001-8708},
   review={\MR{2493616}},
   doi={10.1016/j.aim.2008.10.018},
}

\bib{ve_mein}{article}{
   author={Li-Bland, David},
   author={Meinrenken, Eckhard},
   title={On the van Est homomorphism for Lie groupoids},
   journal={Enseign. Math.},
   volume={61},
   date={2015},
   number={1-2},
   pages={93--137},
   issn={0013-8584},
   review={\MR{3449284}},
   doi={10.4171/LEM/61-1/2-5},
}

\bib{mec}{article}{
    AUTHOR = {Loja Fernandes, Rui},
    AUTHOR = {M\u{a}rcu\c{t}, Ioan},
     TITLE = {Multiplicative {E}hresmann connections},
   JOURNAL = {Adv. Math.},
    VOLUME = {427},
      YEAR = {2023},
     PAGES = {Paper No. 109124, 84},
      ISSN = {0001-8708,1090-2082},
       DOI = {10.1016/j.aim.2023.109124},
       URL = {https://doi.org/10.1016/j.aim.2023.109124},
}

\bib{mackenzie_duality}{article}{
   author={Mackenzie, Kirill C. H.},
   title={Duality and triple structures},
   conference={
      title={The breadth of symplectic and Poisson geometry},
   },
   book={
      series={Progr. Math.},
      volume={232},
      publisher={Birkh\"auser Boston, Boston, MA},
   },
   isbn={0-8176-3565-3},
   date={2005},
   pages={455--481},
   review={\MR{2103015}},
   doi={10.1007/0-8176-4419-9\_15},
}

\bib{mackenzie_doubles}{article}{
   author={Mackenzie, Kirill C. H.},
   title={Ehresmann doubles and Drinfel'd doubles for Lie algebroids and Lie
   bialgebroids},
   journal={J. Reine Angew. Math.},
   volume={658},
   date={2011},
   pages={193--245},
   issn={0075-4102},
   review={\MR{2831518}},
   doi={10.1515/CRELLE.2011.092},
}

\bib{mackenzie}{book}{
    AUTHOR = {Mackenzie, Kirill C. H.},
     TITLE = {General theory of {L}ie groupoids and {L}ie algebroids},
    SERIES = {London Mathematical Society Lecture Note Series},
    VOLUME = {213},
 PUBLISHER = {Cambridge University Press, Cambridge},
      YEAR = {2005},
     PAGES = {xxxviii+501},
      ISBN = {978-0-521-49928-3; 0-521-49928-3},
       DOI = {10.1017/CBO9781107325883},
       URL = {https://doi.org/10.1017/CBO9781107325883},
}

\bib{bialgebroids}{article}{
   author={Mackenzie, Kirill C. H.},
   author={Xu, Ping},
   title={Lie bialgebroids and Poisson groupoids},
   journal={Duke Math. J.},
   volume={73},
   date={1994},
   number={2},
   pages={415--452},
   issn={0012-7094},
   review={\MR{1262213}},
   doi={10.1215/S0012-7094-94-07318-3},
}

% \bib{moerdijk}{article}{
%     AUTHOR = {Moerdijk, Ieke},
%      TITLE = {Lie groupoids, gerbes, and non-abelian cohomology},
%    JOURNAL = {$K$-Theory},
%   FJOURNAL = {$K$-Theory. An Interdisciplinary Journal for the Development,
%               Application, and Influence of $K$-Theory in the Mathematical
%               Sciences},
%     VOLUME = {28},
%       YEAR = {2003},
%     NUMBER = {3},
%      PAGES = {207--258},
%       ISSN = {0920-3036,1573-0514},
%    MRCLASS = {58H05 (18B40 22A22)},
%   MRNUMBER = {2017529},
% MRREVIEWER = {Michael\ Murray},
%        DOI = {10.1023/A:1026251115381},
%        URL = {https://doi.org/10.1023/A:1026251115381},
% }

\bib{meinrenken_pike}{article}{
   author={Meinrenken, Eckhard},
   author={Pike, Jeffrey},
   title={The Weil algebra of a double Lie algebroid},
   journal={Int. Math. Res. Not. IMRN},
   date={2021},
   number={11},
   pages={8550--8622},
   issn={1073-7928},
   review={\MR{4266146}},
   doi={10.1093/imrn/rnz361},
}

\bib{lie2}{article}{
   author={Moerdijk, Ieke},
   author={Mr\v{c}un, Janez},
   title={On integrability of infinitesimal actions},
   journal={Amer. J. Math.},
   volume={124},
   date={2002},
   number={3},
   pages={567--593},
   issn={0002-9327},
   review={\MR{1902889}},
}

\bib{symplectic_groupoids}{article}{
   author={Weinstein, Alan},
   title={Symplectic groupoids and Poisson manifolds},
   journal={Bull. Amer. Math. Soc. (N.S.)},
   volume={16},
   date={1987},
   number={1},
   pages={101--104},
   issn={0273-0979},
   review={\MR{0866024}},
   doi={10.1090/S0273-0979-1987-15473-5},
}

%\bib{semiring}{book}{
%    AUTHOR = {Golan, Jonathan S.},
%     TITLE = {Semirings and their applications},
% PUBLISHER = {Kluwer Academic Publishers, Dordrecht},
%      YEAR = {1999},
%     PAGES = {xii+381},
%      ISBN = {0-7923-5786-8},
%   MRCLASS = {16Y60 (68Q01)},
%  MRNUMBER = {1746739},
%MRREVIEWER = {Udo Hebisch},
%       DOI = {10.1007/978-94-015-9333-5},
%       URL = {https://doi.org/10.1007/978-94-015-9333-5},
%}



%%% \bib{HA}{book}{
%%%   author={Quillen, Daniel G.},
%%%   title={Homotopical Algebra},
%%%   series={Lecture Notes in Mathematics},
%%%   volume={43},
%%%   publisher={Springer-Verlag},
%%%   address={Berlin-New York},
%%%   date={1967}
%%% }



  \end{biblist}
\end{bibdiv}
\end{document}